\documentclass[11pt,a4paper]{article}
\usepackage{titlesec}
\usepackage{a4wide}
\usepackage{graphicx}
\usepackage{float}
\usepackage{amssymb}
\usepackage{amsmath}
\usepackage{amsthm}
\usepackage{appendix}
\usepackage{color}
\usepackage{array}
\usepackage{eucal}
\usepackage{mathrsfs}
\usepackage{esint}
\usepackage{lmodern}
\usepackage{tikz}
\usepackage{bbm}
\usepackage{hyperref}
\usepackage{geometry}
\usepackage{changepage}
\changepage{0pt}{}{}{}{}{0pt}{}{0pt}{10pt}
\usepackage[numbers]{natbib}
\hypersetup{
	pdfpagemode=UseThumbs,
	pdftoolbar=true,        % show Acrobat's toolbar?
	pdfmenubar=true,        % show Acrobats menu?
	pdffitwindow=false,     % window fit to page when opened
	pdfstartview={Fit},    % fits the width of the page to the window
	pdftitle={Maxwell's equations with hypersingularities \\[6pt]
	at a negative index material conical tip},    % title
	pdfauthor={A.-S. Bonnet-Ben Dhia, L. Chesnel, M. Rihani},     % author
	pdfsubject={},  % subject of the document
	pdfcreator={A.-S. Bonnet-Ben Dhia, L. Chesnel, M. Rihani},   % creator of the document
	pdfproducer={A.-S. Bonnet-Ben Dhia, L. Chesnel, M. Rihani}, % producer of the document
	pdfkeywords={}, % list of keywords
	pdfnewwindow=true,      % links in new window
colorlinks=true,       % false: boxed links; true: colored links
linkcolor=magenta,          % color of internal links
citecolor=red,        % color of links to bibliography
filecolor=cyan,      % color of file links
urlcolor=blue           % color of external links
}

\def\u{\boldsymbol{u}}
\def\v{\boldsymbol{v}}
\def\e{\boldsymbol{e}}
\def\vb{\overline{\boldsymbol{v}}}
\def\w{\boldsymbol{w}}
\def\E{\boldsymbol{E}}

\def\H{\boldsymbol{H}}
\def\Hspace{\boldsymbol{\mrm{H}}}
\def\Lspace{\boldsymbol{\mrm{L}}}

\def\Houtspace{\pmb{\mathscr{H}}_N^{\hspace{0.2mm}\out}(\curl)}
\def\Xoutspace{\pmb{\mathscr{X}}_N^{\hspace{0.2mm}\out}}
\def\HToutspace{\pmb{\mathscr{H}}^{\hspace{0.2mm}\out}(\curl)}
\def\XToutspace{\pmb{\mathscr{X}}_T^{\hspace{0.2mm}\out}}

\def\J{\boldsymbol{J}}

\def\Cgras{\pmb{\mathscr{C}}}

\newcommand{\mA}{\mrm{A}}

\def\C{\mathbb C}		% complexes
\newcommand{\dsp}{\displaystyle}

\newcommand{\eps}{\varepsilon}
\newcommand{\om}{\omega}
\newcommand{\Om}{\Omega}
\newcommand{\mrm}[1]{\mathrm{#1}}

\newcommand{\Cplx}{\mathbb{C}}
\newcommand{\N}{\mathbb{N}}
\newcommand{\R}{\mathbb{R}}

\renewcommand{\div}{\mrm{div}}
\newcommand{\mL}{\mrm{L}}
\newcommand{\mH}{\mrm{H}}
\newcommand{\mV}{\mrm{V}}
\newcommand{\mmV}{\boldsymbol{\mrm{V}}}
\newcommand{\mX}{\boldsymbol{\mrm{X}}}
\newcommand{\mY}{\boldsymbol{\mrm{Y}}}
\newcommand{\mZ}{\boldsymbol{\mrm{Z}}}

\newcommand{\out}{\mrm{out}}

\newcommand{\curl}{\boldsymbol{\mrm{curl}}\,}
\newcommand{\psib}{\boldsymbol{\psi}}

\usepackage{bbm}

\newtheorem{theorem}{Theorem}[section]
\newtheorem{lemma}[theorem]{Lemma}
\newtheorem{remark}[theorem]{Remark}

\newtheorem{proposition}[theorem]{Proposition}
\newtheorem{Assumption}{Assumption} 
\usepackage{dsfont}
\everymath{\displaystyle}
\begin{document}
\begin{center}
{\sc \bf\fontsize{20}{20}\selectfont  	Maxwell's equations with hypersingularities \\[6pt]
	at a negative index material conical tip}
\end{center}
\begin{center}
	\textsc{Anne-Sophie Bonnet-Ben Dhia}$^1$, \textsc{Lucas Chesnel}$^2$, \textsc{Mahran Rihani}$^{3}$\\[16pt]
	\begin{minipage}{0.9\textwidth}
		{\small
$^1$  POEMS, CNRS, Inria, ENSTA Paris, Institut Polytechnique de Paris, 91120 Palaiseau, France;\\
$^2$ Inria, ENSTA Paris, Institut Polytechnique de Paris, 91120 Palaiseau, France;\\
$^3$ CMAP, \'Ecole Polytechnique, Institut Polytechnique de Paris, 91128 Palaiseau, France.\\[10pt]
			E-mails:
			\texttt{anne-sophie.bonnet-bendhia@ensta-paris.fr}, \texttt{lucas.chesnel@inria.fr},\\ \texttt{mahran.rihani@polytechnique.edu}.\\[-14pt]
			\begin{center}
				(\today)
			\end{center}
		}
	\end{minipage}
\end{center}
\textbf{Abstract:} We study a transmission problem for the time harmonic Maxwell's equations between a classical positive material and a so-called negative index material in which both the permittivity $\eps$ and the permeability $\mu$ take negative values. Additionally, we assume that the interface between the two domains is smooth everywhere except at a point where it coincides locally with a conical tip. In this context, it is known that for certain critical values of the contrasts in $\eps$ and in $\mu$, the corresponding scalar operators are not of Fredholm type in the usual $\mH^1$ spaces. In this work, we show that in these situations, the   Maxwell's equations are not well-posed in the classical  $\Lspace^2$  framework due to existence of hypersingular fields which are of infinite energy at the tip.
By combining the $\mrm{T}$-coercivity approach and the Kondratiev theory, we explain how to construct new functional frameworks to recover well-posedness of the Maxwell's problem. We also explain how to select the setting which is consistent with the limiting absorption principle.
From a technical point of view, the fields as well as their curls decompose as the sum of an explicit singular part, related to the black hole singularities of the scalar operators, and a smooth part belonging to some weighted spaces. The analysis we propose rely in particular on the proof of new key results of scalar and vector potential representations of singular fields.\\[6pt]

\tableofcontents

\section{Introduction}
Recent progress in the conception of artificial microstructured materials with unusual effective properties, related to exciting physical experiments, have led to the development of new fields of research in applied mathematics. In particular, there has been a need to reconsider equations of acoustics, electromagnetism and elastodynamics with material laws which do not fit into the classical theories. In electromagnetism for example, it is generally assumed that the dielectric permittivity $\varepsilon$ and the magnetic permeability $\mu$ are real positive quantities or positive-definite tensors for anisotropic materials. However, nowadays it seems possible to engine metamaterials which are well-approximated in certain frequency ranges by effective $\varepsilon$, $\mu$  which are real negative functions or negative-definite tensors. In practice, to obtain interesting devices for applications, one needs to combine these non standard materials with classical ones to create interfaces where unusual phenomena occur. For instance, as described by Veselago in his pioneering article \cite{Vese68}, for a planar interface, a plane wave impinging from one side is refracted to the other side in a direction which is opposite to the standard one. This is the so-called negative refraction which may be useful to create perfect lenses \cite{Pend00} and which has given the naming ``negative index materials''. Mathematically, in time-harmonic regime, this leads to study non classical transmission problems which have been the motivation for a series of articles that we present now.\\
\newline
In 2D, the Maxwell's equations can be reduced to a scalar Helmholtz-type problem for which the theory now is quite complete \cite{BCCC16,dhia2012t,bonnet2013radiation,costabel1985direct,nguyen2016limiting}. Let us give a brief summary of the main results. In the analysis, the smoothness of the interface $\Sigma$ between the so-called positive and negative materials, as well as the contrast, defined as the ratio of the values of the negative coefficient over the positive coefficient at $\Sigma$, play a key role. When $\Sigma$ is of class $\mathscr{C}^1$, there is no strong consequence of the change of sign of the coefficients, except in the very critical case where the contrast is equal to $-1$. The change of sign of the coefficients has a more striking effect when the interface has geometric singularities. Thus, when $\Sigma$ has corners, the operator associated with the problem in the classical Sobolev space $\mH^1$ is not Fredholm for a whole interval of contrasts (the so-called critical interval) which contains $-1$. This is due to the existence of non classical singularities which oscillates more and more when approaching the corner (see Figure \ref{PicturePropa} left). 

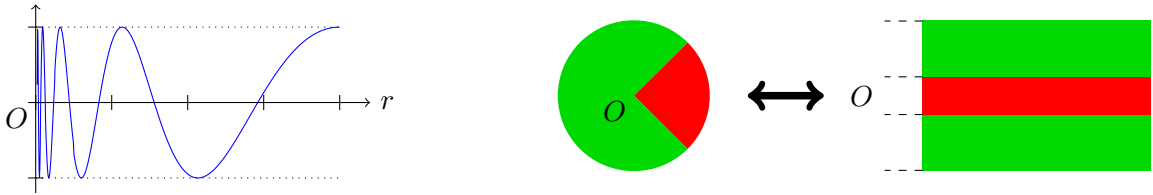
\begin{figure}[h!]
\centering
\begin{tikzpicture}
\draw[->] (0,0) -- (4.4,0) node [right] {$r$};
\draw (-0.25,-0.2) node {$O$};
\draw[dotted] (0,1) -- (4,1);
\draw[dotted] (0,-1) -- (4,-1);
\draw[samples=800,domain=0.02:4,blue] plot(\x,{cos((deg(ln(\x*0.25)*5)))});
\draw[->] (0,-1.2) -- (0,1.3);
\draw (-0.3,1);
\draw (0.3,-1);
\foreach \x in {0,...,4} \draw (\x,0.1) -- (\x,-0.1);
\foreach \y in {-1,...,1} \draw (-0.1,\y) -- (0.1,\y);
\end{tikzpicture}\qquad\qquad\quad
\raisebox{2mm}{\begin{tikzpicture}
\draw[fill=red,draw=none] (0,0) -- (-45:1) arc (-45:45:1)--cycle; 
\draw[fill=green!85!black,draw=none] (0,0) -- (45:1) arc (45:315:1)--cycle; 
\draw (-0.25,-0.2) node {$O$};
\draw[<->,line width=1mm] (1.5,0)--(2.5,0);
\begin{scope}[xshift=0.8cm]
\draw[fill=green!85!black,draw=none] (3,-1) rectangle (6,1);
\draw[fill=red,draw=none] (3,-0.25) rectangle (6,0.25);
\draw[dashed] (3,-0.25) -- (2.5,-0.25);
\draw[dashed] (3,0.25) -- (2.5,0.25);
\draw[dashed] (3,-0.99) -- (2.5,-0.99);
\draw[dashed] (3,0.99) -- (2.5,0.99);
\draw (2.2,0) node {$O$};
\end{scope}
\end{tikzpicture}}
\caption{Left: real part of the radial behaviour of the hyper-singularity in 2D ($\Re e\,r^{i\eta}=\cos(\eta\ln r)$). Right: schematic correspondence between the corner and waveguide problems.
\label{PicturePropa}}
\end{figure}

To understand the phenomenon, it is convenient to rely on the analogy with a waveguide problem which appears naturally when considering the change of variables $(r,\theta) \mapsto (\ln r,\theta)$, where $(r,\theta)$ are the polar coordinates centered at the corner. Observe that this map transforms the vicinity of the corner into a semi-infinite strip (the waveguide) with the corner rejected to infinity (Figure \ref{PicturePropa} right). Moreover, singularities at the corner correspond to modes in the strip. Importantly, one finds that propagating modes exist if and only if the contrast belongs to  the critical interval. Back in polar coordinates, propagating modes become what we call hyper-singularities or propagating singularities. Their radial behaviour is of the form $r^{\pm i\eta}=e^{\pm i\eta\ln r}$ for some real $\eta$. Like for scattering problems in waveguides, it is necessary  to impose a radiation condition at infinity in the strip, or equivalently at the corner, to obtain a well-posed problem. This radiation condition allows one to select the physically relevant solution in the sense that it is the limit of the solutions obtained by adding some small dissipation to the medium, which corresponds to add a small uniformly positive imaginary part to the sign-changing coefficient. This is the so-called limiting absorption principle. Note that due to the sign-changing coefficient, the selection of the outgoing behaviour, $r^{i\eta}$ or $r^{- i\eta}$, cannot be deduced from the sign of $\eta$. Instead, one uses energy considerations: the outgoing propagating singularity is the one which carries energy to the corner and for this reason, it is sometimes called black-hole wave. Justifying rigorously all this formal analysis and providing a framework taking into account the black-hole wave to recover a well-posed problem is a task in itself, which has been realized in \cite{bonnet2013radiation} thanks to the theory of detached asymptotics in Kondratiev spaces (see the reviews \cite{Naza99a,Naza99b}). \\
\newline
While the 2D scalar case has been intensively studied, the 3D Maxwell's equations with sign-changing coefficients have received much less attention. One of the reasons, in addition to intrinsic difficulties of the analysis of Maxwell's equations, is that a good knowledge of the corresponding 3D scalar problems (with sign-changing coefficients) is necessary to address the problem. And this is much more complicated than in 2D as soon as the interface is not smooth.  This explains why, up to recently, the only references were \cite{dhia2014t} and \cite{nguyen2020limiting}. In \cite{nguyen2020limiting}, the authors  consider in detail the case of a smooth interface between a material with negative $\varepsilon$, $\mu$ and another one with positive $\varepsilon$, $\mu$. In \cite{dhia2014t}, it is proved that the Maxwell's problem is well-posed in the classical framework, that is with electric and magnetic fields $\E,\H$ in $\Lspace^2$, as soon as the contrasts in $\varepsilon$ and in $\mu$ are not critical. Note that if the interface is not smooth, it is known that the latter condition is satisfied for contrasts outside of a given interval which contains $-1$. In the existing literature, let us also mention \cite{BCRR21} where the authors derive a homogenized model for a composite medium with periodically distributed small inclusions of negative material.\\
\newline
The goal of the present article is to complement the previous studies and more precisely to consider the time-harmonic 3D Maxwell's equations in configurations involving non smooth interfaces with critical contrasts so that a solution cannot be found in general with $\E,\H$ in $\Lspace^2$. 
We focus our attention on situations where the interface is smooth except at one point where it coincides with a conical tip, which is the 3D configuration that most closely looks like the case of the 2D corner. We also assume that the base of the conical tip is smooth, which excludes the possibility of having edges.\\
\newline
The present study constitutes a follow-up to the article \cite{dhia2022maxwell}. In that work, we addressed a part of the difficulties by considering the case where $\varepsilon$ is critical but $\mu$ is not. Additionally, we assumed that the scalar problem for $\eps$, which involves the differential operator $\div(\varepsilon \nabla \cdot)$, admits exactly one propagating singularity, like for 2D corners. Let us mention here that in 3D, the propagating singularity behaves like $r^{-1/2\pm i\eta}$ with $\eta\in\R$. We proved that, due to the black-hole effect at the tip, the electric field $\E$ must not be searched in $\Lspace^2$. Instead, it must be a linear combination of a $\Lspace^2$ function and the gradient of the propagating singularity. To prescribe this behaviour, we adapted the theory of detached asymptotics in Kondratiev spaces to Maxwell's equations, which seems to be new. Note that even for the application of the more classical Kondratiev theory to the standard Maxwell's equations, there are only few works (see \cite{costabel2002weighted} and \cite{PlPo2014,PlPo18}). Additionally, it must be added that in the framework proposed in \cite{dhia2022maxwell}, $\curl\E$ belongs to $\Lspace^2$ (because $\mu$ is not critical) while, as expected, $\H\in\Lspace^2$, $\curl\H\notin\Lspace^2$.\\ 
\newline
In this article, we extend the analysis of \cite{dhia2022maxwell} in several directions. First, we suppose that both $\varepsilon$ and $\mu$ are critical, so that in general $\E$, $\curl\E$, $\H$, $\curl\H$ are not in $\Lspace^2$. In addition, we take into account the fact that they may exist several black-hole singularities for the two scalar problems involving the differential operators $\div(\varepsilon \nabla\cdot)$ and $\div(\mu \nabla\cdot)$ respectively. Note that roughly speaking, the smaller the cone aperture, the higher the number of black-hole singularities. We will see in \S\ref{ParaHyperSingu} that the black-hole singularities are defined from eigenvalues and eigenfunctions of a non-selfadjoint problem. For this reason, for certain exceptional values of the  contrasts in $\varepsilon$ or $\mu$, Jordan blocks occur. With the Mandelstam principle, we will provide mathematical frameworks where well-posedness holds even in this case that was avoided in \cite{dhia2022maxwell}. Finally, the geometric setting that we consider here is a bit more general than the one of the previous article because we allow the conical tip to touch the boundary of the domain (see Figure \ref{FigGeo} center). This might seem a minor technical point but it forces us to write more systematic proofs for the vector potential representations of singular fields (see the appendix).\\
\newline
Let us mention that time-harmonic Maxwell's equations with sign-changing coefficients also appear in the study of magnetized plasma, however with an important difference compare to what is done here. Indeed, cold plasma models lead to consider smooth $\varepsilon$ which vanish when changing sign (which never happens for $\varepsilon$ and $\mu$ below). This is responsible for the phenomenon of hybrid resonance near the interface where $\varepsilon$ is equal to zero (see \cite{DeIW14,NiCD19}). On the other hand, note that there are several connections between the scalar operators with sign-changing coefficients and the spectral theory of the Neumann-Poincar\'e operator. Indeed, the latter acts on functions defined on a surface which plays a role similar to the interface in our work. It is known that when the surface has corners in 2D \cite{BoZh19} or tips in 3D \cite{HePe18}, the spectrum of the Neumann-Poincar\'e operator contains an interval of essential spectrum with possibly embedded eigenvalues \cite{LiSh19,BoHM21,LiPS20}. This interval of essential spectrum is in exact correspondence with our critical interval and is generated by the same black-hole singularities. Similar results have been proved for the strongly singular volume integral operator that describes the scattering of time-harmonic electromagnetic waves (see \cite{CoDS12} for 3D smooth interfaces and \cite{CoDS15} for 2D Lipschitz interfaces).\\
\newline
The paper is organized as follows. We start by presenting the problem and the notation. In Section \ref{SectionScalar}, we recall some properties concerning the scalar operators involved in the analysis of Maxwell's equations. In particular, we present how one should incorporate in the functional frameworks some of the hypersingularities of the corresponding problems to get well-posedness. Section \ref{SectionPbElectrique} constitutes the heart of the article. There we propose a new functional framework to study the electric problem. The electric field decomposes on the hypersingularities associated with the scalar problem for $\eps$ while its curl decomposes on the hypersingularities associated with the scalar problem for $\mu$. Additionally, we show that the new framework satisfies the limiting absorption principle as soon as the two scalar problems do. In Section \ref{SectionH}, we summarize the results for the magnetic problem.
In Section \ref{SectionNecessity}, we prove that the Maxwell's equations are not well-posed in the classical $\Lspace^2$ setting. Then we give a few words of conclusion before proving crucial results of representation by potential as well as results of compact embeddings in weighted Sobolev spaces in Section \ref{SectionAppendix}.
The main results of this work are Theorem \ref{SectionEResFinal} for the electric field and Theorem \ref{SectionHResFinal} for the magnetic field.

\section{Setting of the problem }
\label{Setting}

\subsection{General notation}
Let $\Om$ be an open, connected and bounded subset of $\R^3$ with a Lipschitz-continuous boundary $\partial\Om$. To simplify, we assume that $\Om$ is simply connected and that $\partial\Om$ is connected. When this assumption is not satisfied, the results below can be adapted by working as in \cite[\S8.2]{dhia2014t}. For some frequency $\om\ne0$, ($\om\in\R$), the time-harmonic Maxwell's equations are given by
\begin{equation}\label{EqMaxwellInitiale}
\curl\E-i\om\mu\H = 0\qquad\mbox{ and }\qquad\curl\H+i\om\eps\E = \J\,\mbox{ in }\Om.
\end{equation}
Above $\E$ and $\H$ are respectively the electric and magnetic components of the electromagnetic field while $\J$ stands for some current density injected in  $\Om$. In this work, we suppose that  $\Om$ is surrounded  by  a  perfect conductor. This leads us  to complete the previous system of equations with  the boundary conditions
	\begin{equation}\label{CLMaxwell}
		\E\times\nu=0\qquad\mbox{ and }\qquad\mu\H\cdot\nu=0\,\mbox{ on }\partial\Om.
	\end{equation}
Here $\nu$ denotes the unit outward normal vector to $\partial\Om$. The dielectric permittivity $\eps$ and the magnetic permeability $\mu$ in (\ref{EqMaxwellInitiale}) are assumed to be real valued functions such that $\eps$, $\mu\in\mL^{\infty}(\Om)$ and $\eps^{-1}$, $\mu^{-1}\in\mL^{\infty}(\Om)$. However their signs change in $\Om$ as described below. As it is classical in the study of Maxwell's equations, we will work with the spaces
\[
\begin{array}{rcl}
\Lspace^2(\Omega)&:=&(\mL^2(\Omega))^3\\
\Hspace^1(\Omega)&:=&(\mH^1(\Omega))^3\\
\mH^1_{0}(\Omega)&:=&\{\varphi\in\mH^1(\Om)\,|\,\varphi=0\mbox{ on }\partial\Om\}\\
\mH^1_{\#}(\Omega)&:=&\{\varphi\in\mH^1(\Om)\,|\,\int_{\Om}\varphi\,dx=0\}\\
\dsp \Hspace(\curl) &:=& \dsp \{ \boldsymbol{H}\in \Lspace^2(\Omega) \,|\, \curl \boldsymbol{H}\in\Lspace^2(\Omega)\}\\[2pt]
\dsp \Hspace_N(\curl) &:=& \dsp \{ \boldsymbol{E}\in \Hspace(\curl) \,|\, \boldsymbol{E}\times\nu=0 \mbox { on } \partial\Omega\}.
\end{array}
\]
The density current $\J$ in (\ref{EqMaxwellInitiale}) belongs to a subspace of $\Lspace^2(\Omega)$ that will be specified later and satisfies $\div\,\J=0$ in $\Om\backslash\{O\}$. This leads us to introduce, for $\xi\in \mrm{L}^{\infty}(\Om)$, the spaces
\begin{equation}\label{DefSpaceMaxwell}
\begin{array}{rcl}
\mX_N(\xi) & := &\left\{\boldsymbol{E}\in \Hspace_N(\curl)\,|\,\div(\xi\boldsymbol{E})=0\right\}\\[3pt]
\mX_T(\xi) & :=& \left\{\boldsymbol{H}\in \Hspace(\curl)\,|\,\div(\xi\boldsymbol{H})=0,\,\xi\boldsymbol{H}\cdot\nu=0 \mbox{ on }\partial\Om \right\}.
\end{array}
\end{equation}
We denote indistinctly by $(\cdot,\cdot)_{\Om}$ the classical inner products of $\mL^2(\Om)$ and $\Lspace^2(\Om)$. Moreover, $\|\cdot\|_{\Om}$ stands for the corresponding norms. We endow the spaces $\Hspace(\curl)$, $\Hspace_N(\curl)$, $\mX_N(\xi)$, $\mX_T(\xi)$ with the norm 
\[
\|\cdot\|_{\Hspace(\curl)}:=(\|\cdot\|^2_{\Om}+\|\curl\cdot\|^2_{\Om})^{1/2}.
\]
Let us recall a well-known property for the particular spaces $\mX_N(1)$ and $\mX_T(1)$  (cf. \cite{Webe80,AmrBerDau98}).
\begin{proposition}\label{PropoEmbeddingCla}
The embeddings of $\mX_N(1)$ in $\Lspace^2(\Om)$ and of $\mX_T(1)$ in $\Lspace^2(\Om)$ are compact. Moreover there is a constant $C>0$ such that 
\[
\|\u\|_{\Om}\le C\,\|\curl\u\|_{\Om},\qquad \forall\u\in\mX_N(1)\cup\mX_T(1).
\]
Therefore, in $\mX_N(1)$ and in $\mX_T(1)$, $\|\curl\cdot\|_{\Om}$ is a norm which is equivalent to $\|\cdot\|_{\Hspace(\curl)}$.
\end{proposition}

\subsection{Description of the conical tip}\label{paragraphGeom}

In this paragraph, we detail the assumptions made on the geometry. We study a situation where $\Om$ is partitioned into two non empty subdomains $\Om_+$ and $\Om_-$ corresponding respectively to the positive and negative materials. We assume that $\overline{\Om}=\overline{\Om_+}\cup\overline{\Om_-}$, $\Om_+\cap\Om_-=\emptyset$ and that both $\partial\Om_+$, $\partial\Om_-$ are Lipschitz-continuous. The functions $\eps$, $\mu$ are such that
\[
\eps=\begin{array}{|ll}
\eps_+>0 & \mbox{ in }\Om_+\\[2pt]
\eps_-<0 & \mbox{ in }\Om_-,
\end{array}\qquad\qquad \mu=\begin{array}{|ll}
\mu_+>0 & \mbox{ in }\Om_+\\[2pt]
\mu_-<0 & \mbox{ in }\Om_-,
\end{array}
\]
where $\eps_\pm$, $\mu_\pm$ are some constants. In the study of Problem 
(\ref{EqMaxwellInitiale})-(\ref{CLMaxwell}), the contrasts $\kappa_\eps$, $\kappa_\mu$ defined by
\[
\kappa_\eps:=\eps_-/\eps_+,\qquad\qquad\kappa_\mu:=\mu_-/\mu_+
\]
play a key role. Let us come now to the description of the conical tip. We assume that $\Om_-$ is of class $\mathscr{C}^2$ except at some point $O$ where $\Om_-$ coincides locally with a  cone $\mathcal{K}_-$. More precisely, we choose the system of coordinates such that $O=(0,0,0)$ and we assume that there is $\rho>0$ as well as some smooth domain $\varpi_-$ (of class $\mathscr{C}^2$) of the unit sphere $\mathbb{S}^2:=\{x\in\R^3\,|\,|x|=1\}$ such that 
\begin{equation}\label{defRho}
\Om_-\cap B(O,\rho)=\mathcal{K}_-\cap B(O,\rho)\qquad\mbox{ with }\mathcal{K}_-:=\{r\boldsymbol{\theta}\,|\,r>0,\,\boldsymbol{\theta}\in\varpi_-\}.
\end{equation}
Here $B(O,\rho)$ denotes the open  ball centered at $O$ and of radius $\rho$. Additionally, we assume that the setting satisfies one of the two following assumptions:
\begin{equation}\label{ChoixCase}
\begin{array}{l|}
\begin{minipage}{0.875\textwidth}
\mbox{\textbf{Case 1:} There holds $O\in\Om$ and so we can choose $\rho$ such that $B(O,\rho)\subset\Om$} (internal conical tip, see Figure \ref{FigGeo} left). In that situation, we set $\varpi:=\mathbb{S}^2$. \\
\newline
\textbf{Case 2:} There holds $O\in\partial\Om$ (conical tip on the boundary). In that situation, to simplify, we assume that $\Om$ also coincides with a conical tip in a neighbourhood of $O$:
\[
\Om\cap B(O,\rho)=\mathcal{K}\cap B(O,\rho)\qquad\mbox{ with }\mathcal{K}:=\{r\boldsymbol{\theta}\,|\,r>0,\,\boldsymbol{\theta}\in\varpi\}
\]
where $\varpi$ is a smooth domain of $\mathbb{S}^2$ such that $\overline{\varpi_-}\subset\varpi$ (see Figure \ref{FigGeo} center).
\end{minipage}
\end{array}
\end{equation}

\begin{figure}[h!]
\includegraphics[width=0.2\textwidth,angle =-90]{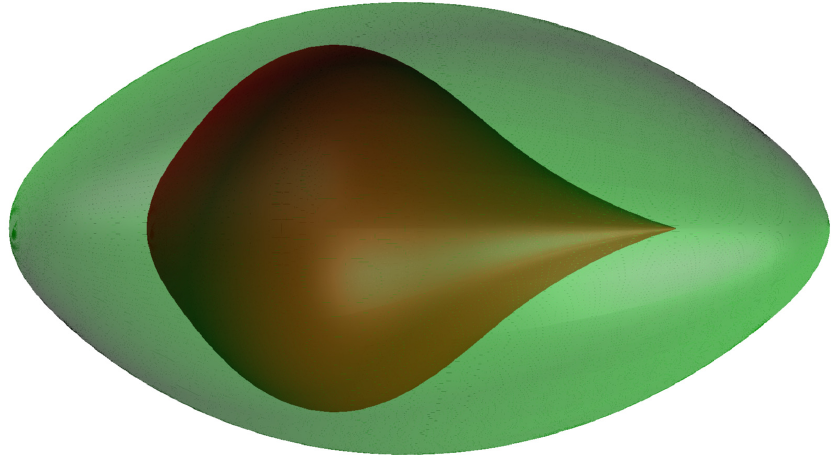}
\includegraphics[width=0.2\textwidth,angle =-90]{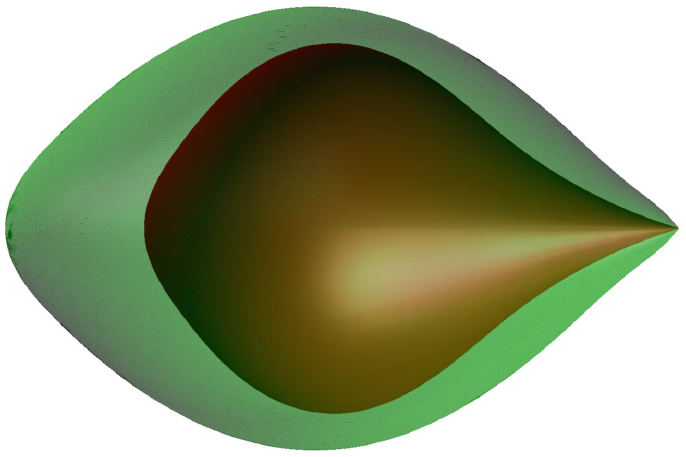}
\includegraphics[width=0.2\textwidth,angle =-90]{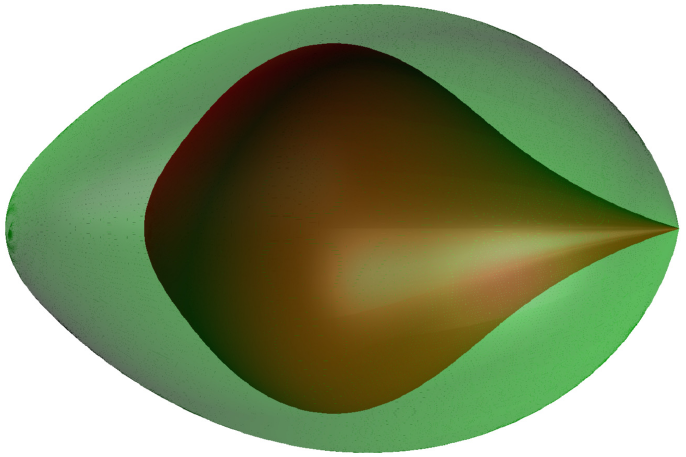}
\caption{Internal conical tip (left), conical tips on the boundary (center and right).}
\label{FigGeo}
\end{figure}

\begin{remark}
Admittedly, the case 2 above is rather academic. However it will force us to write more systematic proofs. Note also that by rectifying the boundary with a diffeomorphism, adapting for example the ideas of \cite[\S1.3.7, vol. 1]{MaNP00}, we could consider the case $O\in\partial\Om$ where $\partial\Om$ is smooth in a neighbourhood of $O$ (see Figure \ref{FigGeo} right).
\end{remark}

Now that the geometry is fixed, it remains to clarify the assumptions made on the contrasts $\kappa_\eps$, $\kappa_\mu$. The study of the Maxwell's system is directly related to the properties of two scalar operators that we present now. In this document, for a Banach space $\mrm{X}$, $\mrm{X}^{\ast}$ stands for the topological antidual space of $\mrm{X}$ (the set of continuous anti-linear forms on $\mrm{X}$). We denote by $\langle\cdot,\cdot\rangle$ the corresponding duality pairing. Define $\mA_\eps:\mH^1_0(\Om)\to(\mH^1_0(\Om))^\ast$  such that 
\begin{equation}\label{Def_Aeps}
	\langle \mA_{\eps}\varphi,\varphi'\rangle=\int_{\Om}\eps\nabla\varphi\cdot\nabla\overline {\varphi'}\,dx,\qquad \forall \varphi,\varphi'\in\mH^1_0(\Om)
\end{equation}
and $\mA_\mu:\mH^1_\#(\Om)\to (\mH^1_\#(\Om))^\ast$ such that
\begin{equation}\label{Def_Amu}
	\langle \mA_{\mu}\varphi,\varphi'\rangle=\int_{\Om}\mu\nabla\varphi\cdot\nabla\overline{\varphi'}\,dx,\qquad \forall \varphi,\varphi'\in\mH^1_{\#}(\Om).
\end{equation}
It is proved in \cite{dhia2014t} that when the functions $\eps$, $\mu$ are such that $\mA_\eps$, $\mA_\mu$ are of Fredholm type (without assumption of sign for $\eps$, $\mu$), then the Maxwell's system is also well-posed in the Fredholm sense in the classical $\Lspace^2$ framework (see Section \ref{SectionNecessity} below for more details). In this work, our goal is to study a situation where precisely $\mA_\eps$, $\mA_\mu$ are not of Fredholm type. When the sign of $\eps$ and $\mu$ is not constant, determining if this holds or not is not straightforward. For a general Lipschitz-continuous  interface $\partial\Om_+\cap\partial\Om_-$ between the materials, it is known that $\mA_\eps$ (resp. $\mA_\mu$) is of Fredholm type if and only if $\kappa_\eps\in\R_-^\ast\backslash I_\eps$  (resp. $\kappa_\mu\in\R_-^\ast\backslash I_\mu$), where $I_\eps$ (resp. $I_\mu$)  is a bounded closed subset of $\R_-^\ast:=(-\infty;0)$ called the critical interval.\\
\newline
As mentioned in the introduction, when the interface is smooth, one can show that $I_\eps=I_\mu=\{-1\}$ \cite{costabel1985direct,nguyen2016limiting}. When the interface is not smooth, the situation is different. This has been investigated in details for the case of 2D interfaces with corners in \cite{dauge2011non,dhia2012t,bonnet2013radiation} and for the case of 3D interfaces with conical tips in \cite[chapter 2]{rihani2022maxwell}.  In the latter works, it has been shown that in these configurations $I_\eps$, $I_\mu$ are intervals with a non empty interior. In 2D for corners, the expressions of $I_\eps$, $I_\mu$ can be derived explicitly. One finds that $I_\eps$, $I_\mu$ get even larger as the corner becomes sharp. For 3D conical tips, there is a larger variety of situations and in general, we cannot get simple expressions for $I_\eps$, $I_\mu$. However we can still give a characterization of $I_\eps$, $I_\mu$ and that will be done in the next section where we recall how to study the scalar problems. Let us mention that for wedges, which will not be considered here, one may consult \cite{dhia2012t} and the more recent work \cite{Kale2022}.

\section{Study of the scalar problems}\label{SectionScalar}

\subsection{Kondratiev spaces}
We start by introducing weighted Sobolev spaces adapted to the kind of singularities we want to handle. For $\beta \in \mathbb{R}$ and $m \in \N$, we define the Kondratiev space (see \cite{kond67,MaNP00,KoMR97})  $\mV_\beta^m(\Om)$ as the closure of $\mathscr{C}^{\infty}_0(\overline{\Om}\backslash \{O\})$
	for the norm
	\[
	\|\varphi\|_{\mV_\beta^m(\Om)}=\left(\sum_{|\alpha| \leq m} \|r^{|\alpha|-m+\beta} \partial_x^{\alpha}\varphi\|_{\Om}^2\right)^{1/2}.
	\]
Here $r=|x|$ and $\mathscr{C}^{\infty}_0(\overline{\Om}\backslash \{O\})$ denotes the space of infinitely differentiable functions which are supported in $\overline{\Om}\backslash \{O\}$. For $m\in\mathbb{N}^\ast:=\{1,2,\dots\}$ and $\beta\in\mathbb{R}$, we have the inclusion
	\begin{equation}
		\mV_\beta^m(\Om)\subset\mV_{\beta-1}^{m-1}(\Om).
		\label{InclusionKondratiev}
	\end{equation}
Clearly there holds $\mV^m_\beta(\Om)\subset \mV^m_\gamma(\Om)$ when $\beta\le\gamma$. Additionally, the elements of $\mV^m_\beta(\Om)$ belong to $\mH^m$ of any region excluding a neighbourhood of $O$. Furthermore, we have the following
compactness result (see e.g. \cite[Lemma 6.2.1]{KoMR97}).
\begin{lemma}	\label{CompacitePoids}
For $m\in\N^\ast$ and $\beta<\gamma$, the embedding
$\mV^m_\beta(\Om)\subset \mV^{m-1}_{\gamma-1}(\Om)$
is compact.
\end{lemma}
It is obvious that  $\mV_0^1(\Om)\subset\mH^1(\Om)$. Moreover,   since $\Om$ is bounded, a classical Hardy inequality (see e.g. \cite[Theorem~7.1.1]{KoMR97}) guarantees that $\mH^1(\Om)=\mV_0^1(\Om)$ (note that this is not true in 2D).\\
\newline
To study problems with Dirichlet boundary conditions, introduce the space $\mathring{\mV}_\beta^1(\Om)$ defined as the closure of  $\mathscr{C}^{\infty}_0(\Om\backslash \{O\})$ for the norm
	$\|\cdot\|_{\mV_\beta^1(\Om)}$ (note that in case 2 where $O\in\partial\Om$, see (\ref{ChoixCase}), one gets $\mathscr{C}^{\infty}_0(\Om\backslash \{O\})=\mathscr{C}^{\infty}_0(\Om)$). We have the characterization
\[
\mathring{\mV}_\beta^1(\Om)=\{\varphi\in\mV_\beta^1(\Om)\,|\,\varphi=0\mbox{ on }\partial\Om\}.
\]
There holds $\mH_0^1(\Om)=\mathring{\mV}_0^1(\Om)$ and for  $\beta>0$,  one has   the inclusions 
\[
\mathring{\mV}_{-\beta}^1(\Om)\subset \mH^1_0(\Om)\subset \mathring{\mV}_\beta^1(\Om)\qquad\mbox{ and so }\qquad(\mathring{\mV}^1_{\beta}(\Om))^{\ast}\subset (\mH^1_0(\Om))^{\ast}\subset (\mathring{\mV}^1_{-\beta}(\Om))^{\ast}.
	\]
On the other hand, for $0\le \beta\le1$, from \eqref{InclusionKondratiev}, one gets $\mH^1_0(\Om)\subset\mV^0_{-\beta}(\Om)\subset\mL^2(\Om)$ which implies 
\[
\mL^2(\Om)\subset\mV^0_\beta(\Om)\subset (\mH^1_0(\Om))^\ast.
\]
Above, we have also used that $(\mV^0_\beta(\Om))^\ast=\mV^0_{-\beta}(\Om)$.\\
\newline
To study problems with Neumann boundary conditions, we will work in spaces of mean free functions. To define them, first observe that for all $\beta >-3/2$, we have
\[
\int_{\Om}r^{2\beta}\,dx<+\infty.
\]
As a consequence, for $u\in\mV^1_{\beta}(\Om)$ with $\beta<5/2$, we can write
\begin{equation}\label{EstimZeroMean}
\Big|\int_\Om u\,d x\Big|\le \|1\|_{\mV^0_{-\beta+1}(\Om)}\,\|u\|_{\mV^0_{\beta-1}(\Om)}  \le C\,\|u\|_{\mV^1_{\beta}(\Om)}.
\end{equation}
This allows us to define for all  $\beta<5/2,$ the space
\[
\mathcal{V}^1_\beta(\Om):=\{ u\in\mV_\beta^1(\Om)| \int_\Om u\,d x =0 \}.
\]
Note in particular that we have	$\mH_\#^1(\Om)=\mathcal{V}^1_0(\Om). $

\subsection{Scalar operators in Kondratiev spaces}\label{SectionScalaireCritique}

For $\beta\in\mathbb{R}$, define the continuous operator $\mA^{\beta}_{\eps}:\mathring{\mV}^1_{\beta}(\Om)\to(\mathring{\mV}^1_{-\beta}(\Om))^{\ast}$ such that
\[
\langle \mA^{\beta}_{\eps}\varphi,\varphi'\rangle=\int_{\Om}\eps\nabla\varphi\cdot\nabla\overline{\varphi'}\,dx,\qquad \forall\varphi\in\mathring{\mV}^1_{\beta}(\Om),\,\varphi'\in\mathring{\mV}^1_{-\beta}(\Om).
\]
In the same way, for $\beta\in(-5/2;5/2)$, define $\mA^{\beta}_{\mu}:\mathcal{V}^1_{\beta}(\Om)\to(\mathcal{V}^1_{-\beta}(\Om))^{\ast}$ such that
\[
\langle \mA^{\beta}_{\mu}\varphi,\varphi'\rangle=\int_{\Om}\mu\nabla\varphi\cdot\nabla\overline{\varphi'}\,dx,\qquad \forall\varphi\in\mathcal{V}^1_{\beta}(\Om),\varphi'\in\mathcal{V}^1_{-\beta}(\Om).
\]	
	
When applying the Kondratiev approach \cite{kond67,NaPl94,MaNP00,KoMR97,dauge2011non} and in particular, the Mellin transform, to analyse, roughly speaking, the properties of $\mA^{\beta}_{\eps}$, $\mA^{\beta}_{\mu}$ in a neighbourhood of $O$, one is led to study the operators $\mathscr{L}_\eps(\lambda):\mH^1_0(\varpi)\to\mH^1_0(\varpi)$, $\mathscr{L}_\mu(\lambda):\mH^1(\varpi)\to\mH^1(\varpi)$ defined via the Riesz representation theorem such that for $\lambda\in\Cplx$, 
\begin{equation}\label{DefLeps}
(\mathscr{L}_\eps(\lambda)\varphi,\phi)_{\mH^1(\varpi)}=\int_{\varpi}\eps\nabla_S \varphi\cdot\nabla_S \overline{\phi}\,d\boldsymbol{\theta}-\lambda(\lambda+1)\int_{\varpi} \eps\varphi \overline{\phi}\,d\boldsymbol{\theta},\qquad\forall \varphi,\phi\in\mH^1_0(\varpi),
\end{equation}
\begin{equation}\label{DefLmu}
(\mathscr{L}_\mu(\lambda)\varphi,\phi)_{\mH^1(\varpi)}=\int_{\varpi}\mu\nabla_S \varphi\cdot\nabla_S \overline{\phi}\,d\boldsymbol{\theta}-\lambda(\lambda+1)\int_{\varpi} \mu\varphi \overline{\phi}\,d\boldsymbol{\theta},\qquad\forall \varphi,\phi\in\mH^1(\varpi).
\end{equation}
Above $\varpi\subset\mathbb{S}^2$ is the domain introduced in (\ref{ChoixCase}), $\mH^1_0(\varpi):=\{\varphi\in\mH^1(\varpi)\,|\,\varphi=0\mbox{ on }\partial\varpi\}$ (note that $\mH^1_0(\varpi)=\mH^1(\varpi)$ if $\varpi=\mathbb{S}^2$) while $\nabla_S$ stands for the surface gradient on $\mathbb{S}^2$. Abusively, here for $\sigma=\eps$, $\mu$, we redefine $\sigma$ as the function such that $\sigma=\sigma_+$ in $\varpi_+:=\varpi\backslash\overline{\varpi_-}$ and $\sigma=\sigma_-$ in $\varpi_-$.

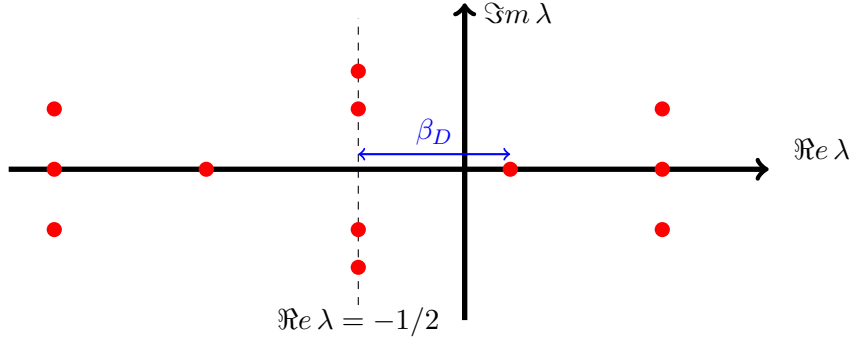
\begin{figure}[!ht]
\centering
\begin{tikzpicture}
\draw[draw=black,line width=2pt,->](-6,0)--(4,0);
\draw[draw=black,line width=2pt,->](0,-2)--(0,2.2);
\begin{scope}[xshift=-1.4cm]
\draw[draw=black,dashed](0,-1.8)--(0,2);
\node at (1,0.8) [anchor=north] {$\textcolor{blue}{\beta_D}$};
\draw[draw=blue,thick,<->](0,0.2)--(2,0.2);
\filldraw [red,draw=none] (2,0) circle (0.1);
\filldraw [red,draw=none] (-2,0) circle (0.1);
\filldraw [red,draw=none] (0,0.8) circle (0.1);
\filldraw [red,draw=none] (0,-0.8) circle (0.1);
\filldraw [red,draw=none] (0,1.3) circle (0.1);
\filldraw [red,draw=none] (0,-1.3) circle (0.1);
\filldraw [red,draw=none] (4,0) circle (0.1);
\filldraw [red,draw=none] (-4,0) circle (0.1);
\filldraw [red,draw=none] (4,0.8) circle (0.1);
\filldraw [red,draw=none] (4,-0.8) circle (0.1);
\filldraw [red,draw=none] (-4,0.8) circle (0.1);
\filldraw [red,draw=none] (-4,-0.8) circle (0.1);
\node at (0,-1.7) [anchor=north] {$\Re e\,\lambda=-1/2$};
\end{scope}
\node at (4.7,0) [anchor=south] {$\Re e\,\lambda$};
\node at (0.7,1.8) [anchor=south] {$\Im m\,\lambda$};
\end{tikzpicture}\\[-8pt]
\caption{Schematic picture of the eigenvalues of $\mathscr{L}_\eps$ in the complex plane when Assumption \ref{AssumptionCritique} is satisfied. One has something similar for $\mathscr{L}_\mu$. \label{fig eigenvalues}}
\end{figure}

Note that for $\sigma=\eps$, $\mu$, the symbol $\mathscr{L}_\sigma(\cdot)$ appears naturally when looking for functions with separate variables in coordinates $(r,\boldsymbol{\theta})$ such that $\div(\sigma\nabla s)=0$ in $\Om\cap B(O,\rho)$ together with the homogeneous Dirichlet/Neumann (according to the case) boundary conditions on $\partial\Om\cap \partial B(O,\rho)$. We say that $\lambda$ is in the spectrum of $\mathscr{L}_\sigma$ ($\lambda\in\mrm{spec}(\mathscr{L}_\sigma)$) if $\mathscr{L}_\sigma(\lambda)$ is not invertible. Additionally, $\lambda$ is an eigenvalue of $\mathscr{L}_\sigma$ if $\ker\mathscr{L}_\sigma(\lambda)\ne\{0\}$. Note that the analysis of the spectral properties of $\mathscr{L}_\sigma$ is made complicated by the presence of the sign-changing coefficient $\sigma$ both in the principal and compact part which prevents from identifying inner products. For this reason, we are unable to  recast the spectra corresponding to (\ref{DefLeps}), (\ref{DefLmu}) as the the spectra of some self-adjoint operators and actually, there are situations where complex eigenvalues exist. However, it has been shown in \cite[Chapter 3]{rihani2022maxwell} that when $\kappa_\sigma\ne-1$, the spectrum of $\mathscr{L}_\sigma$ is discrete, made of eigenvalues which do not accumulate in bounded regions of $\mathbb{C}$. Writing $\lambda(\lambda+1)=(\lambda+1/2)^2-(1/2)^2$, one observes that if $\lambda\in\mrm{spec}(\mathscr{L}_\sigma)$, then $-1-\lambda\in\mrm{spec}(\mathscr{L}_\sigma)$ (symmetry with respect to the point $-1/2+0i$, see Figure \ref{fig eigenvalues}). Additionally, using that $\sigma$ is real valued, one finds that $\lambda\in\mrm{spec}(\mathscr{L}_\sigma)$ implies $\overline{\lambda}\in\mrm{spec}(\mathscr{L}_\sigma)$ (symmetry with respect to the line $\Im m\,\lambda=0$). Then adapting the Kondratiev approach to the present situation with sign-changing coefficients (see \cite[Chapter 2]{rihani2022maxwell} for the details), one establishes the following result.
\begin{proposition}\label{PropoFredholmCriterion}
For $\sigma=\eps$, $\mu$, assume that $\kappa_\sigma\ne-1$. Then for $\beta\in\R$, the operator $\mA^{\beta}_{\sigma}$ is of Fredholm type if and only if $\mathscr{L}_\sigma$ has no eigenvalue on the line $\{\lambda\in\Cplx\,|\,\Re e\,\lambda=-1/2+\beta\}$.
\end{proposition}

\subsection{Hypersingularities}\label{ParaHyperSingu}

\begin{figure}[h!]
\centering
\includegraphics[width=5cm,angle =90]{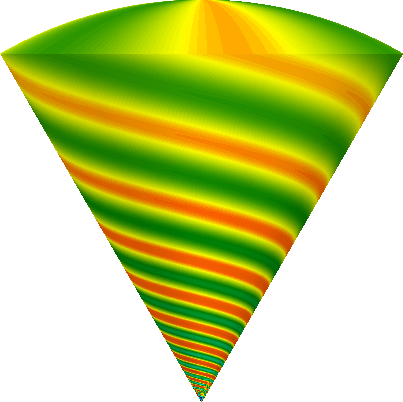}
\caption{Restriction to $\mathcal{K}_-$ of the imaginary part of a propagating singularity for an internal circular conical tip.}
\label{FigSingu}
\end{figure}

Observe that since $\mathring{\mV}^1_0(\Om)=\mH^1_0(\Om)$, $\mathcal{V}^1_0(\Om)=\mH^1_\#(\Om)$, we have $\mA_\eps^0=\mA_\eps$ and $\mA_\mu^0=\mA_\mu$ where $\mA_\eps$, $\mA_\mu$ are defined in (\ref{Def_Aeps}), (\ref{Def_Amu}). For this reason, for $\sigma=\eps$, $\mu$, the eigenvalues of $\mathscr{L}_\sigma$ located on the energy line $\Re e\,\lambda=-1/2$ ($\beta=0$) are of particular importance. Dividing (\ref{DefLeps}), (\ref{DefLmu}) by a constant if necessary, we see that the spectrum of $\mathscr{L}_\sigma$, which is a priori a function of $\sigma_+$, $\sigma_-$ , depends only on the contrast $\kappa_\sigma=\sigma_-/\sigma_+$. The Proposition \ref{PropoFredholmCriterion} directly gives the following characterization result for the set $I_\sigma$. 
\begin{proposition}For $\sigma=\eps$, $\mu$, we have
\[
I_\sigma=\{-1\}\cup\{\kappa_\sigma\in\R^\ast_-\,|\,\mbox{$\mathscr{L}_\sigma$ has an eigenvalue on the line $\Re e\,\lambda=-1/2$}\}.
\]
\end{proposition} 

In this work, we shall make the following assumption.
\begin{Assumption}	\label{AssumptionCritique}
We suppose that $\eps$ and $\mu$ are such that $\kappa_\eps\in I_\eps\backslash\{-1\}$, $\kappa_\mu \in I_\mu\backslash\{-1\}$.
\end{Assumption}
We say that $\kappa_\eps$, $\kappa_\mu$ are critical when, respectively, $\kappa_\eps\in I_\eps\backslash\{-1\}$, $\kappa_\mu\in I_\mu\backslash\{-1\}$. Note that in this work, we shall systematically exclude the cases where $\kappa_\eps=-1$ or $\kappa_\mu=-1$. In these latter situations, somehow singularities appear not only at the tip but all along the interface. The analysis is then completely different and not fully understood yet. We refer the interested reader to \cite{Ola95,behrndt2018indefinite,pankrashkin2018self,cacciapuoti2019self,nguyen2016limiting,nguyen2020limiting} for more details in this direction.\\
\newline
Let us look at Assumption \ref{AssumptionCritique} in the simplest geometrical setting. Assume here that $O\in\Om$ and that $\mathcal{K}_-$ coincides with a circular conical tip, that is 
\begin{equation}\label{ConicalTip}
\varpi_-=\{(\cos\theta\cos\phi,\sin\theta\cos\phi,\sin\phi)\,|\,-\pi\le\theta\le\pi,\,-\pi/2\le\phi <-\pi/2+\alpha\}
\end{equation}
for some $\alpha\in(0;\pi)$ (see Figure \ref{FigGeo} left). In this situation, as mentioned in \cite[\S2.1]{dhia2022maxwell}, adapting the results of \cite{KCHWS14,HePe18,LiPS20}, one can show that $I_\eps=I_\mu=(-1;-\aleph_{\alpha})$ (resp. $I_\eps=I_\mu=(-\aleph_{\alpha};-1)$ when $\alpha<\pi/2$ (resp. $\alpha>\pi/2$). Here $\aleph_{\alpha}$ is the constant defined by
\begin{equation}\label{DefInterSimple}
\aleph_{\alpha}:= \dfrac{_2\mrm{F}_1(1/2,1/2,1,\cos^2(\alpha/2))\,_2\mrm{F}_1(3/2,3/2,2,\sin^2(\alpha/2))}{_2\mrm{F}_1(1/2,1/2,1,\sin^2(\alpha/2))\,_2\mrm{F}_1(3/2,3/2,2,\cos^2(\alpha/2))}>0,
\end{equation}
where $_2\mrm{F}_1$ stands for the Gauss's hypergeometric function. Note that we have $\aleph_{\alpha}=1/\aleph_{\pi-\alpha}$ and $\aleph_{\alpha}\in(0;1)$ for $\alpha\in(0;\pi/2)$. Additionally, there holds for example $\aleph_{\pi/4}\approx0.218$ as well as $\textstyle \lim_{\alpha\to \pi/2}\aleph_{\alpha}=1$, $\textstyle\lim_{\alpha\to 0^+}\aleph_{\alpha}=0^+$, $\textstyle\lim_{\alpha\to \pi^-}\aleph_{\alpha}=+\infty$. As a consequence, for given materials such that $\kappa_\eps$, $\kappa_\mu\in(-1;0)$ (resp. $\kappa_\eps$, $\kappa_\mu\in(-\infty;-1)$), by taking $\alpha$ small enough (resp. large enough), one can always find geometries such that Assumption \ref{AssumptionCritique} is satisfied.\\
\newline
We come back to the general analysis. According to Proposition \ref{PropoFredholmCriterion}, when Assumption \ref{AssumptionCritique} holds, $\mA_\eps$, $\mA_\mu$ are not of Fredholm type. This is directly related to the existence of hypersingularities supported by the tip that we describe now. For $\sigma=\eps$, $\mu$, we denote by 
\[
\lambda^\sigma_j=-1/2+i\eta^\sigma_j,\qquad\qquad j=1,\dots,n^\sigma, 
\]
the eigenvalues of $\mathscr{L}_\sigma$ located on the line $\Re e\,\lambda=-1/2$. Let $I^\sigma_j=\dim\,\ker\mathscr{L}_\sigma(\lambda^\sigma_j)$ stand for the geometric multiplicity of $\lambda^\sigma_j$ and let $\varphi^\sigma_{j,1,0},\dots,\varphi^\sigma_{j,I^\sigma_j,0}$ be a canonical system of eigenfunctions of $\mathscr{L}_\sigma$ corresponding to the eigenvalue $\lambda^\sigma_j$. For certain contrasts, generalized eigenfunctions can also exist that we have to take into account. Denote by $\kappa^{\sigma}_{j,k}$ the partial multiplicity of $\lambda^\sigma_j$ (see \cite[\S5.1.1]{KoMR97} for the definition). Finally, let 
\[
\varphi^\sigma_{j,k,0},\dots,\varphi^\sigma_{j,k,\kappa^{\sigma}_{j,k}-1},\qquad\quad k=1,\dots,I^\sigma_j,
\]
be a canonical system of Jordan chains (see again \cite[\S5.1.1]{KoMR97}) corresponding to the eigenvalue $\lambda^\sigma_j$. For $j=1,\dots,n^{\sigma}$, $k=1,\dots,I^\sigma_j$, $l=0,\dots,\kappa^{\sigma}_{j,k}-1$, we define the hypersingularity
\begin{equation}\label{defHypersingu}
s^{\sigma}_{j,k,l}:=r^{-1/2+i\eta^\sigma_j}\sum_{p=0}^{l}\frac{1}{p!} (\log r)^p\varphi^\sigma_{j,k,l-p}(\boldsymbol{\theta}).
\end{equation}
Note that if $\lambda_j=-1/2+i\eta_j$ is a simple eigenvalue (geometric multiplicity equal to one and no generalized eigenfunction), we just have $s^{\sigma}_{j,1,0}=r^{-1/2+i\eta}\varphi^\sigma_{j,1,0}(\boldsymbol{\theta})$ (see Figure \ref{FigSingu} for a representation of such function). Additionally, using the definition of Jordan chains, we find that for $\lambda^\sigma_j\ne-1/2$, we have $\kappa^\sigma_{j,k}=1$, \textit{i.e.} absence of generalized eigenfunction associated with the eigenpair $(\lambda^\sigma_j,\varphi^\sigma_{j,k,0})$, as soon as there holds
\[
\int_{\varpi} \sigma\varphi^\sigma_{j,k,0}\,\varphi^\sigma_{j,\tilde{k},0}\,d\boldsymbol{\theta} \ne0\qquad\mbox{ for some }\tilde{k}\in\{1,\dots,I^\sigma_j\}.
\]
The hypersingularities $s^\sigma$ in (\ref{defHypersingu}) satisfy
\begin{equation}\label{PropertyKernel}
\div(\sigma\nabla s^{\sigma})=0\qquad\mbox{ in }\Om\cap B(O,\rho)
\end{equation}
together with the homogeneous Dirichlet/Neumann (according to the case) boundary conditions on $\partial\Om\cap \partial B(O,\rho)$. Moreover, they are in $\mL^2(\Om)$ but lay ``just outside'' of $\mH^1(\Om)$. More precisely, we have $s^{\sigma}\notin \mH^1(\Om)$ and $r^{\tau}s^{\sigma}\in\mH^1(\Om)$ for all $\tau>0$. This is directly related to the calculus
\[
\int_{\delta}^1|\nabla(r^{-1/2+i\eta})|^2\,r^2dr=|-1/2+i\eta|^2\int_{\delta}^1r^{-1}dr=(1/4+\eta^2)|\ln\delta|\underset{\delta\to0^+}{\longrightarrow}+\infty.
\]
The hypersingularities are also known as propagating or black hole singularities. In the following, to simplify, we will simply call them singularities. They can be interpreted as waves guided by the interface  between the two materials that propagate to or from $O$. 
The group velocity of these waves tend to $0$ as $r\to0$ so that they never reach $O$ (this is why they are called black hole waves). Everything happens as if energy was trapped at the tip, see the important identity (\ref{ImportantIdentity}) below to understand this sentence, and the point $O$ plays the role of infinity. For adapted numerical methods to catch them, one may look at \cite{BCCC16,HeKa18,BoHM21,HeKR21}. For related problems, we refer the reader to \cite{BoCh13,PePu17,BoZh19,BDTZ19,HaPa20,Perf21}.\\
\newline
To discard in particular problems of boundary conditions on $\partial\Om$ far from $O$, classically we multiply the singularities by a cut-off function. Introduce $\chi\in\mathscr{C}^{\infty}(\R)$ such that $\chi(r)=1$ for $r\le\rho/2$ and $\chi(r)=0$ for $r\ge \rho$. For $\sigma=\eps,\mu$, this leads us to define the space
\begin{equation}\label{DefSpace}
\mathcal{S}_\sigma:=\mrm{span}\{\ \chi(r)s^{\sigma}_{j,k,l}(x)\ ,\,j=1,\dots,n^{\sigma},\,k=1,\dots,I^\sigma_j,\,l=0,\dots,\kappa^{\sigma}_{j,k}-1\}.
\end{equation}
Exploiting (\ref{PropertyKernel}), for $\mathfrak{s}\in\mathcal{S}_\sigma$, we obtain
\begin{equation}\label{PropertySingu}
\div(\sigma\nabla \mathfrak{s})\in\mL^2(\Om)\qquad\mbox{ and }\qquad \div(\sigma\nabla \mathfrak{s})=0\mbox{ in }\Om\cap B(O,\rho/2).
\end{equation}
The study of \cite[Chapter 2]{rihani2022maxwell} gives the following result. 
\begin{lemma}
Suppose that Assumption \ref{AssumptionCritique} holds. Then $\mathcal{S}_\eps$ and $\mathcal{S}_\mu$ are of finite dimension and their dimensions are even, i.e. we have $\dim(\mathcal{S}_\eps)=2N_\eps$, $\dim(\mathcal{S}_\mu)=2N_\mu$ with $N_\eps,N_\mu\in\N^\ast$.
\end{lemma}
\begin{remark}
In the particular case of the circular conical tip described in (\ref{ConicalTip}), one can show that by taking contrasts $\kappa_\eps$, $\kappa_\mu$ close to $-1$, one can get dimensions $N_\eps$, $N_\mu$ as large as desired.
\end{remark}

In what follows we explain in a brief way how to take into account the singularities in the functional framework to get Fredholm operators.

\subsection{Additional properties for the scalar operators in Kondratiev spaces}
Introduce the quantities 
\begin{equation}\label{defBetaDN}
\begin{array}{l}
\beta_D:=\min \{\Re e\,(\lambda-1/2) \,|\,\lambda\in\mrm{spec}(\mathscr{L}_\eps)\mbox{ and }\Re e\,\lambda>-1/2\},\\[6pt]
\beta_N:=\min (\{\Re e\,(\lambda-1/2) \,|\,\lambda\in\mrm{spec}(\mathscr{L}_\mu)\mbox{ and }\Re e\,\lambda>-1/2\}\cup\{5/2\})
\end{array}
\end{equation}
(see the illustration of Figure \ref{fig eigenvalues}). The results of \cite[Chapter 2]{rihani2022maxwell} guarantee that we have both $\beta_D>0$ and $\beta_N>0$. With this definition, all the eigenvalues of $\mathscr{L}_\eps$ (resp. $\mathscr{L}_\mu$) in the strip $-\beta_D<\Re e\,\lambda<\beta_D$ ($-\beta_N<\Re e\,\lambda<\beta_N$) are located on the line $\Re e\,\lambda=-1/2$. Additionally, we have the following result (we recall that the index of a Fredholm operator is defined as the difference of the dimensions of its kernel and of its cokernel). 
\begin{proposition}\label{PropositionFreholm}
Suppose that Assumption \ref{AssumptionCritique} holds.\\[2pt]
For $\beta\in(0;\beta_D)$, $\mA_\eps^{\beta}$ is of index $N_\eps$ while $\mA_\eps^{-\beta}$ is of index $-N_\eps$.\\[2pt]
For $\beta\in(0;\beta_N)$, $\mA_\mu^{\beta}$ is of index $N_\mu$ while $\mA_\mu^{-\beta}$ is of index $-N_\mu$.
\end{proposition}		
Note that we can prove that the dimension of $\ker\mA_\eps^{-\beta}$ (resp. $\ker\mA_\mu^{-\beta}$) is independent of $\beta\in(0;\beta_D)$ (resp. $\beta\in(0;\beta_N)$). Therefore $\mA_\eps^{-\beta}$ (resp. $\mA_\mu^{-\beta}$) is injective for all $\beta\in(0;\beta_D)$ (resp. $\beta\in(0;\beta_N)$) if and only if it is injective for one $\beta\in(0;\beta_D)$ (resp. one $\beta\in(0;\beta_N)$). To simplify the analysis below and to avoid to be obliged to handle kernels of finite dimensions, we shall make  the
\begin{Assumption}\label{AssumptionNoTrappedMode}
We suppose that $\eps$ (resp. $\mu$) is  such that $\mA_\eps^{-\beta}$ (resp. $\mA_\mu^{-\beta}$) is injective for all $\beta\in(0;\beta_D)$ (resp. $\beta\in(0;\beta_N)$).
\end{Assumption}
With Proposition \ref{PropositionFreholm}, this gives the estimates
\begin{equation}\label{Estimmonomorphism}
\begin{array}{lll}
\forall\beta\in[0;\beta_D),&\qquad	\|u\|_{\mathring{\mV}^1_{-\beta}(\Om)}\leq C\,\|\mA^{-\beta}_\eps u\|_{(\mathring{\mV}^1_\beta(\Om))^\ast},\ \ \quad \forall u\in\mathring{\mV}^1_{-\beta}(\Om),\\[8pt]
\forall\beta\in[0;\beta_N),&\qquad			\|u\|_{\mathcal{V}^1_{-\beta}(\Om)}\leq C\,\|\mA^{-\beta}_\mu u\|_{(\mathcal{V}^1_\beta(\Om))^\ast},\qquad \forall u\in\mathcal{V}^1_{-\beta}(\Om).
\end{array}
\end{equation}
With the help of the residue theorem applied to the symbols $\lambda\mapsto\mathscr{L}_\eps(\lambda)$, $\lambda\mapsto\mathscr{L}_\mu(\lambda)$ (adapt \cite[Theorem 5.4.2]{KoMR97}), one shows the following regularity result.
\begin{proposition}\label{RegulariteScalaire}
Suppose that Assumptions \ref{AssumptionCritique}-\ref{AssumptionNoTrappedMode}  hold.\\[2pt]
If $u\in\mH^1_0(\Om)$ satisfies  $\div(\eps\nabla u)\in(\mathring{\mV}^1_\beta(\Om))^\ast\subset(\mathring{\mV}^1_{-\beta}(\Om))^\ast$ for some  $\beta\in(0;\beta_D)$, then $u\in\mathring{\mV}^1_{-\beta}(\Om)$.\\[2pt]
If $u\in\mH^1_\#(\Om)$ satisfies  $\div(\mu\nabla u)\in(\mathcal{V}^1_\beta(\Om))^\ast\subset(\mathcal{V}^1_{-\beta}(\Om))^\ast$ for some  $\beta\in(0;\beta_N)$, then $u\in\mathcal{V}^1_{-\beta}(\Om)$.		
\end{proposition}
The singularities belong to the domains of the operators $\mA_\eps^{\beta}$, $\mA_\mu^{\beta}$. However these operators are not yet satisfactory because they are onto (as adjoints of injective operators) but not injective. In order to construct isomorphisms, we have to make a selection of the singularities and incorporate in the functional framework only some of them. This is the next step in the analysis.

\subsection{Mandelstam radiation principle }
The properties (\ref{PropertySingu}) ensure that for $\sigma=\eps,\mu,$ we can define the form $q_\sigma(\cdot,\cdot):\mathcal{S}_\sigma\times\mathcal{S}_\sigma\to\C$ such that
$$q_\sigma(u,v)=\int_{\Om}-\div(\sigma\nabla u)\overline{v}+u\,\div(\sigma\nabla \overline{v})\,dx\qquad\forall u,v\in\mathcal{S}_\sigma.
$$
Note that $q_\sigma(\cdot,\cdot)$ is sesquilinear and anti-Hermitian, i.e. it is symplectic.
Using  the dominated convergence
theorem and integrating by parts, we obtain
$$ q_\sigma(u,v)=\lim\limits_{\delta\to0}\int_{\{|x|=\delta\}\cap\Om}\sigma (\partial_r u \overline{v}- \partial_r u \overline{v})\,d s.$$
This shows that $q_\sigma(\cdot,\cdot)$ does not depend on the choice of the cut-off function $\chi$ in (\ref{DefSpace}).   On the other hand, we observe that for all $u \in\mathcal{S}_\sigma$,
\[
q_\sigma(u,u)=2i\Im m(\int_{\Om}u\,\div(\sigma\nabla\overline{u})\,dx)\in i\R. 
\]
Physically, the magnitude $q_{\sigma}(u,u)$ represents the energy transported by the wave   $u\in\mathcal{S}_\sigma$ to or from the point $O$ (depending on the sign). A wave $u\in\mathcal{S}_{\sigma}$ is said to be outgoing (resp. incoming) if $\Im m(q_\sigma(u,u))>0$ (resp. $\Im m(q_\sigma(u,u))<0$ ). If $u\in\mathcal{S}_{\sigma}$ satisfies $q_\sigma(u,u)=0$, we say that $u$ is unclassified.
It has been shown in \cite[Chapter 2]{rihani2022maxwell} that $q_\sigma(\cdot,\cdot)$ is non-degenerate ($q_\sigma(u,v)=0$ for all $v\in\mathcal{S}_\sigma$ implies $u\equiv0$).	By applying the Sylvester's law of inertia, we obtain the following result.
\begin{lemma}\label{ChoiceDeLaBase}
Suppose that Assumption \ref{AssumptionCritique} holds. For $\sigma=\eps$, $\mu$, there exists  $(\mathfrak{s}^\pm_{\sigma,j})_{j=1,\dots,N_\sigma}$  a basis of $\mathcal{S}_\sigma$ such that we have
\begin{equation}\label{OrthoConditions}
q_\sigma(\mathfrak{s}^\pm_{\sigma,j},\mathfrak{s}^\pm_{\sigma,k})=\pm i\delta_{j,k},\qquad q_\sigma(\mathfrak{s}^\pm_{\sigma,j},\mathfrak{s}^\mp_{\sigma,k})=0 \quad\mbox{ and }\quad\mathfrak{s}^+_{\sigma,j}=\overline{\mathfrak{s}^-_{\sigma,j}}\qquad \mbox{ for }j,k=1,\dots,N_\sigma.
\end{equation}
\end{lemma}
In the literature (see in particular \cite{NaPl94,Naza13,nazarov2014umov}), the decomposition presented in the previous lemma is known as the Mandelstam radiation principle. It means that the space of waves can be decomposed as the sum of a space of outgoing waves and a space of incoming waves.
\begin{remark}\label{RqChoixdeLaBase}
We emphasize that the choice of the basis  $(\mathfrak{s}^\pm_{\sigma,j})_{j=1,\dots,N_\sigma}$ is not unique. More precisely, one can find an infinite number of bases satisfying all the conditions of  Lemma \ref{ChoiceDeLaBase}. From a mathematical point of view, this is not a problem and any choice of basis provides a functional framework in  which the corresponding scalar problem is well-posed. More physically however, in most cases there is only one particular choice which is consistent with the limiting  absorption principle. We will come back to this point in \S\ref{SectionELimiting}.
\end{remark}
From now on, for $\sigma=\eps$, $\mu$, we fix  $(\mathfrak{s}^\pm_{\sigma,j})_{j=1,\dots,N_\sigma}$ a basis of   $\mathcal{S}_\sigma$ satisfying  the orthogonality relations of Lemma \ref{ChoiceDeLaBase}. Moreover, we define the space
\begin{equation}\label{DefBases}
\mathcal{S}_\sigma^+:=\mrm{span}\{\mathfrak{s}_{\sigma,j}^+,\,j=1,\dots,N_\sigma\}.
\end{equation}
Using (\ref{OrthoConditions}), one gets directly the following result. 
\begin{lemma}	\label{FluxScalairDefPosi}
		Suppose that Assumption \ref{AssumptionCritique} holds. For $\sigma=\eps$, $\mu$, if $\mathfrak{s}\in\mathcal{S}_\sigma^+$ satisfies $q_\sigma(\mathfrak{s},\mathfrak{s})=0$, then $\mathfrak{s}\equiv0$.
	\end{lemma}
\subsection{A new functional framework for the scalar problems}
For $\beta>0$, define the so-called space with detached asymptotics
\[
\mathring{\mV}^\out_{\beta}(\Om):=\mathring{\mV}^1_{-\beta}(\Om)\oplus\mathcal{S}^+_\eps.
\]
Endowed with the norm
\[
\|u\|_{\mathring{\mV}^\out_{\beta}(\Om)}:=(\|\tilde{u}\|_{\mV^1_{-\beta}(\Om)}+\|\mathfrak{s}^+_\eps\|_{\mV^1_{\beta}(\Om)})^{1/2},\quad u=\tilde{u}+\mathfrak{s}^+_\eps,\quad \tilde{u}\in\mathring{\mV}^1_{-\beta}(\Om),\ \mathfrak{s}_\eps^+\in\mathcal{S}_\eps^+,
\]
this is a Banach space. Note that $\mathcal{S}_\eps^+$ is of finite dimension so that in this space $\|\cdot\|_{\mV^1_{\beta}(\Om)}$ is one norm among others. However it will be convenient in the sequel. Then introduce the operator $\mA^\out_\eps:\mathring{\mV}^\out_\beta(\Om)\to(\mathring{\mV}^1_\beta(\Om))^\ast$ such that 
\[
\langle\mA^\out_\eps u,v\rangle=\fint_\Om\eps\nabla u\cdot\nabla \overline{v}\,dx,\qquad\forall u=\tilde{u}+\mathfrak{s}^+_\eps\in\mathring{\mV}^\out_\beta(\Om),\, v\in\mathring{\mV}^1_\beta(\Om),
\]
where
\begin{equation}\label{defExtenInt}
\fint_\Om\eps\nabla u\cdot\nabla \overline{v}\,dx:=\int_\Om\eps\nabla\tilde{u}\cdot\nabla \overline{v}\,dx-\int_{\Om}\div(\eps\nabla \mathfrak{s}_\eps^+)\overline{v}\,dx.
\end{equation}
We emphasize that since $\div(\eps\nabla \mathfrak{s}_\eps^+)$ vanishes in a neighbourhood of the origin, (\ref{defExtenInt}) is well defined and we have 
\[
\Big|\fint_\Om\eps\nabla u\cdot\nabla \overline{v}\,dx\Big| \le C\,\|u\|_{\mathring{\mV}^\out_{\beta}(\Om)}\|v\|_{\mV^1_{\beta}(\Om)}
\]
where $C$ is independent of $u$, $v$. This guarantees that  $\mA^\out_\eps:\mathring{\mV}^\out_\beta(\Om)\to(\mathring{\mV}^1_\beta(\Om))^\ast$ is continuous. Note also that for $\varphi\in\mathscr{C}^\infty_0(\Om\backslash\{O\})$, we simply have
\[
\langle\mA^\out_\eps u,\varphi\rangle=\int_\Om\eps\nabla u\cdot\nabla \overline{\varphi}\,dx.
\]
Similarly, for $\beta>0$, set
\[
\mathcal{V}^\out_{\beta}(\Om):=\{u\in\mV^1_{-\beta}(\Om)\oplus\mathcal{S}^+_\mu\,|\, \int_\Om u\, dx=0\}.
\]	  
Equipped with the norm 
\[
\|u\|_{\mathcal{V}^\out_\beta(\Om)}:=(\|\tilde{u}\|_{\mV^1_{-\beta}(\Om)}+\|\mathfrak{s}^+_\mu\|_{\mV^1_{\beta}(\Om)})^{1/2},\quad u=\tilde{u}+\mathfrak{s}^+_\mu,\quad \tilde{u}\in\mV^1_{-\beta}(\Om),\ \mathfrak{s}_\mu^+\in\mathcal{S}_\mu^+,
\]
this is also a Banach space. For $\beta\in(0;5/2)$, define the continuous operator $\mA^\out_\mu:\mathcal{V}^\out_\beta(\Om)\to (\mathcal{V}^1_{\beta}(\Om))^\ast$ such that
\[
\langle\mA^\out_\mu u,v\rangle=\fint_\Om\mu\nabla u\cdot\nabla \overline{v}\,dx,\qquad \forall  u=\tilde{u}+\mathfrak{s}_\mu^+\in\mathcal{V}^\out_\beta(\Om),\,v\in\mathcal{V}^1_\beta(\Om)
\]
where
\[
\fint_\Om\mu\nabla u\cdot\nabla \overline{v}\,dx:=\int_\Om\mu\nabla\tilde{u}\cdot\nabla \overline{v}\,dx-\int_{\Om}\div(\mu\nabla\mathfrak{s}_\mu^+)\overline{v}\,dx.
\]
As for $\mA^\out_\eps$, note that for $\varphi\in\mathscr{C}^\infty_0(\Om\backslash\{O\})$, we simply have
\[
\langle\mA^\out_\mu u,\varphi\rangle=\int_\Om\mu\nabla u\cdot\nabla \overline{\varphi}\,dx.
\]
Owing to \cite[Chapter 2]{rihani2022maxwell}, we admit here the following crucial theorem for the scalar problems, which extends to 3D the results of \cite{bonnet2013radiation}. 
\begin{theorem}\label{MainThmScalaire}
		Suppose that Assumptions
		\ref{AssumptionCritique}-\ref{AssumptionNoTrappedMode} hold. Then for all $\beta\in(0;\beta_D)$ (resp.  $\beta\in(0;\beta_N)$), the operator $\mA^\out_\eps$(resp. $\mA^\out_\mu$) is an isomorphism. Here $\beta_D$, $\beta_N$ are defined in  (\ref{defBetaDN}).
\end{theorem}
\begin{remark}
Note that looking for solutions in $\mathring{\mV}^\out_{\beta}(\Om)$, $\mathcal{V}^\out_\beta(\Om)$ boils down to impose radiation conditions at the origin.
\end{remark}

\section{A new framework for the electric problem}\label{SectionPbElectrique}

In Section \ref{SectionNecessity} we will prove that when $\kappa_\eps$ and $\kappa_\mu$ are critical, the Maxwell's equations in the classical $\Lspace^2$ spaces are not well-posed. In this section, the heart of the article, we propose a new framework. Let us fix $\beta$ once for all such that
\begin{equation}\label{ChoixPoids}
\beta\in(0;\beta_\star)\qquad\mbox{with }\beta_\star:=\min\{\beta_D,\beta_N,1/2\}.
\end{equation}
Here $\beta_D$, $\beta_N$ are the weights appearing in (\ref{defBetaDN}). For $m\in\N$, set
\[ 
\mmV^m_{\beta}(\Om):=(\mV^m_\beta(\Om))^3
\]
and endow this space with the norm 
\[
\|\u\|_{\mmV^m_\beta(\Om)}:=\Big(\sum_{i=1}^3 \|u_i\|^2_{\mV^m_\beta(\Om)}\Big)^{1/2},\qquad\forall \u=(u_1,u_2,u_3)\in\mmV^m_\beta(\Om).
\]
Then define 
\[
\begin{array}{l}
\nabla \mathcal{S}^+_\eps:=\mrm{span}\{\nabla\mathfrak{s}^+_{\eps,j},j=1,\dots,N_\eps\},\qquad \nabla \mathcal{S}^+_\mu:=\mrm{span}\{\nabla\mathfrak{s}^+_{\mu,j},j=1,\dots,N_\mu\}\\[8pt]
\Hspace_N^{\out,\beta}(\curl):= \{\u\in\nabla\mathcal{S}^+_\eps\oplus\mmV^0_{-\beta}(\Om)\,|\,\curl \u\in\mmV^0_\beta(\Om),\, \u\times\nu=0\,\mbox{on }\partial\Om\backslash\{O\}\} \\[8pt]
	\Houtspace:= \{\u\in\nabla\mathcal{S}^+_\eps\oplus\mmV^0_{-\beta}(\Om)\,|\,\curl \u\in\mu\nabla \mathcal{S}^+_\mu\oplus\mmV^0_{-\beta}(\Om),\, \u\times\nu=0\,\mbox{on }\partial\Om\backslash\{O\}\}.
\end{array}
\]
To simplify the  presentation of the results below, we adopt the following convention:\\[4pt]
- if $\u\in\Hspace_N^{\out,\beta}(\curl)$, let $\tilde{\u}$, $\mathfrak{s}_{\u,\eps}$ be the elements of $\mmV^0_{-\beta}(\Om)$, $\mathcal{S}^+_\eps$ such that $\u=\tilde{\u}+\nabla \mathfrak{s}_{\u,\eps}$;\\[4pt]
- if $\u\in\Houtspace$, let $\boldsymbol{\psi}_{\u}$, $\mathfrak{s}_{\u,\mu}$ be the elements of $\mmV^0_{-\beta}(\Om)$, $\mathcal{S}^+_\mu$ such that $\curl\u=\boldsymbol{\psi}_{\u}+\mu\nabla \mathfrak{s}_{\u,\mu}$.\\
\newline
We equip $\nabla \mathcal{S}^+_\eps$, $\nabla \mathcal{S}^+_\mu$ with the norm $\mmV^0_{\beta}(\Om)$, which is one norm among others in these spaces of finite dimension. On the other hand, 
for $\u=\tilde{\u}+\nabla \mathfrak{s}_{\u,\eps}^+\in \Hspace_N^{\out,\beta}(\curl)$, we set
\[
\|\u \|_{\Hspace_N^{\out,\beta}(\curl)}:=(\|\Tilde{\u}\|^2_{\mmV^0_{-\beta}(\Om)}+\|\nabla \mathfrak{s}_{\u,\eps}^+\|^2_{\mmV^0_{\beta}(\Om)}+\|\curl\u\|^2_{\mmV^0_\beta(\Om)})^{1/2},
\]
while for $\u=\Tilde{\u}+\nabla\mathfrak{s}_{\u,\eps}^+\in\Houtspace$   such that  $\curl\u=\boldsymbol{\psi}_{\u}+\mu\nabla \mathfrak{s}_{\u,\mu}$ , we denote
\[
\|\u \|_{\Houtspace}:=(\|\Tilde{\u}\|^2_{\mmV^0_{-\beta}(\Om)}+\|\nabla \mathfrak{s}_{\u,\eps}^+\|^2_{\mmV^0_{\beta}(\Om)}+\|\boldsymbol{\psi}_{\u}\|^2_{\mmV^0_{-\beta}(\Om)}+\|\nabla \mathfrak{s}_{\u,\mu}^+\|^2_{\mmV^0_{\beta}(\Om)})^{1/2}. 
\]
Endowed with their associated norms, all the previous spaces are hilbertian. Note also that we have the continuous inclusions
\[
\Houtspace\subset \Hspace_N^{\out,\beta}(\curl)\subset \mmV^0_{\beta}(\Om)
\]
and that these spaces are not empty because they contain 
$\Cgras^{\infty}_0(\Om\backslash\{O\}):=\mathscr{C}^{\infty}_0(\Om\backslash\{O\})^3$. Observe also that we have both $\Houtspace\not\subset\Hspace_N(\curl)$ and $\Hspace_N(\curl)\not\subset\Houtspace$.

\begin{remark}
Note that if $\u\in\Houtspace$, according to (\ref{PropertySingu}), we have 
\begin{equation}\label{ProprieteHcurl}
\div\,\boldsymbol{\psi}_{\u}=-\div(\mu\nabla \mathfrak{s}_{\u,\mu})\in\mL^2(\Om).
\end{equation}
\end{remark}

\subsection{Definition of the electric problem}
In \S\ref{SectionELimiting} below, we explain that the limiting absorption principle leads to look for an electric field in the space $\Houtspace$. For this reason, we consider the following electric problem associated with \eqref{EqMaxwellInitiale}-\eqref{CLMaxwell}

\begin{equation}\label{SectionEEqE}
		\begin{array}{|ll}
		\multicolumn{2}{|l}{\mbox{Find } \u\in\Houtspace\mbox{ such that  }}\\[3pt]
			\curl(\mu^{-1}\curl \u)-\om^2 \eps \u=i\om\boldsymbol{J} & \mbox{in }  \Om\backslash\{O\} \\[3pt]
			\u\times \nu=0  & \mbox{on } \partial \Om\backslash\{O\}.
\end{array}
	\end{equation}
Above we wrote the equations in $\Om\backslash\{O\}$ and not in $\Om$ to allow for singular behaviours at the origin. The next step consists in writing a variational formulation of \eqref{SectionEEqE}. We assume that $\J$ belongs to $\mmV^0_{-\eta}(\Om)$ for some $\eta>0$ and satisfies $\div\,\J=0$ in $\Om\backslash\{O\}$. Note that if $\u\in\Houtspace$, the relation $\curl\u=\boldsymbol{\psi}_{\u}+\mu\nabla \mathfrak{s}_{\u,\mu}$ implies
\[
\curl(\mu^{-1}\curl \u)=\curl(\mu^{-1}\psib_{\u})\quad\mbox{ in }\Om\backslash\{O\}. 
\]
Exploiting this, we consider the problem
\begin{equation}\label{SectionEFv1}
		\begin{array}{|l}
			\text{Find } \u \in   \Houtspace \text{ such that } \\
			\int_\Om \mu^{-1} \boldsymbol{\psi}_{\u}\cdot \curl\vb\,dx -\om^2\fint_\Om \eps\u\cdot\vb\,dx =i\om\int_\Om \boldsymbol{J} \cdot\vb\,dx,\qquad\forall\v\in  \Hspace_N^{\out,\beta}(\curl),
		\end{array}
	\end{equation}
where for $\u\in\Houtspace$, $\v\in\Hspace_N^{\out,\beta}(\curl)$, we set
	$$ \fint_\Om \eps\u\cdot\vb\,dx:= \int_\Om \eps\Tilde{\u}\cdot\overline{\Tilde{\v}}\,dx+\int_{\Om}\eps\nabla \mathfrak{s}_{\u,\eps}\cdot\overline{\Tilde{\v}}\,dx+\int_{\Om}\eps\Tilde{\u}\cdot\nabla\overline{\mathfrak{s}_{\v,\eps}}\,dx-\int_\Om\div(\eps\nabla\mathfrak{s}_{\u,\eps})\,\overline{\mathfrak{s}_{\v,\eps}}\,dx.$$
It is not difficult to show that the sesquilinear form
 $$(\u,\v)\mapsto \fint_\Om \eps\u\cdot\vb\,dx$$
  is well-defined for  $\u$, $\v\in\Hspace_N^{\out,\beta}(\curl)$ and  is continuous in this space. However, it is not hermitian. Indeed, for  $\u$,$\v\in\Hspace_N^{\out,\beta}(\curl)$, we have
	\begin{equation}
		\fint_\Om \eps\u\cdot\vb\,dx-  \overline{\fint_\Om \eps\v\cdot\overline{\u}}\,dx=-\int_\Om\div(\eps\nabla\mathfrak{s}_{\u,\eps})\overline{\mathfrak{s}_{\v,\eps}}\,dx+\int_\Om\mathfrak{s}_{\u,\eps}\,\div(\eps\nabla\overline{\mathfrak{s}_{\v,\eps}})\,dx=q_\eps(\mathfrak{s}_{\u,\eps},\mathfrak{s}_{\v,\eps}).
\end{equation}
Note that in (\ref{SectionEFv1}) the solution and the test functions do not belong to the same space. The interest of this formulation is justified by the following result. 
\begin{proposition}\label{SectionEResEquivDis}
Every solution of \eqref{SectionEEqE} solves \eqref{SectionEFv1}. Conversely, every solution of \eqref{SectionEFv1} solves \eqref{SectionEEqE}. 
\end{proposition}
\begin{proof}
		Since $\Cgras^{\infty}_0(\Om\backslash\{O\})\subset\Hspace_N^{\out,\beta}(\curl)$, any solution to  \eqref{SectionEFv1} is   a solution to  \eqref{SectionEEqE}. Now, let us show the converse statement.   To proceed, we start by observing that  if $\u$ solves \eqref{SectionEEqE}, then using Proposition \ref{AppednixResDensity} which guarantees that $\Cgras^{\infty}_0(\Om\backslash\{O\})$ is dense in $\{\v\in\Hspace_N^{\out,\beta}(\curl)\,|\,\mathfrak{s}_{\v,\eps}=0\}$, we find that $\u$ satisfies
$$
\int_\Om \mu^{-1} \boldsymbol{\psi}_{\u}\cdot \curl\vb\,dx -\om^2\int_\Om \eps\u\cdot\vb\,dx =i\om\int_\Om \boldsymbol{J} \cdot\vb\,dx,\quad\forall \v\in\Hspace_N^{\out,\beta}(\curl)\mbox{ such that }\mathfrak{s}_{\v,\eps}=0.
$$
Therefore, it only remains to prove that \eqref{SectionEFv1} holds true for  $\v=\nabla\varphi\in\nabla \mathcal{S}_\eps^+$. Introduce $(\varphi_n)_{n}$ a sequence of elements of $\mathscr{C}^{\infty}_0(\Om\backslash\{O\})$ that converges to $\varphi$ in $\mathring{\mV}^1_\beta(\Om)$. Using Lemma \ref{SectionELemmeEquiv} below, we can write
\[
\begin{array}{ll}
 &\int_\Om \mu^{-1} \boldsymbol{\psi}_{\u}\cdot \curl\vb\,dx -\om^2\fint_\Om \eps\u\cdot\vb\,dx \\[10pt]
=& -\om^2\fint_\Om \eps\u\cdot\nabla\overline{\varphi}\,dx = \lim_{n\to+\infty}-\om^2\int_\Om \eps\u\cdot\nabla\overline{\varphi_n}\,dx = \lim_{n\to+\infty} i\om\int_\Om \J\cdot\nabla\overline{\varphi_n}\,dx =i\om\int_\Om \boldsymbol{J} \cdot\vb\,dx.
\end{array}
\]		
This ends the proof
\end{proof}
\begin{lemma}\label{SectionELemmeEquiv}
Let $\u$ be a field of $\Hspace_N^{\out,\beta}(\curl)$ and $(\varphi_n)_{n}$ a sequence of elements of $\mathscr{C}^{\infty}_0(\Om\backslash\{O\})$ that converges to some $\varphi\in\mathring{\mV}^1_\beta(\Om)$. Then we have
\[
\lim_{n\to+\infty} \int_\Om\eps \u\cdot \nabla \overline{\varphi_n}\,dx=\fint_\Om\eps \u\cdot \nabla \overline{\varphi}\,dx .
\]
\end{lemma}

\begin{proof}
By integrating by parts,  we obtain 
$$\int_\Om\eps \u\cdot \nabla \overline{\varphi_n}\,dx=\int_\Om\eps\tilde{\u}\cdot\nabla\overline{\varphi_n}\,dx-\int_\Om\div(\eps\nabla \mathfrak{s}_{\u,\eps})\,\overline{ \varphi_n}\,dx.
$$
Since $\div(\eps\nabla \mathfrak{s}_{\u,\eps})\in\mL^2(\Om)$ vanishes in a neighbourhood of $O$ (see (\ref{PropertySingu})) and since the convergence of $(\varphi_n)_{n}$ to $\varphi$ in $\mathring{\mV}^1_\beta(\Om)$ implies  the convergence of
$(\nabla\varphi_n)_{n}$ to $\nabla \varphi$ in $\mmV^0_\beta(\Om)$ as well as the one of  $(\varphi_n)_{n}$ to $\varphi$ in $\mL^2(\Om\backslash\overline{B(O,\delta)})$ for any given $\delta>0$, we obtain the desired result.
\end{proof}

\subsection{Equivalent formulation of the electric problem}
For all $\varphi\in\mathring{\mV}^1_{-\beta}(\Om)$,  we have $\nabla\varphi\in\Houtspace$. Using this remark, we deduce that (\ref{SectionEFv1}) with $\om=0$ has a kernel of infinite dimension. To deal with this issue, we work as with $\kappa_\eps$, $\kappa_\mu$ outside of the critical intervals and take into account the divergence free condition. This  leads  us to introduce   the spaces
\begin{equation}\label{DefSpaceChampE}
\begin{array}{rcl}
\mX_N^{\out,\beta}&\hspace{-0.2cm}:=&\hspace{-0.2cm}\{\u\in\Hspace_N^{\out,\beta}(\curl) \, |\, \div(\eps\u)=0\mbox{ in }\Om\backslash\{O\}\},\\[6pt]
\Xoutspace&\hspace{-0.2cm}:=&\hspace{-0.2cm}\{\u\in \Houtspace \, |\, \div(\eps\u)=0\mbox{ in }\Om\backslash\{O\}\}.
\end{array}
\end{equation}
We endow $\mX_N^{\out,\beta}$, $\Xoutspace$ respectively with the norms of $\Hspace_N^{\out,\beta}(\curl)$,  $\Houtspace$.
\begin{remark}\label{SectionERqDivNull}
Let us make two observations concerning the elements $\u\in\mX_N^{\out,\beta}$.\\[3pt]
1) The constraint $\div(\eps\u)=0$ in $\Om\backslash\{O\}$ must be understood as 
\[		
\int_\Om\eps \u\cdot\nabla\overline{\varphi}\,dx=0,\qquad\forall \varphi\in\mathscr{C}^{\infty}_0(\Om\backslash\{O\}).
\]
Therefore, from Lemma \ref{SectionELemmeEquiv}, we infer that if $\u\in\mX_N^{\out,\beta}$, we have
\begin{equation}\label{defDivNulle}
\fint_\Om\eps\u\cdot\nabla \overline{ v}\,dx=0,\qquad\forall v\in\mathring{\mV}^1_\beta(\Om).
\end{equation}
2) If $\u=\Tilde{\u}+\mathfrak{s}_{\u,\eps}\in\mX_N^{\out,\beta}$, then we have 
\begin{equation}\label{defDivNulle2}
\div(\eps\Tilde{\u})=-\div(\eps\nabla\mathfrak{s}_{\u,\eps})\in\mL^2(\Om).
\end{equation}
\end{remark}
By replacing $\Houtspace$, $\Hspace_N^{\out,\beta}(\curl)$ respectively by $\Xoutspace$, $\mX_N^{\out,\beta}$ in \eqref{SectionEFv1}, we get the problem
\begin{equation}\label{SectionEFv2}
\begin{array}{|l}
			\text{Find } \u \in   \Xoutspace \text{ such that } \\[6pt]
			\int_\Om \mu^{-1} \boldsymbol{\psi}_{\u}\cdot \curl\vb\,dx -\om^2\fint_\Om \eps\u\cdot\vb\,dx =i\om\int_\Om \boldsymbol{J} \cdot\vb\,dx,\qquad\forall\v\in  \mX_N^{\out,\beta}.
		\end{array}
\end{equation}
As for (\ref{SectionEFv1}), observe that in (\ref{SectionEFv2}) the solution and the test functions do not belong to the same space. This feature is important to show the well-posedness of the problem and we have not been able to get rid of it (see Remark \ref{RemarkNotSym} for explanations). Besides, we emphasize that in $\Xoutspace$, the space for the solution, both the field and its curl are singular at the origin.
\begin{proposition}\label{SectionEResEquiv}
Assume that $\om\ne0$.\\[3pt]
$\bullet$ Every solution of \eqref{SectionEFv1} solves \eqref{SectionEFv2}.\\[3pt]
$\bullet$ Suppose that Assumptions \ref{AssumptionCritique}-\ref{AssumptionNoTrappedMode} hold. Then every solution of \eqref{SectionEFv2} solves \eqref{SectionEFv1}. 
\end{proposition}
	\begin{proof}
		To show the first part of the statement, one   needs to justify that every solution  $\u$ of \eqref{SectionEFv1} satisfies the equation $\div(\eps \u)=0$ in $\Om\backslash\{O\}$. To proceed, take  $\v=\nabla\varphi$ in  \eqref{SectionEFv1} with $\varphi\in\mathscr{C}^{\infty}_0(\Om\backslash\{O\})$ and use that $\div\boldsymbol{J}=0$ in $\Om\backslash\{O\}$.\\
The proof of the second part is a bit more involved. Assume that $\u\in\Xoutspace$ solves \eqref{SectionEFv2}. Since $\Xoutspace\subset
\Houtspace$, it suffices to show that the variational identity \eqref{SectionEFv2} is also valid for $\v\in\Hspace_N^{\out,\beta}(\curl)$. Consider some $\v=\Tilde{\v}+\nabla\mathfrak{s}_{\v,\eps}\in\Hspace_N^{\out,\beta}(\curl)$ with $\Tilde{\v}\in\mmV^0_{-\beta}(\Om)$ and $\mathfrak{s}_{\v,\eps}\in\mathcal{S}_\eps^+$. The first item of Proposition \ref{AppendixHelmholtzWeighted} guarantees that the function $\Tilde{\v}$  admits the decomposition 
\begin{equation}\label{DecompoPartReg}
\Tilde{\v}=\nabla \varphi+\curl \boldsymbol{\zeta}
\end{equation}
		with $\varphi\in\mathring{\mV}^1_{-\beta}(\Om)$ and $ \boldsymbol{\zeta}\in\mX_T(1)$ such that $ \curl \boldsymbol{\zeta}\in\mmV^0_{-\beta}(\Om)$. Since we have $\div(\eps(\curl \boldsymbol{\zeta}+\nabla\mathfrak{s}_{\v,\eps}))\in(\mathring{\mV}_\beta^1(\Om))^\ast$, Theorem \ref{MainThmScalaire}	 guarantees that there is a unique $\phi\in\mathring{\mV}_\beta^\out(\Om)$ such that 
\[
\langle\mA^\out_\eps \phi,\phi'\rangle=\fint_\Om\eps\nabla \phi\cdot\nabla \overline{\phi'}\,dx=\fint_\Om\eps(\curl \boldsymbol{\zeta}+\nabla\mathfrak{s}_{\v,\eps})\cdot\nabla\overline{\phi'}\,dx,\qquad\forall \phi'\in\mathring{\mV}_\beta^1(\Om).
\]
Now, set $\hat{\v}:=\curl \boldsymbol{\zeta}+\nabla\mathfrak{s}_{\v,\eps}-\nabla \phi=\v-\nabla \varphi-\nabla \phi$. By observing that $\div(\eps \hat{\v})=0$ in $\Om\backslash\{O\}$, we deduce that $\hat{\v}\in\mX_N^{\out,\beta}$.  As a result,  one  can  take $\hat{\v}$ as a test function in  \eqref{SectionEFv2}.
		But, on the other hand, using (\ref{defDivNulle}) and the fact that $\div\,\J=0$ in $\Om\backslash\{O\}$, we obtain
		$$\left\{\begin{array}{rcl}
			\int_\Om \mu^{-1} \boldsymbol{\psi}_{\u}\cdot \curl\vb\,dx&=&\int_\Om \mu^{-1} \boldsymbol{\psi}_{\u}\, \cdot\curl\Bar{\hat{\v}}\,dx\\[10pt]
			\fint_\Om \eps\u\cdot\vb\,dx&=&\fint_\Om \eps\u\cdot\Bar{\hat{\v}}\,dx+\fint_\Om \eps\u\cdot \nabla \overline{ (\varphi+\phi)}\,dx=\fint_\Om \eps\u\cdot\Bar{\hat{\v}}\,dx\\[10pt]
			\int_\Om \boldsymbol{J}\cdot\vb\,dx&=&\int_\Om \boldsymbol{J}\cdot\overline{\hat{\v}}\,dx.
		\end{array}\right.$$
This shows that $\u$ satisfies \eqref{SectionEFv1} and ends the proof of the second item. 
\end{proof}
In the rest of this section, we focus our attention on the study of Problem (\ref{SectionEFv2}). Define the continuous operators $\mathbb{A}^\out_N,\,\mathbb{K}^\out_N:\Xoutspace\to(\mX_N^{\out,\beta})^\ast$ such that for $\u\in\Xoutspace$, $\v\in\mX_N^{\out,\beta}$,
\begin{equation}\label{DefTerm}
\langle\mathbb{A}^\out_N\u,\v\rangle=\int_\Om\mu^{-1}\boldsymbol{\psi}_{\u}\cdot\curl\vb\,dx,\qquad\quad 	\langle\mathbb{K}^\out_N\u,\v\rangle=\fint_\Om\eps\u\cdot \vb\,dx.
\end{equation}
Finally, set $\mathscr{A}^\out_N(\om):=\mathbb{A}^\out_N-\om^2\mathbb{K}^\out_N$ so that for $\u\in\Xoutspace$, $\v\in \mX_N^{\out,\beta}$, 
\begin{equation}\label{DefAoperateur}
\langle\mathscr{A}^\out_N(\om) \u,\v\rangle=\int_\Om\mu^{-1}\boldsymbol{\psi}_{\u}\cdot\curl\vb\,dx-\om^2\fint_\Om\eps\u\cdot \vb\,dx.
\end{equation}
Before getting into details, observe that for $\u,\v\in\Xoutspace$, using relation (\ref{ProprieteHcurl}), we obtain 
\begin{equation}\label{defAExpand}
\langle\mathbb{A}^\out_N\u,\v\rangle=\dsp\int_\Om\mu^{-1}\boldsymbol{\psi}_{\u}\cdot\overline{ \boldsymbol{\psi}_{\v}}\,dx+\int_\Om\div(\mu\nabla\mathfrak{s}_{\u,\mu})\,\overline{ \mathfrak{s}_{\v,\mu}}\,dx.
\end{equation}
On the other hand, exploiting (\ref{defDivNulle}), (\ref{defDivNulle2}), we can write, for $\u\in\Xoutspace$, $\v\in\mX_N^{\out,\beta}$,
\begin{equation}\label{defKExpand}
\langle\mathbb{K}^\out_N\u,\v\rangle=\int_\Om\eps\u\cdot \overline{\tilde{\v}}\,dx=\int_\Om\eps\Tilde{\u}\cdot\overline{\tilde{\v}}\,dx+\int_{\Om}\mathfrak{s}_{\u,\eps}\,\div(\eps\nabla\overline{\mathfrak{s}_{\v,\eps}})\,dx. 
\end{equation}
As a consequence, for $\u,\v\in\Xoutspace$, we obtain 
\begin{equation}\label{ImportantIdentity}
\langle\mathscr{A}^\out_N(\om) \u,\v\rangle-\overline{\langle\mathscr{A}^\out_N(\om) \v,\u\rangle}=-q_\mu(\mathfrak{s}_{\u,\mu},\mathfrak{s}_{\v,\mu})-\om^2q_\eps(\mathfrak{s}_{\u,\eps},\mathfrak{s}_{\v,\eps}).
\end{equation}
This is an important identity. The terms on the right hand side are the ones which are left when integrating twice by parts. For $\v=\u$, they represent the energy which is trapped at the tip. A part of it comes from the singularity of the field, another from the singularity of the curl of the field.

\subsection{Equivalent norms in \texorpdfstring{$\mX_N^{\out,\beta}$ and $\Xoutspace$ }{TEXT}}
	
	The goal  of this section is to introduce new  ``simpler'' equivalent norms in $\mX_N^{\out,\beta}$ and  $\Xoutspace$. First, consider the space  $\mX_N^{\out,\beta}$.
\begin{proposition}\label{SectionEEquivNormes1}
Suppose that Assumptions \ref{AssumptionCritique}-\ref{AssumptionNoTrappedMode} hold.  There is a constant $C>0$ such that
		\begin{equation}
			\|\tilde{\u}\|_{\mmV^0_{-\beta}(\Om)}+\|\nabla \mathfrak{s}_{\u,\eps}\|_{\mmV^0_{\beta}(\Om)}\leq C\,\|\curl \u\|_{\mmV^0_\beta(\Om)},\qquad\forall \u\in \mX_N^{\out,\beta}.
		\end{equation}
Consequently, the norms $\|\cdot\|_{\Hspace_N^{\out,\beta}(\curl)}$ and $\|\curl \cdot\|_{\mmV^0_\beta(\Om)}$ are equivalent in $\mX_N^{\out,\beta}$.
	\end{proposition}
\begin{proof}
	Let  $\u$ be an element of $\mX_N^{\out,\beta}$.  By definition of  $ \mX_N^{\out,\beta}$, we have $\u=\Tilde{\u}+ \nabla  \mathfrak{s}_{\u,\eps}$ with $\Tilde{\u}\in\mmV^0_{-\beta}(\Om)$ and $\mathfrak{s}_{\u,\eps}\in\mathcal{S}_\eps^+$. By means of item $iii)$ of Proposition \ref{AppendixHelmoltzclassique},  one can decompose  $\Tilde{\u}$ as
\begin{equation}\label{DecompouTilde}
\Tilde{\u}=\nabla \varphi+\curl \boldsymbol{\psi}
\end{equation}
with $\varphi\in\mrm{H}^1_{0}(\Om)$ and $\boldsymbol{\psi}\in \mX_T(1)$. Remarking that $\curl (\curl \boldsymbol{\psi})=\curl\u\in\mmV^0_\beta(\Om)$ and that $
	\curl\boldsymbol{\psi}\times\nu=0$ on $\partial\Om\backslash\{O\}$ yields 
\begin{equation}\label{SectionHzbeta}
\curl\boldsymbol{\psi}\in\mZ^\beta_N:=\{\v\in\Lspace^2(\Om)~|~ \curl\v\in \mmV^0_\beta(\Om),~ \div\,\v=0, ~\v\times\nu= 0\mbox{ on } \partial\Om\backslash\{O\}\}.
\end{equation}
Therefore, according to Proposition \ref{AppendixHelweightedRegularity},  we obtain  $\curl\boldsymbol{\psi}\in\mmV^0_{-\beta}(\Om)$ with the estimate
\begin{equation}\label{SectionHrotpsiChampE}
\|\curl\boldsymbol{\psi}\|_{\mmV^0_{-\beta}(\Om)}\leq C\,\|\curl \u\|_{\mmV^0_{\beta}(\Om)}.		
\end{equation}
\noindent Besides, since $\div(\eps\u)=0$ in $\Om\backslash\{O\}$, there holds
\[
\langle\mA^\out_\eps(\mathfrak{s}_{\u,\eps}+\varphi),\phi'\rangle=\fint_\Om\eps\nabla(\mathfrak{s}_{\u,\eps}+\varphi)\cdot\nabla \overline{\phi'}\,dx=-\int_\Om\eps\,\curl\boldsymbol{\psi}\cdot\nabla\overline{\phi'}\,dx,\qquad\forall \phi'\in\mathring{\mV}_\beta^1(\Om).
\]
Exploiting that $\mA^\out_\eps:\mathring{\mV}^\out_\beta(\Om)\to(\mathring{\mV}^1_\beta(\Om))^\ast$ is an isomorphism (Theorem \ref{MainThmScalaire}), we get 
\[
\|\nabla\varphi\|_{\mmV^0_{-\beta}(\Om)}+\|\nabla\mathfrak{s}_{\u,\eps}\|_{\mmV^0_{\beta}(\Om)}\leq C\,\|\curl\boldsymbol{\psi}\|_{\mmV^0_{-\beta}(\Om)}.
\]
By combining (\ref{DecompouTilde}), \eqref{SectionHrotpsiChampE} and the previous estimate, we obtain the desired result.
\end{proof}		
	Now, we turn our attention to $\Xoutspace$. 
	\begin{proposition}\label{SectionEEquivNormes2}
		Suppose that  Assumptions \ref{AssumptionCritique}-\ref{AssumptionNoTrappedMode} hold. There is a constant $C>0$ such that
		\begin{equation}
			\|\tilde{\u}\|_{\mmV^0_{-\beta}(\Om)}+\|\nabla \mathfrak{s}_{\u,\eps}\|_{\mmV^0_{\beta}(\Om)}\leq C\, \|\boldsymbol{\psi}_{ \u}\|_{\mmV^0_{-\beta}(\Om)},\quad\forall \u\in \Xoutspace\mbox{ with }\curl \u=\boldsymbol{\psi}_{\u}+\mu\nabla\mathfrak{s}_{\u,\mu}.
		\end{equation}
		Consequently, in $\Xoutspace$ the map $\u\mapsto\|\boldsymbol{\psi}_{ \u}\|_{\mmV^0_{-\beta}(\Om)}$ is a norm which is equivalent to $\|\cdot\|_{\Houtspace}$.
	\end{proposition}
	\begin{proof}
	Since there holds $\Xoutspace\subset\mX_N^{\out,\beta}$, from Proposition \ref{SectionEEquivNormes1} we see that it is enough to show that we have 
\begin{equation}\label{EstimaToSatisfy}
\|\curl \u\|_{\mmV^0_\beta(\Om)} \leq C\,\|\boldsymbol{\psi}_{ \u}\|_{\mmV^0_{-\beta}(\Om)},\qquad \forall\u\in \Xoutspace.
\end{equation}
Consider some $\u\in\Xoutspace$.  We have 
\[
\curl \u=\boldsymbol{\psi}_{\u}+\mu\nabla\mathfrak{s}_{\u,\mu}=\boldsymbol{\psi}_{\u}+\mu\nabla(\mathfrak{s}_{\u,\mu}-c_{\u})\qquad\mbox{ with }\qquad c_{\u}:=\frac{1}{|\Om|}\int_\Om\mathfrak{s}_{\u,\mu}.
\]
The relation $\u\times\nu=0$ on $\partial\Om\backslash\{O\}$ implies  $\curl\u\cdot\nu=0$ on 	$\partial\Om\backslash\{O\}$. Furthermore, given that  $\div(\curl \u)=0,$ we deduce that $\mathfrak{s}_{\u,\mu}$ and $\psib_{\u}$ satisfy 
\[
\langle\mA^\out_\mu (\mathfrak{s}_{\u,\mu}-c_{\u}),\phi'\rangle=\fint_\Om\mu\nabla (\mathfrak{s}_{\u,\mu}-c_{\u})\cdot\nabla \overline{\phi'}\,dx=-\int_{\Om}\boldsymbol{\psi}_{\u}\cdot\nabla \overline{\phi'}\,dx,\qquad\forall \phi'\in\mathcal{V}^1_{\beta}(\Om).
\]
Note that $\mathfrak{s}_{\u,\mu}-c_{\u}$ is indeed an element of $\mathcal{V}^\out_\beta(\Om)$ because $1\in\mV^1_{-\gamma}(\Om)$ for all $\gamma<1/2$ and there holds $\beta<1/2$. Using that $\mA^\out_\mu:\mathcal{V}^\out_\beta(\Om)\to (\mathcal{V}^1_{\beta}(\Om))^\ast$ is an isomorphism (Theorem \ref{MainThmScalaire}), we obtain 
\[
|c_{\u}|+\|\mathfrak{s}_{\u,\mu}\|_{\mV^1_\beta(\Om)}\leq C\,\|\boldsymbol{\psi}_{ \u}\|_{\mmV^0_{-\beta}(\Om)}
\]
where here and below $C$ is independent of $\u$. This gives 
$$
\|\nabla\mathfrak{s}_{\u,\mu}\|_{\mmV^0_{\beta}(\Om)}\leq C\,\|\boldsymbol{\psi}_{ \u}\|_{\mmV^0_{-\beta}(\Om)}.
$$
and inserting this estimate into  
\[
\|\curl \u\|_{\mmV^0_\beta(\Om)}\leq C (\|\nabla\mathfrak{s}_{\u,\mu}\|_{\mmV^0_\beta(\Om)}+ \|\boldsymbol{\psi}_{\u}\|_{\mmV^0_\beta(\Om)})\leq C (\|\nabla\mathfrak{s}_{\u,\mu}\|_{\mmV^0_\beta(\Om)}+ \|\boldsymbol{\psi}_{\u}\|_{\mmV^0_{-\beta}(\Om)})
\]
finally leads to (\ref{EstimaToSatisfy}). 
\end{proof}

	\subsection{Analysis of the principal part }\label{paragraphePrincPart}
	In this section, we study the operator $\mathbb{A}^\out_N=\mathscr{A}^\out_N(0)$ defined in (\ref{DefTerm}). We emphasize that the operator $\mathbb{T}$ that we construct in the next proposition constitutes the main ingredient in the analysis of the electric problem.
	\begin{proposition}\label{propoOntoE}
		\label{SectionEResPrincipalPart} Suppose that Assumptions \ref{AssumptionCritique}-\ref{AssumptionNoTrappedMode} hold. There exists a continuous operator $\mathbb{T}:\mX_N^{\out,\beta}\to\Xoutspace$ such that
\begin{equation}\label{variaIdentity}
\langle\mathbb{A}^\out_N \circ \mathbb{T}\u,\v\rangle=\int_\Om r^{2\beta}\curl \u \cdot\curl \vb\,dx,\qquad \forall\u,\v \in\mX_N^{\out,\beta}.
\end{equation}
\end{proposition}
\begin{remark}\label{RemarkNotSym}
Note that considering a variational formulation (\ref{SectionEFv2}) with test functions in $\Xoutspace\subset\mX_N^{\out,\beta}$ would lead us to work with an operator $\tilde{\mathbb{A}}^\out_N:\Xoutspace\to(\Xoutspace)^\ast$. Then to prove invertibility of $\tilde{\mathbb{A}}^\out_N$ would require to construct some continuous operator $\tilde{\mathbb{T}}:\Xoutspace\to\Xoutspace$ such that $\langle\tilde{\mathbb{A}}^\out_N \circ \tilde{\mathbb{T}}\,\cdot,\cdot\rangle$ is coercive in $\Xoutspace$. But the norm of $\Xoutspace$ is more restrictive than the one of $\mX_N^{\out,\beta}$ and we do not know if we can find such a $\tilde{\mathbb{T}}$. 
\end{remark}
\begin{proof}
Pick some $\u\in\mX_N^{\out,\beta}$. 	Let us work in three steps.  \\[2pt]
		{\bf Step 1.} The operator $\mA^\out_\mu:\mathcal{V}^\out_\beta(\Om)\to (\mathcal{V}^1_{\beta}(\Om))^\ast$ is an isomorphism. As a result, there is a unique $\varphi=\tilde{\varphi}+\mathfrak{s}^+_\mu\in\mathcal{V}_\beta^\out(\Om)$, $\tilde{\varphi}\in\mrm{V}^1_{-\beta}(\Om)$, $\mathfrak{s}^+_\mu\in\mathcal{S}_\mu^+$, such that 
\[
\langle\mA^\out_\mu \varphi,\varphi'\rangle=\fint_\Om\mu\nabla \varphi\cdot\nabla \overline{\varphi'}\,dx=-\int_\Om\mu\, r^{2\beta}\curl \u\cdot\nabla\overline{\phi'}\,dx,\qquad\forall \phi'\in\mathcal{V}^1_{\beta}(\Om).
\]
Additionally, we have the estimate
		\begin{equation}\label{SectionEProofT-coestmiTcoE}
			\|\tilde{\varphi}\|_{\mrm{V}^1_{-\beta}(\Om)}+\|\nabla\mathfrak{s}_\mu^+\|_{\mmV^0_\beta(\Om)}\leq C\,\|r^{2\beta}\curl \u\|_{\mmV^0_{-\beta}(\Om)}= C\,\|\curl \u\|_{\mmV^0_{\beta}(\Om)}.
		\end{equation}
		{\bf Step 2.} Define the function  $\boldsymbol{F}:=\mu(\nabla \varphi+r^{2\beta}\curl \u)\in\mmV^0_\beta(\Om)$. There holds
$$ \begin{array}{|rcll}
			\div\,\boldsymbol{F}&=&0& \mbox{in } \Om\\[3pt]
			\boldsymbol{F}\cdot \nu&=&0 &\mbox{on } \partial\Om.
		\end{array}$$
		Since $\beta\in(0;1/2)$, Proposition  \ref{AppendixWeightedPotenials} guarantees that there exists a unique $\boldsymbol{\zeta}\in\mZ_N^\beta$ (see \eqref{SectionHzbeta}) for the definition of $\mZ_N^\beta$) such that
		$$\curl\boldsymbol{\zeta}=\boldsymbol{F}=\mu(\nabla \varphi+r^{2\beta}\curl \u).$$
		Furthermore, Proposition  \ref{AppendixHelweightedRegularity} ensures that $\boldsymbol{\zeta}\in\mmV^0_{-\beta}(\Om)$.\\

		{\bf Step 3.} Since $\mA^\out_\eps:\mathring{\mV}^\out_\beta(\Om)\to(\mathring{\mV}^1_\beta(\Om))^\ast$ is an isomorphism, there is a unique $\phi=\tilde{\phi}+\mathfrak{s}_\eps^+\in\mathring{\mV}_\beta^\out(\Om)$, $\tilde{\phi}\in\mathring{\mV}^1_{-\beta}(\Om)$, $\mathfrak{s}_\eps^+\in\mathcal{S}_\eps^+$, such that 		
\[
\langle\mA^\out_\eps \phi,\phi'\rangle=\fint_\Om\eps\nabla \phi\cdot\nabla \overline{\phi'}\,dx=\int_\Om\eps\,\boldsymbol{\zeta}\cdot\nabla\overline{\phi'}\,dx,\qquad\forall \phi'\in\mathring{\mV}_\beta^1(\Om).
\]		
Finally, we set  $\mathbb{T}\u=\boldsymbol{\zeta}-\nabla \phi$.
		One can check that $\mathbb{T}\u$ belongs to $\Xoutspace$. In addition, we have $\curl(\mathbb{T}\u)=\boldsymbol{\psi}_{\mathbb{T}(\u)}+\mu\nabla \mathfrak{s}_{\mathbb{T}(\u),\mu}$ with 
\begin{equation}\label{DefTcoerExpan}
			\boldsymbol{\psi}_{\mathbb{T}(\u)}=\mu(\nabla \Tilde{\varphi}+r^{2\beta}\curl \u),\qquad \qquad			\mathfrak{s}_{\mathbb{T}(\u),\mu}=\mathfrak{s}^+_\mu.
\end{equation}
		Using \eqref{SectionEProofT-coestmiTcoE}, we obtain
		$$  \|\boldsymbol{\psi}_{\mathbb{T}(\u)}\|_{\mmV^0_{-\beta}(\Om)}\leq  C\,\|\curl \u\|_{\mmV^0_{\beta}(\Om)}.$$
With Proposition \ref{SectionEEquivNormes2}, this shows that $\mathbb{T}:\mX_N^{\out,\beta}\to\Xoutspace$ is continuous.  Besides, using (\ref{DefTerm}), (\ref{DefTcoerExpan}), we find
		$$\langle\mathbb{A}^\out_N \circ \mathbb{T}\u,\v\rangle=\int_\Om \mu^{-1}\boldsymbol{\psi}_{\mathbb{T}\u}\cdot\curl\vb\,dx= \int_\Om r^{2\beta}\curl \u\cdot \,\curl \vb\,dx,\qquad\forall \v \in\mX_N^{\out,\beta},$$
which is nothing but the desired identity (\ref{variaIdentity}).
	\end{proof}
	\begin{theorem}\label{SectionEPrinicpalPartIsom} 
	Suppose that Assumptions \ref{AssumptionCritique}-\ref{AssumptionNoTrappedMode} hold. Then the operator $\mathbb{A}^\out_N:\Xoutspace\to(\mX_N^{\out,\beta})^\ast$ is an isomorphism.
	\end{theorem}
	\begin{proof}
Let $\mathbb{T}$ be the operator of Proposition \ref{propoOntoE}. From identity (\ref{variaIdentity}) together with the Lax-Milgram theorem and the result of Proposition \ref{SectionEEquivNormes1}, we infer that $\mathbb{A}^\out_N\circ \mathbb{T}: \mX_N^{\out,\beta}\to(\mX_N^{\out,\beta})^\ast $  is an  isomorphism. This guarantees that $\mathbb{A}^\out_N$ is onto. It remains to show that $\mathbb{A}^\out_N$ is injective. Let $\u$ be an element  of  $\Xoutspace$ such that $\langle\mathbb{A}^\out_N\u,\v\rangle=0$ for all $\v\in\mX_N^{\out,\beta}$. By definition of $\Xoutspace$, we have $\curl\u=\boldsymbol{\psi}_{\u}+\mu\nabla \mathfrak{s}_{\u,\mu}$. Taking $\v=\u$ in the relation \eqref{ImportantIdentity} with $\om=0$ and using Lemma \ref{FluxScalairDefPosi}, we infer that $\mathfrak{s}_{\u,\mu}=0$. From (\ref{defAExpand}), this implies
		$$0=\langle\mathbb{A}^\out_N\u,\v\rangle=\int_\Om\mu^{-1}\boldsymbol{\psi}_{\u}\cdot\overline{\boldsymbol{\psi}_{\v}}\,dx=\overline{\langle\mathbb{A}^\out_N\v,\u\rangle},\qquad\forall\v\in\Xoutspace. $$
By taking $\v=\mathbb{T}\u$ in the previous relation, according to (\ref{variaIdentity}), we obtain $\curl\u=0$ and so $\boldsymbol{\psi}_{\u}=0$. From Proposition \ref{SectionEEquivNormes2},  we deduce that $\u\equiv0$.
	\end{proof}
	\subsection{Compactness result}\label{paragraphCompactness}
Now we focus our attention on the operator $\mathbb{K}^\out_N$ defined in (\ref{DefTerm}).
\begin{theorem}\label{SectionERescompcaite}
Suppose that Assumptions \ref{AssumptionCritique}-\ref{AssumptionNoTrappedMode} hold. Then $\mathbb{K}^\out_N:\Xoutspace\to(\mX_N^{\out,\beta})^\ast$ is compact.
\end{theorem}
\begin{proof}
Using (\ref{defKExpand}), we obtain the estimate
\begin{equation}\label{SectionEestimatecompact}
		\|\mathbb{K}^\out_N\u\|_{(\mX_N^{\out,\beta})^\ast }\leq C(\|\Tilde{\u}\|_{\mmV^0_{-\beta}(\Om)}+\|\mathfrak{s}_{\u,\eps}\|_{\mmV^1_{\beta}(\Om)},\qquad \forall\u\in\Xoutspace,
\end{equation}
where $C>0$ is independent of $\u$. Now consider $(\u_n)_{n}$ a bounded sequence of elements of $\Xoutspace$. By definition of $\Xoutspace$, one can introduce, for all $n\in\mathbb{N}$,  $\Tilde{\u}_n\in\mmV^0_{-\beta}(\Om)$ and $\mathfrak{s}_n\in\mathcal{S}_{\eps}^+$ such that
		$ \u_n=\Tilde{\u}_n+\mathfrak{s}_n$.
		The sequences $(\mathfrak{s}_n)_{n}$, $(\Tilde{\u}_n)_{n}$ are bounded respectively in $\mmV^1_{\beta}(\Om)$ and in $\mmV^0_{-\beta}(\Om)$.
		Since $\mathcal{S}_\eps^+$ is of finite dimension,  one can extract a sub-sequence from $(\mathfrak{s}_n)_{n}$, still denoted $(\mathfrak{s}_n)_{n}$, that converges  in $\mathcal{S}_{\eps}^+$. Therefore, to prove our claim, thanks to \eqref{SectionEestimatecompact},    it is enough to show that up to a subsequence, $(\tilde{\u}_n)_{n}$ converges in $\mmV^0_{-\beta}(\Om)$. Owing to Proposition \ref{AppendixHelmholtzWeighted},  for all $n\in\N,$  the vector field $\Tilde{\u}_n$ decomposes as
		$$\Tilde{\u}_n=\nabla\varphi_n+\curl \boldsymbol{\psi}_n,$$
		where $\varphi_n\in\mathring{\mV}^1_{-\beta}(\Om)$ and $\boldsymbol{\psi}_n\in\mX_T(1)$  is such that $\boldsymbol{\psi}_n\in\mmV^0_{-\beta}(\Om)$ (Proposition \ref{AppendixWeightedRegularityX(1)}). Given that
		$\nabla\varphi_n\times \nu=0$ on $\partial\Om$ and $\curl \curl\boldsymbol{\psi}_n=\curl\u_n\in\mmV^0_\beta(\Om),$ we infer that   $\curl\boldsymbol{\psi}_n\in\mZ_N^\beta$ (see (\ref{SectionHzbeta})). Since by assumption $\beta\in(0;1/2)$,  one  deduces, using Proposition \ref{AppendixHelweightedRegularity}, that $(\curl \boldsymbol{\psi_n})_{n}$ converges, up to a sub-sequence still denoted  $(\curl \boldsymbol{\psi_n})_{n}$, in $\mmV^0_{-\beta}(\Om)$.
Besides, we know that $\div(\eps\u_n)=0$ in $\Om\backslash\{O\}$. This implies 
\[
\langle\mA^\out_\eps \varphi_n,\varphi'\rangle=\fint_\Om\eps\nabla \varphi_n\cdot\nabla \overline{\varphi'}\,dx=-\fint_\Om\eps(\nabla \mathfrak{s}_n+\curl\boldsymbol{\psi}_n)\cdot\nabla\overline{\varphi'}\,dx,\qquad\forall \varphi'\in\mathring{\mV}_\beta^1(\Om).
\]
By noting that the functional appearing in the right hand side above converges in $(\mathring{\mV}^1_{\beta}(\Om))^\ast$ and by using that $\mA^\out_\eps$ is an isomorphism,  we deduce that   $(\varphi_n)_{n}$ converges in $\mathring{\mV}^1_{-\beta}(\Om)$. This gives the desired result: $(\tilde{\u}_n)_{n}$ converges, up to a sub-sequence, in $\mmV^0_{-\beta}(\Om)$.
	\end{proof}
	\subsection{Main result for the electric problem}\label{SectionEMainResSection}
	
Finally, we come to the properties of the operator  $\mathscr{A}^\out_N(\om)=\mathbb{A}^\out_N-\om^2\mathbb{K}^\out_N$ introduced in (\ref{DefAoperateur}). By  combining Theorems \ref{SectionEPrinicpalPartIsom}, \ref{SectionERescompcaite} and the analytical Fredholm  theorem (note that $\mathscr{A}^\out_N(0)=\mathbb{A}^\out_N$ is injective according to Theorem \ref{SectionEPrinicpalPartIsom}), we obtain the following result.
	\begin{theorem}\label{SectionEMainRes}
Suppose that  Assumptions \ref{AssumptionCritique}-\ref{AssumptionNoTrappedMode} hold.  \\[3pt]
$\bullet$ The operator   $\mathscr{A}^\out_N(\om)$ is Fredholm of index zero for all $\om\in\R$. \\[2pt]
$\bullet$ $\mathscr{A}^\out_N(\om)$ is an isomorphism for all $\om\in\R\backslash\mathscr{S}_{\E}$  where $\mathscr{S}_{\E}$ is a discrete set which can accumulate only at infinity. 
	\end{theorem}
The rest of this paragraph is devoted to the study of $\mathscr{S}_{\E}$, the set of values of $\om\in\R$ such that $\mathscr{A}^\out_N(\om)$ is not injective. 
\begin{proposition}\label{PropoKernelPoids}
Suppose that Assumptions \ref{AssumptionCritique}-\ref{AssumptionNoTrappedMode}  hold. The set $\mathscr{S}_{\E}$ appearing in Theorem \ref{SectionEMainRes} is independent of the choice of $\beta$ satisfying (\ref{ChoixPoids}).	
\end{proposition}
\begin{proof} 
 Let $\u$ be an element of $\ker\mathscr{A}^\out_N(\om)$ for some $\beta$ satisfying (\ref{ChoixPoids}). Then we have $\langle \mathscr{A}^\out_N(\om) \u,\u\rangle-\overline{\langle\mathscr{A}^\out_N(\om) \u,\u\rangle}=0$, which according to (\ref{ImportantIdentity}) implies
\[
q_\mu(\mathfrak{s}_{\u,\mu},\mathfrak{s}_{\u,\mu})+\om^2 q_\eps(\mathfrak{s}_{\u,\eps},\mathfrak{s}_{\u,\eps})=0.
\]
Given that $\mathfrak{s}_{\u,\eps}\in\mathcal{S}_\eps^+$ and   $\mathfrak{s}_{\u,\mu}\in\mathcal{S}_\mu^+,$ we infer that $q_\eps(\mathfrak{s}_{\u,\eps},\mathfrak{s}_{\u,\eps})= q_\mu(\mathfrak{s}_{\u,\mu},\mathfrak{s}_{\u,\mu})=0$. Then Lemma \ref{FluxScalairDefPosi} guarantees that $\mathfrak{s}_{\u,\eps}=\mathfrak{s}_{\u,\mu}=0$. This shows that $\u$ belongs to $\mX_N(\eps)$ and satisfies the equation $\curl (\mu^{-1}\curl\u)=\om^2\eps\u$. We deduce that $\mu^{-1}\curl\u\in\mX_T(\mu)$. Owing to Proposition \ref{AppendixX(eps)} this implies that $\u,\curl\u\in\mmV^0_{-\beta}(\Om)$ for all $\beta$ satisfying (\ref{ChoixPoids}). Thus $\u\in\ker\mathscr{A}^\out_N(\om)$ for all $\beta$ satisfying (\ref{ChoixPoids}).
	\end{proof}
In the same vein, we have the following result.
	\begin{proposition}\label{SectionEIndepedencebeta}
		Suppose that Assumptions \ref{AssumptionCritique}-\ref{AssumptionNoTrappedMode} hold and let $\om\in\R\backslash \mathscr{S}_{\E}$. The solution of \eqref{SectionEFv2} is independent of the choice of $\beta$ satisfying (\ref{ChoixPoids}). 
\end{proposition}
\begin{proof}
Let $\beta_1<\beta_2$ satisfying (\ref{ChoixPoids}). For $j=1,2$, denote by 
\begin{center}
\begin{tabular}{|l}
$\u_j$ the solution of \eqref{SectionEFv2} \\[2pt]
$\mathscr{A}^\out_{N,j}(\om)$ the operator defined in (\ref{DefAoperateur})\\[2pt]
${\pmb{\mathscr{X}}_{N,j}^{\hspace{0.2mm}\out}}$ the space introduced in (\ref{DefSpaceChampE})
\end{tabular}
\end{center}
for $\beta=\beta_j$. Since $\beta_1<\beta_2$, we  have ${\pmb{\mathscr{X}}_{N,2}^{\hspace{0.2mm}\out}}\subset{\pmb{\mathscr{X}}_{N,1}^{\hspace{0.2mm}\out}}$. As a consequence, the function $\e:=\u_1-\u_2$ belongs to ${\pmb{\mathscr{X}}_{N,1}^{\hspace{0.2mm}\out}}$.\\[4pt] 
$\star$ When $\om\ne0$, by exploiting Propositions \ref{SectionEResEquivDis} and \ref{SectionEResEquiv}, we find that both $\u_1$ and $\u_2$ satisfy the equations of (\ref{SectionEEqE}) in the sense of distributions of $\Om\backslash\{O\}$. This implies that $\e$ is an element of $\ker\mathscr{A}^\out_{N,1}(\om)$.  Since $\om\in\R\backslash \mathscr{S}_{\E}$, we obtain $\e\equiv0$ and so $\u_1\equiv\u_2$.\\[4pt]
$\star$ When $\om=0$, we work as in the end of the proof of Theorem \ref{SectionEPrinicpalPartIsom}. First, by observing that $\langle\mathscr{A}^\out_{N,1}(0)\e,\e\rangle=0$, we get, thanks to \eqref{ImportantIdentity}, 
\[
0=\langle\mathscr{A}^\out_{N,1}(0) \e,\e\rangle-\overline{\langle\mathscr{A}^\out_{N,1}(0) \e,\e\rangle}=-q_\mu(\mathfrak{s}_{\e,\mu},\mathfrak{s}_{\e,\mu}).
\]
From Lemma \ref{FluxScalairDefPosi}, this gives $\mathfrak{s}_{\e,\mu}\equiv0$ and so 
\[
\langle\mathscr{A}^\out_{N,1}(0)\e,\v\rangle=\int_\Om\mu^{-1}\boldsymbol{\psi}_{\e}\cdot\overline{\boldsymbol{\psi}_{\v}}\,dx=0,\qquad\forall\v\in{\pmb{\mathscr{X}}_{N,1}^{\hspace{0.2mm}\out}}.
\]
By taking $\v=\mathbb{T}\e$ in the previous relation, where $\mathbb{T}$ is the operator of Proposition \ref{propoOntoE} constructed with $\beta=\beta_1$, we obtain $\boldsymbol{\psi}_{\e}=0$. From Proposition \ref{SectionEEquivNormes2},  we deduce that $\e\equiv0$ and so $\u_1\equiv\u_2$.
\end{proof}
	Now, by gathering Proposition \ref{SectionEResEquiv} and Theorem \ref{SectionEMainRes}, we can state the main result of the section.
\begin{theorem}\label{SectionEResFinal}
Suppose that Assumptions \ref{AssumptionCritique}-\ref{AssumptionNoTrappedMode} hold, $\J\in\mmV^0_{-\eta}(\Om)$ for some $\eta>0$ and $\div\,\J=0$ in $\Om$.\\[3pt]
$\bullet$ For all $\om\in\R\backslash \mathscr{S}_{\E}$, where $\mathscr{S}_{\E}$ appears in Theorem \ref{SectionEMainRes}, the problem \eqref{SectionEFv2} (or equivalently  \eqref{SectionEFv1} when $\om\ne0$) admits a unique solution.\\[3pt]
$\bullet$ When $\om\in \mathscr{S}_{\E}$  the problem \eqref{SectionEFv2} (or equivalently  \eqref{SectionEFv1})  is well-posed in the Fredholm sense. Moreover, it has a kernel of finite dimension that is independent of $\beta$ satisfying (\ref{ChoixPoids}).
	\end{theorem}
\begin{remark}
Above we could remove the assumption $\div\,\J=0$ in $\Om\backslash\{O\}$ and simply suppose that $\J\in\mmV^0_{-\eta}(\Om)$. In this situation, to study the electric problem, the first step would be to introduce some potential $p\in\mathring{\mV}^\out_{\beta}(\Om)$, $\beta\in(0;\min(\beta_D,\eta))$, such that $\mA^\out_\eps p=-(i\om)^{-1}\div\,\J\in(\mathring{\mV}^1_\beta(\Om))^\ast$, and then to work with $\E_0:=\E-\nabla p$.
\end{remark}
	
Up to now, we have shown that given spaces of outgoing propagating singularities $\mathcal{S}_\eps^+$ and $\mathcal{S}^+_\mu$ lead to a functional framework in which the electric problem is well posed. On the other hand, there is an infinite number of choices for $\mathcal{S}_\eps^+$, $\mathcal{S}^+_\mu$ (see Remark \ref{RqChoixdeLaBase}) and they all provide different functional frameworks for the Maxwell's problem. The goal of the next section is to explain how to identify the one that is coherent with the limiting absorption principle.
\subsection{Limiting absorption principle}\label{SectionELimiting}

Let $\J$ be as in Theorem \ref{SectionEResFinal}. To model the dissipation of the materials, introduce, for $\delta>0$,  the functions $\eps_\delta:=\eps+i\delta$ and $\mu_\delta:=\mu+i\delta$. Denote by $\mA_{\eps_\delta}:\mH^1_0(\Om)\to (\mH^1_0(\Om))^\ast$, $\mA_{\mu_\delta}:\mH^1_\#(\Om)\to(\mH^1_\#(\Om))^\ast$ the operators defined as $\mA_\eps$, $\mA_\mu$ in (\ref{Def_Aeps}), (\ref{Def_Amu}) with $\eps$, $\mu$ replaced by $\eps_\delta$, $\mu_\delta$. Since the imaginary part of $\eps_\delta$, $\mu_\delta$ is positive in $\Om$, the Lax-Milgram theorem ensures that $\mA_{\eps_\delta}$, $\mA_{\mu_\delta}$ are isomorphisms for all $\delta>0$. Moreover, it guarantees that for all $\om\in\R$, the problem
\begin{equation}\label{SectionEMaxwelldissipation}
\begin{array}{|l}
\mbox{Find }\u_\delta \in  \mX_N(\eps_\delta)~\text{such that}\\[3pt]
\curl(\mu_\delta^{-1}\curl \u_\delta)-\om^2\eps_\delta\u_\delta=i\om\J\quad\mbox{ in }\Om
\end{array}		
\end{equation} 
admits a unique solution. Roughly speaking, our goal is to show that if the spaces $\mathcal{S}_\eps^+$, $\mathcal{S}_\mu^+$ in (\ref{DefBases}) have been chosen such that the limiting absorption principle holds for the scalar problems, then $(\u_\delta)_\delta$ converges to the solution of \eqref{SectionEFv2}, in other words that the limiting absorption principle is valid for the electric problem written in the framework \eqref{SectionEFv2}. Let us make this more precise.

\begin{Assumption}\label{Assumptionlimiting}
Suppose that  Assumptions \ref{AssumptionCritique}-\ref{AssumptionNoTrappedMode} hold. Assume that $\mathcal{S}_\eps^+$, $\mathcal{S}_\mu^+$ are such that\\[6pt]
$\bullet$ if $(f_\delta)_{\delta>0}$ converges to $f$ in $(\mathring{\mV}^1_\beta(\Om))^\ast\subset(\mH_0^1(\Om))^\ast$, then $\lim_{\delta\to0^+}\|(\mA_{\eps_\delta})^{-1}f_\delta-(\mA^\out_\eps)^{-1}f\|_{\mV^1_{\beta}(\Om)}=0$;\\[6pt]
$\bullet$ if $(f_\delta)_{\delta>0}$ converges to $f$ in $(\mathcal{V}^1_\beta(\Om))^\ast\subset(\mH_\#^1(\Om))^\ast$, then $\lim_{\delta\to0^+}\|(\mA_{\mu_\delta})^{-1}f_\delta-(\mA^\out_\mu)^{-1}f\|_{\mV^1_{\beta}(\Om)}=0$.
\end{Assumption}

Note that the previous assumption requires that the frameworks obtained for the scalar problems via the limiting absorption principle satisfies the Mandelstam radiation principle. Generally speaking, this seems to happen most of the times but may be wrong in certain rare circumstances. In the articles \cite{nazarov2020anomalies,nazarov2014umov}, the author gives examples of problems involving elliptic PDEs in unbounded domains for which the two principles contradict each other at the so-called cut-off frequencies.  In our case, it can be shown (see \cite[Chapter 2]{rihani2022maxwell}) that the validity of the previous assumption depends on the behaviour of the spectrum, eigenfunctions and generalized eigenfunctions of $\mathscr{L}_{\eps_\delta}$ , $\mathscr{L}_{\mu_\delta}$, the Mellin symbols defined as $\mathscr{L}_{\eps}$ , $\mathscr{L}_{\mu}$ in (\ref{DefLeps}), (\ref{DefLmu}) with $\eps$, $\mu$ replaced by $\eps_\delta$, $\mu_\delta$. For the particular case of the circular conical tip (\ref{ConicalTip}), it is proved in \cite[Chapter 2]{rihani2022maxwell} that Assumption \ref{Assumptionlimiting} is valid except for a discrete set of contrasts.\\
\newline	
The main result of this section is given by the following theorem which justifies the physical relevance of Problem (\ref{SectionEEqE}).
\begin{theorem}\label{SectionELimitingMainTheorem}
Suppose that Assumptions \ref{AssumptionCritique}-\ref{AssumptionNoTrappedMode}-\ref{Assumptionlimiting} hold, $\J$ is as in Theorem \ref{SectionEResFinal} and $\om\in\R\backslash \mathscr{S}_{\E}$, where $\mathscr{S}_{\E}$ is defined in Theorem \ref{SectionEMainRes}. We have
\[
\lim_{\delta\to0^+}\|\u_\delta-\u\|_{\mmV^0_\beta(\Om)}+\|\curl\u_\delta-\curl\u\|_{\mmV^0_\beta(\Om)} = 0
\]
where $\u_\delta$, $\u$ denote respectively the solutions of (\ref{SectionEMaxwelldissipation}), (\ref{SectionEEqE}).
\end{theorem}
	The proof of the previous theorem is mainly  based on the following  proposition.
	\begin{proposition}
Suppose that Assumptions \ref{AssumptionCritique}-\ref{AssumptionNoTrappedMode}-\ref{Assumptionlimiting}  hold. Let $(\u_\delta)_{\delta}$ be a sequence of elements of $\mX_N(\eps_\delta)$ (resp. $\mX_T(\mu_\delta)$) such that  $(\curl\u_\delta)_{\delta}$ is bounded in $\mmV^0_{\beta}(\Om)$. Then, up to a sub-sequence, $(\u_\delta)_{\delta}$ converges in $\mmV^0_{\beta}(\Om)$ to an element of $\nabla\mathcal{S}^+_\eps\oplus \mmV^0_{-\beta}(\Om)$ (resp. $\nabla\mathcal{S}^+_\mu\oplus \mmV^0_{-\beta}(\Om)$) as $\delta \to0^+$.
	\end{proposition}
\begin{proof}
Let $(\u_\delta)_{\delta>0}$ be a sequence of elements of $\mX_N(\eps_\delta)$ such that $(\curl\u_\delta)_{\delta}$ is bounded in $\mmV^0_{\beta}(\Om)$. Owing to item $iii)$ of Proposition \ref{AppendixHelmoltzclassique}, for all $\delta>0$, we have the decomposition 
\begin{equation}\label{decompoAnnexe}
\u_\delta=\nabla\varphi_\delta+\curl\psib_\delta
\end{equation}
with $\varphi_\delta\in\mH^1_0(\Om)$ and $\psib_\delta\in\mX_T(1)$. Additionally, there holds $\curl\psib_\delta\in\mZ^\beta_N$ (see (\ref{SectionHzbeta}) for the definition of this space). But  Proposition \ref{AppendixHelweightedRegularity} ensures that in $\mZ^\beta_N$, $\|\curl\cdot\|_{\mmV^0_{\beta}(\Om)}$ is a norm which is equivalent to the natural one, and that $\mZ^\beta_N$ is compactly embedded in $\mmV^0_{-\beta}(\Om)$. From this, we infer that, up to a subsequence, $(\curl\psib_\delta)_\delta$ converges in $\mmV^0_{-\beta}(\Om)$ as $\delta\to0^+$.\\
Since we have $\div(\eps_\delta\,\u_\delta)=0$ in $\Om$, we deduce that there holds $\mA_{\eps_\delta}\varphi_\delta=-\div(\eps_\delta\,\curl\psib_\delta)$. On the other hand, the fact that $\curl\psib_\delta\in\mmV^0_{-\beta}(\Om)$ implies that $\div(\eps_\delta\,\curl\psib_\delta)\in(\mathring{\mV}^1_\beta(\Om))^\ast$. Furthermore, since $(\curl\psib_\delta)_\delta$ converges in $\mmV^0_{-\beta}(\Om)$, it follows that $(\div(\eps_\delta\,\curl\psib_\delta))_{\delta}$ converges in $(\mathring{\mV}^1_\beta(\Om))^\ast$ as $\delta\to0^+$. As a consequence,  under Assumption \ref{Assumptionlimiting}, we infer that up to a sub-sequence, $(\varphi_\delta)_\delta$ converges in $\mathring{\mV}^1_\beta(\Om)$ to an element of $\mathring{\mV}^\out_\beta(\Om)$. With (\ref{decompoAnnexe}), this gives the desired result. The proof for a bounded sequence of elements of $\mX_T(\mu_\delta)$ is similar. 
\end{proof}
	\begin{proof}[\textit{Proof of Theorem \ref{SectionELimitingMainTheorem}}]
	Let  $(\delta_n)_{n}$ be a sequence of positive numbers that converges to zero as $n\to+\infty$. Denote $\eps_n=\eps+i\delta_n,\mu_n=\mu+i\delta_n$ for all $n\in\N$. Denote by $\u_n$ the solution to \eqref{SectionEMaxwelldissipation} for $\delta=\delta_n$. Let us proceed in two steps. First, we establish the desired result by assuming that $(\curl\u_n)_n$ is bounded in $\mmV^0_\beta(\Om)$. Then we show that this hypothesis is indeed satisfied.\\[5pt]
			\textbf{Step 1.} Assume that  $(\curl \u_n)_{n}$ is bounded in $\mmV^0_\beta(\Om)$. According to the previous proposition, we know that up to a subsequence, $(\u_n)_{n}$ converges in    $\mmV^0_\beta(\Om)$ to an element $\u$ of $\nabla\mathcal{S}^+_\eps\oplus \mmV^0_{-\beta}(\Om)$. This implies in particular that $(\u_n)_{n}$ is bounded in $\mmV^0_\beta(\Om)$.  Next, for all $n\in\N,$ we define the vector field $\v_n:=\mu_n^{-1}\curl\u_n$. There holds $\div(\mu_n\v_n)=0$ in $\Om\backslash\{O\}$ and $\mu_n\v_n\cdot \nu=0$ on $\partial\Om\backslash\{O\}$. Furthermore,	by observing that 
\begin{equation}\label{EqnVol}
\curl \v_n=\om^2\eps_n\u_n+i\om\J\quad\mbox{ in }\Om\backslash\{O\},
\end{equation}
we conclude that $\v_n\in\mX_T(\mu_n)$ and that $(\curl \v_n)_n$ is bounded in $\mmV^0_\beta(\Om)$. Applying again the previous proposition, we deduce that  $(\curl \u_n)_n$ converges in $\mmV^0_\beta(\Om)$ to an element of $\mu^{-1}\nabla\mathcal{S}^+_{\mu}\oplus\mmV^0_{-\beta}(\Om)$ which is nothing but $\curl\u$ (use the convergence in the sense of distributions of $\Om\backslash\{O\}$ to establish this latter property). Thus we have
\[
\lim_{n\to+\infty}\|\u-\u_n\|_{\mmV^0_\beta(\Om)}+\|\curl\u-\curl\u_n\|_{\mmV^0_\beta(\Om)} = 0
\]
and by taking the limit in (\ref{EqnVol}) (again in the sense of distributions of $\Om\backslash\{O\}$), we get
\[
\curl(\mu^{-1}\curl \u)-\om^2\eps \u=i\om\J\quad\mbox{ in }\Om\backslash\{O\}.
\]
This implies in particular that $\u$ belongs to $\Xoutspace$. Given that $\om\in\R\backslash \mathscr{S}_{\E}$, this shows that $\u$ is the solution of (\ref{SectionEFv2}). Thus, we obtain the desired result. \\[5pt]
\textbf{Step 2.} Assume that there exists a sequence $(\u_n)_n$ of solutions to \eqref{SectionEMaxwelldissipation}, associated to some sequence  $(\delta_n)_n$ that tends to zero, such that $\|\curl \u_n\|_{\mmV^0_\beta(\Om)}\to +\infty$ as $n\to+\infty$. By considering the sequence  $(\u_n/\|\curl \u_n\|_{\mmV^0_\beta(\Om)})_n$ and using the result proved in the first step, we obtain a contradiction.
\end{proof}
	
\section{A new framework for the magnetic problem}\label{SectionH}

In this section, we provide an adapted framework to study the magnetic problem when the contrasts $\kappa_\eps$, $\kappa_\mu$ satisfy Assumption \ref{AssumptionCritique}. The procedure is similar to what has been done above for the electric field. For this reason, we do not 
give the details and simply state the main results.\\
\newline
For $\beta$ satisfying (\ref{ChoixPoids}), we define the spaces 
\[
\begin{array}{l}
\Hspace^{\out,\beta}(\curl):= \{\u\in\nabla\mathcal{S}^+_\mu\oplus\mmV^0_{-\beta}(\Om)\,|\,\curl \u\in\mmV^0_\beta(\Om)\} \\[8pt]
\HToutspace:= \{\u\in\nabla\mathcal{S}^+_\mu\oplus\mmV^0_{-\beta}(\Om)\,|\,\curl \u\in\eps\nabla \mathcal{S}^+_\eps\oplus\mmV^0_{-\beta}(\Om)\}.
\end{array}
\]
For $\u=\tilde{\u}+\nabla \mathfrak{s}_{\u,\mu}^+\in \Hspace^{\out,\beta}(\curl)$, we set
\[
\|\u \|_{\Hspace^{\out,\beta}(\curl)}:=(\|\Tilde{\u}\|^2_{\mmV^0_{-\beta}(\Om)}+\|\nabla \mathfrak{s}_{\u,\mu}^+\|^2_{\mmV^0_{\beta}(\Om)}+\|\curl\u\|^2_{\mmV^0_\beta(\Om)})^{1/2}
\]
while for $\u=\Tilde{\u}+\nabla\mathfrak{s}_{\u,\mu}^+\in\HToutspace$   such that  $\curl\u=\boldsymbol{\psi}_{\u}+\eps\nabla \mathfrak{s}_{\u,\eps}$ , we denote
\[
\|\u \|_{\HToutspace}:=(\|\Tilde{\u}\|^2_{\mmV^0_{-\beta}(\Om)}+\|\nabla \mathfrak{s}_{\u,\mu}^+\|^2_{\mmV^0_{\beta}(\Om)}+\|\boldsymbol{\psi}_{\u}\|^2_{\mmV^0_{-\beta}(\Om)}+\|\nabla \mathfrak{s}_{\u,\eps}^+\|^2_{\mmV^0_{\beta}(\Om)})^{1/2}. 
\]
We adopt the following convention:\\[4pt]
- if $\u\in\Hspace^{\out,\beta}(\curl)$, let $\tilde{\u}$, $\mathfrak{s}_{\u,\mu}$ be the elements of $\mmV^0_{-\beta}(\Om)$, $\mathcal{S}^+_\mu$ such that $\u=\tilde{\u}+\nabla \mathfrak{s}_{\u,\mu}$;\\[4pt]
- if $\u\in\HToutspace$, let $\boldsymbol{\psi}_{\u}$, $\mathfrak{s}_{\u,\eps}$ be the elements of $\mmV^0_{-\beta}(\Om)$, $\mathcal{S}^+_\eps$ such that $\curl\u=\boldsymbol{\psi}_{\u}+\eps\nabla \mathfrak{s}_{\u,\eps}$.

\subsection{Definition of the magnetic problem}
Regarding what has been done for the electric problem and using the fact that  the magnetic field $\H$ and the electric field $\E$ are linked by \eqref{EqMaxwellInitiale}, we are led to look for $\H$ in $\HToutspace$. More precisely, we consider the problem 
\begin{equation}\label{SectionEEqH}
		\begin{array}{|rll}
		\multicolumn{2}{|l}{\mbox{Find } \u\in\HToutspace\mbox{ such that  }}\\[3pt]
\curl(\eps^{-1}\curl\u) -\om^2 \mu \u&\hspace{-0.25cm}=\curl(\eps^{-1} \boldsymbol{J}) & \mbox{in }  \Om\backslash\{O\} \\[3pt]
\mu\u\cdot \nu  &\hspace{-0.25cm}=0& \mbox{on } \partial \Om\backslash\{O\}.
\end{array}
\end{equation}
This time we assume that $\J$ belongs to $\mmV^0_{-\beta}(\Om)$ where $\beta$ satisfies (\ref{ChoixPoids}). Then we consider the variational problem 
\begin{equation}
\begin{array}{|l}
\text{Find } \u \in\HToutspace \text{ such that } \\
\int_\Om \eps^{-1} \boldsymbol{\psi}_{\u}\cdot \curl\vb\,dx -\om^2\fint_\Om \mu\u\cdot\vb\,dx =\int_\Om \eps^{-1}\,\boldsymbol{J} \cdot\curl\vb\,dx,\qquad\forall\v\in \Hspace^{\out,\beta}(\curl),
		\end{array}
		\label{SectionEFv1H}
	\end{equation}
	in which the term $\fint_\Om\mu\u\cdot \vb\,dx$ is defined by
	$$ \fint_\Om \mu\u\cdot\vb\,dx:= \int_\Om \mu\Tilde{\u}\cdot\overline{\Tilde{\v}}\,dx+\int_{\Om}\mu\nabla \mathfrak{s}_{\u,\mu}\cdot\overline{\Tilde{\v}}\,dx+\int_{\Om}\mu\Tilde{\u}\cdot\nabla\overline{\mathfrak{s}_{\v,\mu}}\,dx-\int_\Om\div(\mu\nabla\mathfrak{s}_{\u,\mu})\,\overline{\mathfrak{s}_{\v,\mu}}\,dx.$$
	By working as in the proof of Proposition \ref{SectionEResEquivDis}, we obtain the
\begin{proposition}\label{SectionEResEquivDisH}
Every solution of \eqref{SectionEEqH} solves \eqref{SectionEFv1H}. Conversely, every solution of \eqref{SectionEFv1H} solves \eqref{SectionEEqH}. 
\end{proposition}	
	
\subsection{Equivalent formulation for the magnetic field}
In order to take into account the divergence free condition, we define the spaces 
\begin{equation}\label{DefSpaceChampH}
\begin{array}{rcl}
\mX_T^{\out,\beta}&\hspace{-0.2cm}:=&\hspace{-0.2cm}\{\u\in\Hspace^{\out,\beta}(\curl) \, |\, \div(\mu\u)=0\mbox{ in }\Om\backslash\{O\},\,\mu\u\cdot\nu=0\,\mbox{  on }\partial\Om\backslash\{O\}\},\\[6pt]
\XToutspace&\hspace{-0.2cm}:=&\hspace{-0.2cm}\{\u\in \HToutspace \, |\, \div(\mu\u)=0\mbox{ in }\Om\backslash\{O\},\,\mu\u\cdot\nu=0\,\mbox{  on }\partial\Om\backslash\{O\}\}.
\end{array}
\end{equation}
We endow $\mX_T^{\out,\beta}$, $\XToutspace$ respectively with the norms of $\Hspace^{\out,\beta}(\curl)$,  $\HToutspace$. By replacing $\Hspace^{\out,\beta}(\curl)$, $\HToutspace$ respectively by $\mX_T^{\out,\beta}$, $\XToutspace$ in \eqref{SectionEFv1H}, we get the problem
\begin{equation}\label{SectionEFv2H}
\begin{array}{|l}
\text{Find } \u \in   \XToutspace \text{ such that } \\[6pt]
\int_\Om \eps^{-1} \boldsymbol{\psi}_{\u}\cdot \curl\vb\,dx -\om^2\fint_\Om \mu\u\cdot\vb\,dx =\int_\Om \eps^{-1}\,\boldsymbol{J} \cdot\curl\vb\,dx,\qquad\forall\v\in  \mX_T^{\out,\beta}.
		\end{array}
\end{equation}
The analogue of Proposition \ref{SectionEResEquiv} writes
\begin{proposition}\label{SectionEResEquivH}
Assume that $\om\ne0$.\\[3pt]
$\bullet$ Every solution of \eqref{SectionEFv1H} solves \eqref{SectionEFv2H}.\\[3pt]
$\bullet$ Suppose that Assumptions \ref{AssumptionCritique}-\ref{AssumptionNoTrappedMode} hold. Then every solution of \eqref{SectionEFv2H} solves \eqref{SectionEFv1H}. 
\end{proposition}
Define the continuous operators $\mathbb{A}^\out_T,\,\mathbb{K}^\out_T:\XToutspace\to(\mX_T^{\out,\beta})^\ast$ such that for all $\u\in\XToutspace$, $\v\in\mX_T^{\out,\beta}$,
\[
\langle\mathbb{A}^\out_T\u,\v\rangle=\int_\Om\eps^{-1} \boldsymbol{\psi}_{\u}\cdot \curl\vb\,dx,\qquad\quad 	\langle\mathbb{K}^\out_T\u,\v\rangle=\fint_\Om\mu\u\cdot \vb\,dx.
\]
Finally, set $\mathscr{A}^\out_T(\om):=\mathbb{A}^\out_T-\om^2\mathbb{K}^\out_T$ so that for $\u\in\XToutspace$, $\v\in \mX_T^{\out,\beta}$, 
\[
\langle\mathscr{A}^\out_T(\om) \u,\v\rangle=\int_\Om\eps^{-1} \boldsymbol{\psi}_{\u}\cdot \curl\vb\,dx-\om^2\fint_\Om\mu\u\cdot \vb\,dx.
\]
Similarly to (\ref{ImportantIdentity}), we have the important identity, for $\u,\v\in\XToutspace$, 
\begin{equation}\label{ImportantIdentityH}
\langle\mathscr{A}^\out_T(\om) \u,\v\rangle-\overline{\langle\mathscr{A}^\out_T(\om) \v,\u\rangle}=-q_\eps(\mathfrak{s}_{\u,\eps},\mathfrak{s}_{\v,\eps})-\om^2q_\mu(\mathfrak{s}_{\u,\mu},\mathfrak{s}_{\v,\mu}).
\end{equation}

\subsection{Equivalent norms in \texorpdfstring{$\mX_T^{\out,\beta}$ and $\XToutspace$ }{TEXT}}
	
Similarly to Propositions \ref{SectionEEquivNormes1}, \ref{SectionEEquivNormes2}, we have the following results. 
\begin{proposition}\label{SectionEEquivNormes1H}
Suppose that Assumptions \ref{AssumptionCritique}-\ref{AssumptionNoTrappedMode} hold.  There is a constant $C>0$ such that
		\begin{equation}
			\|\tilde{\u}\|_{\mmV^0_{-\beta}(\Om)}+\|\nabla \mathfrak{s}_{\u,\mu}\|_{\mmV^0_{\beta}(\Om)}\leq C\,\|\curl \u\|_{\mmV^0_\beta(\Om)},\qquad\forall \u\in \mX_T^{\out,\beta}.
		\end{equation}
		Consequently, the norms $\|\cdot\|_{\Hspace^{\out,\beta}(\curl)}$ and $\|\curl \cdot\|_{\mmV^0_\beta(\Om)}$ are equivalent in $\mX_T^{\out,\beta}$.
	\end{proposition}

\begin{proposition}\label{SectionEEquivNormes2H}
		Suppose that  Assumptions \ref{AssumptionCritique}-\ref{AssumptionNoTrappedMode} hold. There is a constant $C>0$ such that
		\begin{equation}
			\|\tilde{\u}\|_{\mmV^0_{-\beta}(\Om)}+\|\nabla \mathfrak{s}_{\u,\mu}\|_{\mmV^0_{\beta}(\Om)}\leq C\, \|\boldsymbol{\psi}_{ \u}\|_{\mmV^0_{-\beta}(\Om)},\quad\forall \u\in \XToutspace\mbox{ with }\curl \u=\boldsymbol{\psi}_{\u}+\eps\nabla\mathfrak{s}_{\u,\eps}.
		\end{equation}
		Consequently, in $\XToutspace$ the map $\u\mapsto\|\boldsymbol{\psi}_{ \u}\|_{\mmV^0_{-\beta}(\Om)}$ is a norm which is equivalent to $\|\cdot\|_{\HToutspace}$.
	\end{proposition}

\subsection{Main results for the magnetic problem}

By exchanging the roles of $\eps$ and $\mu$ in the proofs of \S\ref{paragraphePrincPart}, \S\ref{paragraphCompactness}, we obtain the following theorem. 
\begin{theorem}
Suppose that Assumptions \ref{AssumptionCritique}-\ref{AssumptionNoTrappedMode} hold. Then the operator $\mathbb{A}^\out_T:\XToutspace\to(\mX_T^{\out,\beta})^\ast$ is an isomorphism while $\mathbb{K}^\out_T:\XToutspace\to(\mX_T^{\out,\beta})^\ast$ is compact.
\end{theorem}
This allows one to deduce the following theorem. 
\begin{theorem}\label{SectionEMainResH}
Suppose that  Assumptions \ref{AssumptionCritique}-\ref{AssumptionNoTrappedMode} hold.  \\[2pt]
$\bullet$ The operator   $\mathscr{A}^\out_T(\om)$ is Fredholm of index zero for all $\om\in\R$. \\[3pt]
$\bullet$ $\mathscr{A}^\out_T(\om)$ is an isomorphism for all $\om\in\R\backslash\mathscr{S}_{\H}$  where $\mathscr{S}_{\H}$ is a discrete set which can accumulate only at infinity.\\[3pt]
$\bullet$ The set $\mathscr{S}_{\H}$ is independent of $\beta$ satisfying (\ref{ChoixPoids}).\\[3pt]
$\bullet$ If $\om\in\R\backslash\mathscr{S}_{\H}$, the solution of (\ref{SectionEFv2H}) for $\J\in\mmV^0_{-\beta_\star}(\Om)$ is independent of $\beta$ satisfying (\ref{ChoixPoids}) ($\beta_\star$ is defined in (\ref{ChoixPoids})).
\end{theorem}
\begin{remark}
Coming back to the initial Maxwell's equations (\ref{EqMaxwellInitiale})- (\ref{CLMaxwell}) in the sense of distributions in $\Om\backslash\{O\}$ thanks to Propositions \ref{SectionEResEquivDis}, \ref{SectionEResEquiv}, \ref{SectionEResEquivDisH}, \ref{SectionEResEquivH}, one can prove that $\mathscr{S}_{\H}$ coincides with the set $\mathscr{S}_{\E}$ appearing in Theorem \ref{SectionEResFinal}.
\end{remark}
Finally, we state the main result for the magnetic field:
\begin{theorem}\label{SectionHResFinal}
Suppose that Assumptions \ref{AssumptionCritique}-\ref{AssumptionNoTrappedMode} hold and that $\J\in\mmV^0_{-\beta}(\Om)$ with $\beta$ satisfying (\ref{ChoixPoids}).\\[3pt]
$\bullet$ For all $\om\in\R\backslash \mathscr{S}_{\H}$ the problem \eqref{SectionEFv2H}(or equivalently  \eqref{SectionEFv1H} when $\om\ne0$) admits a unique solution.\\[3pt]
$\bullet$ When $\om\in \mathscr{S}_{\H}$  the problem  \eqref{SectionEFv2H} (or equivalently  \eqref{SectionEFv1H})  is well-posed in the Fredholm sense. Moreover, it has a kernel of finite dimension that is independent of $\beta$ satisfying (\ref{ChoixPoids}).
	\end{theorem}

\begin{remark}
One can check that the results of \S \ref{SectionELimiting} concerning the limiting absorption hold when considering the magnetic problem  instead of the electric one.
\end{remark}

\section{Classical Maxwell framework}\label{SectionNecessity}

In this section, we consider the initial system \eqref{EqMaxwellInitiale}-\eqref{CLMaxwell} when looking for fields $\E$, $\H$ in the classical space $\Lspace^2(\Om)$. This leads us to consider the problems (compare with (\ref{SectionEEqE}), (\ref{SectionEEqH}))
\[
\begin{array}{|rcll}
\multicolumn{4}{|l}{\mbox{Find }\E\in\Hspace_N(\curl)\mbox{ such that}}\\[3pt]
\curl(\mu^{-1}\curl\E) -\om^2 \eps \E&\hspace{-0.25cm}=&\hspace{-0.25cm}i\om\boldsymbol{J} & \hspace{-0.25cm}\mbox{in }  \Om \\[4pt]
						\E\times \nu&\hspace{-0.25cm}=\hspace{-0.25cm}&0  & \hspace{-0.25cm}\mbox{on } \partial \Om
					\end{array}\qquad\begin{array}{|rcll}
\multicolumn{4}{|l}{\mbox{Find }\H\in\Hspace(\curl)\mbox{ such that}}\\[3pt]
						\curl(\eps^{-1}\curl\H) -\om^2 \mu \H&\hspace{-0.25cm}=&\hspace{-0.25cm}\curl(\eps^{-1} \boldsymbol{J}) & \hspace{-0.25cm}\mbox{in }  \Om \\[4pt]
						\mu\H\cdot \nu&\hspace{-0.25cm}=&0  & \hspace{-0.25cm}\mbox{on } \partial \Om.
					\end{array}\hspace{-0.3cm}
\]
Here the volumic equations are written in the sense of distributions of $\Om$. These problems are equivalent to the following variational formulations 
\[
\begin{array}{|l}
\text{Find } \E \in   \Hspace_N(\curl) \text{ such that } \\
\int_\Om \mu^{-1} \curl\E\cdot \curl\vb -\om^2 \eps\E\cdot\vb\,dx =i\om\int_\Om \boldsymbol{J} \cdot\vb\,dx,\qquad\forall\v\in  \Hspace_N(\curl)\phantom{\qquad}
\end{array}
\]
\[
\begin{array}{|l}
\text{Find } \H \in   \Hspace(\curl) \text{ such that } \\
\int_\Om \eps^{-1} \curl\H\cdot \curl\vb -\om^2 \mu\H\cdot\vb\,dx =\int_\Om \eps^{-1}\boldsymbol{J} \cdot\curl\vb\,dx,\qquad\forall\v\in  \Hspace(\curl).
\end{array}
\]
Our goal here is to show that this is not a satisfactory framework when $\kappa_\eps$, $\kappa_\mu$ are critical. Define the continuous operators $\mathcal{A}_N(\om): \Hspace_N(\curl)\to (\Hspace_N(\curl))^\ast$ and  $\mathcal{A}_T(\om): \Hspace(\curl)\to (\Hspace(\curl))^\ast$ such that
\[
\langle\mathcal{A}_N(\om)\u,\v\rangle=\int_\Om\mu^{-1}\curl\u\cdot\curl\overline{\v}-\om^2\eps\u\cdot\overline{\v}\,dx,\qquad\forall \u,\v\in\Hspace_N(\curl)
\]
\[
\langle\mathcal{A}_T(\om)\u,\v\rangle=\int_\Om\eps^{-1}\curl\u\cdot\curl\overline{\v}-\om^2\mu\u\cdot\overline{\v}\,dx,\qquad\forall \u,\v\in\Hspace(\curl).\,
\]

\begin{proposition}
Suppose that Assumptions \ref{AssumptionCritique}-\ref{AssumptionNoTrappedMode} hold. Then the operators $\mathcal{A}_N(\om)$ and $\mathcal{A}_T(\om)$ are not of Fredholm type. 
\end{proposition}
\begin{proof}
Let us work on $\mathcal{A}_N(\om)$, the reasoning being similar for $\mathcal{A}_T(\om)$. If $\om=0$, $\nabla\mH^1_0(\Om)$ belongs to the kernel of $\mathcal{A}_N(0)$. Since this space is of infinite dimension, we infer that $\mathcal{A}_N(\om)$ is not of Fredholm type. Note that this is true also when $\kappa_\eps$, $\kappa_\mu$ are not critical, in particular when $\kappa_\eps$, $\kappa_\mu$ are positive, contrary to what follows.\\
Assume now that $\om\ne0$. We proceed by contradiction and assume that $\mathcal{A}_N(\om)$ is of Fredholm type. Since $\mathcal{A}_N(\om)$ is symmetric, necessarily it is of index zero.\\
\newline
If additionally $\mathcal{A}_N(\om)$ is injective, then it is an isomorphism and there is a constant $c>0$ such that there holds
\begin{equation}\label{AssumInjectivity}
\|\u\|_{\Hspace(\curl)} \le c\,\|\mathcal{A}_N(\om)\u\|_{(\Hspace_N(\curl))^\ast},\qquad\forall\u\in\Hspace_N(\curl).
\end{equation}
Let us show that this is impossible. To proceed, for $n\in\N^\ast$, define the field
\begin{equation}\label{DefSingu}
\u_n:=\nabla\mathfrak{s}_n\qquad\mbox{ with }\qquad \mathfrak{s}_n:=r^{1/n}\mathfrak{s}\qquad\mbox{ and }\qquad\mathfrak{s}(x):=\chi(r)s^{\eps}_{1,1,0}(x),
\end{equation}
where $\chi$ and $s^{\eps}_{1,1,0}$ appear respectively in (\ref{DefSpace}) and (\ref{defHypersingu}). Due to the multiplication by the regularization term $r^{1/n}$, the function $\mathfrak{s}_n$ belongs to $\mH^1_0(\Om)$, which ensures that $\u_n$ is in $\Hspace_N(\curl)$. Moreover, we have
\begin{equation}\label{EstimSingu}
\begin{array}{rcl}
\|\u_n\|^2_{\Hspace(\curl)}=\|\nabla\mathfrak{s}_n\|^2_{\Om} &\ge & C\,\int_{0}^{\rho/2}|\nabla(r^{-1/2+i\eta_1^\sigma+1/n})|^2\,r^2dr \\[10pt]
& = & C\,\cfrac{n}{2}\,|-1/2+1/n+i\eta_1^\sigma|^2 \bigg(\cfrac{\rho}{2}\bigg)^{2/n}\underset{n\to+\infty}{\longrightarrow}+\infty.
\end{array}
\end{equation}
Here and below $C>0$ stands for a constant which may change from one line to another but which remains independent of $n$. Now let us compute $\|\mathcal{A}_N(\om)\u_n\|_{(\Hspace_N(\curl))^\ast}$. Since $\u_n=\nabla\mathfrak{s}_n$, we simply have
\begin{equation}\label{NormDual}
\|\mathcal{A}_N(\om)\u_n\|_{(\Hspace_N(\curl))^\ast}= \sup_{\v\in\Hspace_N(\curl)\setminus\{0\}}\cfrac{\bigg|\,\om^2\dsp\int_{\Om}\eps\nabla\mathfrak{s}_n\cdot\v\,dx\bigg|}{\|\v\|_{\Hspace(\curl)}}\,.
\end{equation}
According to the item $iii)$ of Proposition \ref{AppendixHelmoltzclassique}, any $\v\in\Hspace_N(\curl)$ admits the decomposition $\v=\nabla \varphi+\curl \boldsymbol{\psi}$ with $\varphi\in \mH^1_0(\Omega)$ and $\boldsymbol{\psi}\in\mX_T(1)$. By observing that  $\curl\boldsymbol{\psi}$ belongs to $\mX_N(1)$ and by applying  Proposition \ref{AppendixWeightedRegularityX(1)}, we deduce that   $\curl\boldsymbol{\psi}\in\mmV^0_{-\gamma}(\Om)$ for some $\gamma>0$. Furthermore, we have 
\begin{equation}\label{EstimDouble}
\|\nabla\varphi\|_{\Om}+\|\curl\boldsymbol{\psi}\|_{\mmV^0_{-\gamma}(\Om)}\leq C\,\|\v\|_{\Hspace(\curl)}.									
\end{equation}
Let us write
\begin{equation}\label{DecompoNonFred1}
\int_{\Om}\eps\nabla\mathfrak{s}_n\cdot\v\,dx=\dsp\int_{\Om}\eps\nabla\mathfrak{s}_n\cdot\nabla\varphi\,dx+\dsp\int_{\Om}\eps\nabla\mathfrak{s}_n\cdot\curl\boldsymbol{\psi}\,dx.
\end{equation}
To estimate the second integral of the right hand side of (\ref{DecompoNonFred1}), we can remark that 
\[
\|\nabla\mathfrak{s}_n\|_{\mmV^0_{\gamma}(\Om)}\le \|r^{1/n}\nabla\mathfrak{s}\|_{\mmV^0_{\gamma}(\Om)}+\|r^{1/n-1}\mathfrak{s}\|_{\mmV^0_{\gamma}(\Om)} \le C \|\nabla\mathfrak{s}\|_{\mmV^1_{\gamma}(\Om)} \le C,
\]
which yields, together with (\ref{EstimDouble}), 
\begin{equation}\label{DecompoNonFred2}
\bigg|\dsp\int_{\Om}\eps\nabla\mathfrak{s}_n\cdot\curl\boldsymbol{\psi}\,dx\bigg| \le C\,\|\nabla\mathfrak{s}_n\|_{\mmV^0_{\gamma}(\Om)}\,\|\curl\boldsymbol{\psi}\|_{\mmV^0_{-\gamma}(\Om)} \le C\|\v\|_{\Hspace(\curl)}.
\end{equation}
The first integral of the right hand side of (\ref{DecompoNonFred1}) can be expanded as
\begin{equation}\label{DecompoTermGrad}
\begin{array}{rcl}
\dsp\int_{\Om}\eps\nabla\mathfrak{s}_n\cdot\nabla\varphi\,dx&\hspace{-0.3cm}=&\hspace{-0.3cm}\dsp\int_{\Om}\eps\nabla\mathfrak{s}\cdot\nabla(r^{1/n}\varphi)\,dx+\dsp\int_{\Om}\eps\mathfrak{s}\nabla(r^{1/n})\cdot\nabla\varphi\,dx-\dsp\int_{\Om}\eps\nabla\mathfrak{s}\cdot\nabla(r^{1/n})\varphi\,dx\\[12pt]
&\hspace{-0.3cm}=&\hspace{-0.3cm}-\dsp\int_{\Om}\div(\eps\nabla\mathfrak{s})\,r^{1/n}\varphi\,dx+\dsp\int_{\Om}\eps\mathfrak{s}\nabla(r^{1/n})\cdot\nabla\varphi\,dx-\dsp\int_{\Om}\eps\nabla\mathfrak{s}\cdot\nabla(r^{1/n})\varphi\,dx.
\end{array}\hspace{-0.3cm}
\end{equation}
Exploiting that $\div(\eps\nabla\mathfrak{s})=0$ for $r\le\rho/2$, we obtain
\begin{equation}\label{DecompoTermGradEstim1}
\bigg|\dsp\int_{\Om}\div(\eps\nabla\mathfrak{s})\,r^{1/n}\varphi\,dx\bigg|=\bigg|\dsp\int_{\Om\setminus\overline{B(O,\rho/2)}}\div(\eps\nabla\mathfrak{s})\,r^{1/n}\varphi\,dx\bigg|\le C\,\|\nabla\varphi\|_{\Om}.
\end{equation}
Besides, using the Cauchy-Schwarz inequality, we can write
\begin{equation}\label{DecompoTermGradEstim2}
\begin{array}{rcl}
\bigg|\dsp\int_{\Om}\eps\mathfrak{s}\nabla(r^{1/n})\cdot\nabla\varphi\,dx\bigg| &\le& \cfrac{C}{n}\,\bigg|\dsp\int_{\Om}r^{-2}|\mathfrak{s}|^2r^{2/n}\,dx\bigg|^{1/2}\|\nabla\varphi\|_{\Om}\\[10pt]
 &\le& \cfrac{C}{n}\,\bigg|\dsp\int_{\Om}r^{-3+2/n}\,dx\bigg|^{1/2}\|\nabla\varphi\|_{\Om} \le \cfrac{C}{\sqrt{n}}\,\|\nabla\varphi\|_{\Om}.
\end{array}
\end{equation}
For the third term of the right hand side of (\ref{DecompoTermGrad}), we have
\begin{equation}\label{DecompoTermGradEstim3}
\begin{array}{rcl}
\bigg|\dsp\int_{\Om}\eps\nabla\mathfrak{s}\cdot\nabla(r^{1/n})\varphi\,dx\bigg| &\le& \cfrac{C}{n}\,\bigg|\dsp\int_{\Om}r^{-2}|\mathfrak{s}|^2r^{2/n}\,dx\bigg|^{1/2}\|r^{-1}\varphi\|_{\Om}\\[10pt]
 &\le& \cfrac{C}{n}\,\bigg|\dsp\int_{\Om}r^{-3+2/n}\,dx\bigg|^{1/2}\|\nabla\varphi\|_{\Om} \le \cfrac{C}{\sqrt{n}}\,\|\nabla\varphi\|_{\Om}.
\end{array}
\end{equation}
Using (\ref{DecompoTermGradEstim1})--(\ref{DecompoTermGradEstim3}) into (\ref{DecompoTermGrad}), thanks to (\ref{EstimDouble}), we get
\begin{equation}\label{DecompoNonFred3}
\bigg|\dsp\int_{\Om}\eps\nabla\mathfrak{s}_n\cdot\nabla\varphi\,dx\bigg| \le C\,\|\nabla\varphi\|_{\Om} \le C\|\v\|_{\Hspace(\curl)}.
\end{equation}
Gathering (\ref{DecompoNonFred2}) and (\ref{DecompoNonFred3}) in (\ref{DecompoNonFred1}), from the definition (\ref{NormDual}), we conclude that 
\begin{equation}\label{BoundedSrcTerm}
\|\mathcal{A}_N(\om)\u_n\|_{(\Hspace_N(\curl))^\ast}\le C. 
\end{equation}
Together with (\ref{EstimSingu}), this contradicts (\ref{AssumInjectivity}) and proves that $\mathcal{A}_N(\om)$ cannot be an isomorphism.\\
\newline
To complete the proof, it remains to consider the case where $\mathcal{A}_N(\om)$ is assumed to have a kernel of dimension $K$. In that situation, let $\boldsymbol{\lambda}^1,\dots,\boldsymbol{\lambda}^K$, $K\ge1$, be an orthonormal basis of $\ker\mathcal{A}_N(\om)$. Set
\[
\begin{array}{rcl}
\tilde{\Hspace}_N(\curl)&:=& \{\u\in\Hspace_N(\curl)\,|\,(\u,\boldsymbol{\lambda}^k)_{\Hspace(\curl)}=0,\,k=1,\dots, K\}\\[4pt]
\boldsymbol{\mathcal{H}} &:=& \{\boldsymbol{F}\in(\Hspace_N(\curl))^\ast\,|\,\langle \boldsymbol{F},\boldsymbol{\lambda}^k\rangle=0,\,k=1,\dots, K\}
\end{array}
\] 
and define the continuous operator $\tilde{\mathcal{A}}_N(\om):=\mathcal{A}_N(\om)|_{\tilde{\Hspace}_N(\curl)}$. Then $\tilde{\mathcal{A}}_N(\om): \tilde{\Hspace}_N(\curl)\to\boldsymbol{\mathcal{H}}$ is an isomorphism so that there is a constant $c>0$ such that there holds
\begin{equation}\label{AssumInjectivityBis}
\|\u\|_{\Hspace(\curl)} \le c\,\|\tilde{\mathcal{A}}_N(\om)\u\|_{(\Hspace_N(\curl))^\ast},\qquad\forall\u\in\tilde{\Hspace}_N(\curl).
\end{equation}
Let us show that this is impossible. For $n\in\N^\ast$, take $\u_n$ as in (\ref{DefSingu}) and set 
\[
\tilde{\u}_n:=\u_n-\sum_{k=1}^K a^k_n\,\boldsymbol{\lambda}^k\qquad\mbox{ with }\qquad a^k_n:=(\u_n,\boldsymbol{\lambda}^k)_{\Hspace(\curl)}=(\nabla \mathfrak{s}_n,\boldsymbol{\lambda}^k)_{\Om}.
\]
The field $\tilde{\u}_n$ clearly belongs to $\tilde{\Hspace}_N(\curl)$. Moreover one observes that the $\boldsymbol{\lambda}^k$ are not only in $\Hspace_N(\curl)$ but also in $\mX_N(\eps)$ (see (\ref{DefSpaceMaxwell}) for the definition of this space). From the proof of Proposition \ref{PropoKernelPoids}, we infer that the $\boldsymbol{\lambda}^k$ are in $\mmV^0_{-\gamma}(\Om)$ for a certain fixed $\gamma>0$. Since $\u_n$ does not belong to $\mmV^0_{-\gamma}(\Om)$ for $n$ large enough, we deduce that $\tilde{\u}_n$ is non zero for $n$ large enough. On the other hand, since $(a^k_n)_n$ converges to $(\nabla \mathfrak{s},\boldsymbol{\lambda}^k)_{\Om}$ as $n$ tends to $+\infty$, from (\ref{EstimSingu}) we get 
\[
\lim_{n\to+\infty}\|\tilde{\u}_n\|_{\Hspace(\curl)}=+\infty.
\]
Finally, since $\tilde{\mathcal{A}}_N(\om)\tilde{\u}_n=\mathcal{A}_N(\om)\u_n$, estimate (\ref{BoundedSrcTerm}) guarantees that $(\|\tilde{\mathcal{A}}_N(\om)\tilde{\u}_n\|_{(\Hspace_N(\curl))^\ast})_n$ remains bounded when $n$ tends to $+\infty$. From this, we deduce that the estimate (\ref{AssumInjectivityBis}) cannot hold. This completes the proofs of the fact that $\mathcal{A}_N(\om)$ cannot be Fredholm of index zero.
\end{proof}

\section{Concluding remarks}
In this article, we investigated the time-harmonic Maxwell's equations in a setting involving an inclusion of negative index material whose geometry  is smooth except at a point where it has a conical tip. We proved that when the contrasts of the electromagnetic parameters take certain critical values, the  Maxwell's equations are no longer well-posed in the classical framework due to the existence of hypersingularities also known as black-hole singularities. We explained how to take into account certain of these singularities, by imposing radiation conditions at the tip, to get new frameworks in which Fredholmness is recovered. Let us mention that for the corresponding solutions, both the field and the curl of the field are singular. The selection of the outgoing behaviour is realized by applying the Mandelstam radiation principle to the two scalar operators for $\eps$ and $\mu$ which appear in the analysis. However the Mandelstam radiation principle does not provide a unique setting where the problem is well-posed. To select the one which is interesting from a physical point of view, additionally we applied the limiting absorption principle that we were able to justify when it holds for the scalar problems.\\
\newline
In the possible continuations of this work, one could consider situations where the negative index inclusion has some edges. In that case, the coefficient in front of the hypersingularities should be replaced by a function. This makes the analysis more involved. Another interesting direction would be to study the approximation of the solution in the new framework by numerical methods. For this, the difficulty lies in the fact that since the field and the curl of the field are infinitely oscillating when approaching the tip, simple mesh based methods fail to capture the phenomenon and produce spurious reflections. In \cite{BCCC16}, a method using Perfectly Matched Layers (PMLs) has been proposed to deal with the 2D scalar case. More precisely, the idea is to implement PMLs in a neighbourhood of the corner to absorb the energy leaving the domain through the black-hole singularities. For the moment it is not clear if it can be adapted to the Maxwell problem. In practice, the coefficients $\eps$ and $\mu$ depend on the frequency $\om$. Therefore it would be relevant to incorporate this aspect in the analysis as it has been done for the scalar problems in \cite{HaPa20}. For the Maxwell's equations in the singular geometry considered here, however, this seems far from being obvious because the number of black-hole waves and their features (the singular exponent in particular) depend on $\om$.

\section{Appendix}\label{SectionAppendix}
We remind the reader that by assumption, $\Om$ is simply connected and that its boundary is connected. When this hypothesis is not satisfied, the results below must be adapted.

\subsection{Classical  Helmholtz decompositions }
We start with some well-known results.
\begin{proposition}\label{AppendixHelmoltzclassique}~\\[2pt]
								i)\, According to \cite[Theorem~3.12]{AmrBerDau98}, if $\u\in \Lspace^2(\Om)$ satisfies $\div\,\u=0$ in $\Om$, then there exists a unique $\boldsymbol{\psi}\in \mX_T(1)$ such that $\u= \curl \boldsymbol{\psi}$.\\
								\newline
								ii)\, According to \cite[Theorem~3.17]{AmrBerDau98}), if $\u\in \Lspace^2(\Om)$ satisfies $\div\,\u=0$ in $\Omega$ and  $\u\cdot\nu=0$ on $\partial\Omega$, then there exists a unique $\boldsymbol{\psi}\in \mX_N(1)$ such that $\u= \curl \boldsymbol{\psi}$.\\
								\newline
								iii)\, According to \cite[Thereom~3.45]{Mon03}, every $\u\in \Lspace^2(\Om)$ can be decomposed as 
								$
								\u= \nabla p +\curl \boldsymbol{\psi},
								$
								with $p\in \mH^1_0(\Omega)$ and $\boldsymbol{\psi}\in \mX_T(1)$ which are uniquely defined. \\
								\newline
								iv)\, According to \cite[Remark~3.46]{Mon03}, every $\u\in \Lspace^2(\Om)$ can be decomposed as 
								$
								\u= \nabla p +\curl \boldsymbol{\psi},
								$
								with $p\in \mH^1_{\#}(\Om)$ and $\boldsymbol{\psi}\in \mX_N(1)$ which are uniquely defined.
							\end{proposition}

In our analysis, we needed some representation results with potentials similar to above but in weighted Sobolev spaces. To establish them, we recall   results concerning the classical Laplace operator in weighted Sobolev spaces.	For the proofs, we refer the reader to the monographs \cite{KoMR97,KoMR01,MaNP00} (see also \cite{DNBL90a,DNBL90b}).			
						
\subsection{Classical Laplace operator}

For $\gamma\in\R$, define the linear and continuous operator $\mA_D^{\gamma}:\mathring{\mV}^1_{\gamma}(\Om)\to(\mathring{\mV}^1_{-\gamma}(\Om))^{\ast}$ such that
	\[
	\langle \mA^{\gamma}_D\varphi,\varphi'\rangle=\int_{\Om}\nabla\varphi\cdot\nabla\overline{\varphi'}\,dx,\qquad \forall \varphi\in\mathring{\mV}^1_{\gamma}(\Om),\,\varphi'\in\mathring{\mV}^1_{-\gamma}(\Om).
	\]
	In the same way, for $\gamma\in(-5/2;5/2)$, define $\mA^\gamma_N:\mathcal{V}^1_{\gamma}(\Om)\to(\mathcal{V}^1_{-\gamma}(\Om))^\ast$ such that
	\[
	\langle \mA_N^{\gamma}\varphi,\varphi'\rangle=\int_{\Om}\nabla\varphi\cdot\nabla\overline{\varphi'}\,dx,\qquad \forall\varphi\in\mathcal{V}^1_{\gamma}(\Om),\,\varphi'\in\mathcal{V}^1_{-\gamma}(\Om).
	\]
Here the subscripts ${}_D$, ${}_N$	stand for Dirichlet and Neumann. Similarly to $\beta_D$, $\beta_N$ in (\ref{defBetaDN}), set
\begin{equation}\label{defGammaDN}
\begin{array}{l}
\gamma_D:=\min \{\Re e\,(\lambda-1/2) \,|\,\lambda\in\mrm{spec}(\mathscr{L}_D)\mbox{ and }\Re e\,\lambda>-1/2\},\\[6pt]
\gamma_N:=\min (\{\Re e\,(\lambda-1/2) \,|\,\lambda\in\mrm{spec}(\mathscr{L}_N)\mbox{ and }\Re e\,\lambda>-1/2\}\cup\{5/2\})
\end{array}
\end{equation}
where the symbols $\mathscr{L}_D$, $\mathscr{L}_N$ are defined as $\mathscr{L}_\eps$, $\mathscr{L}_\mu$ (see (\ref{DefLeps}), (\ref{DefLmu})) with $\eps\equiv1$, $\mu\equiv1$.	
\begin{proposition}\label{LaplacianPoids} ~\\[2pt]
$\bullet$ For $\gamma\in(-\gamma_D;\gamma_D)$, the operator $\mA_D^{\gamma}:\mathring{\mV}^1_{\gamma}(\Om)\to(\mathring{\mV}^1_{-\gamma}(\Om))^{\ast}$ is an isomorphism.\\[4pt]
$\bullet$ For $\gamma\in(-\gamma_N;\gamma_N)$, the operator $\mA^\gamma_N:\mathcal{V}^1_{\gamma}(\Om)\to(\mathcal{V}^1_{-\gamma}(\Om))^\ast$ is an isomorphism.
\end{proposition}
\begin{remark}
The values for $\gamma_D$, $\gamma_N$ depend on the geometry of $\Om$ near $O$ (see \cite[\S2.2]{KoMR01} as well as \cite[\S6.6]{KoMR97} for the Dirichlet problem and \cite[\S2.3]{KoMR01} for the Neumann problem).  When $O\in\Om$ (case 1 in (\ref{ChoixCase})), one has $\gamma_D=\gamma_N=1/2$. When $O\in\partial\Om$ and $\mathcal{K}$  coincides with a half-space (case 2 in (\ref{ChoixCase})), there holds $\gamma_D=3/2$, $\gamma_N=1/2$. Note that we have always $\gamma_D\ge1/2$, $\gamma_N\ge1/2$. This latter property is used in the definition of $\beta_\star$ in (\ref{ChoixPoids}) and ensures that $\beta_\star\le\gamma_D$, $\beta_\star\le\gamma_N$.
\end{remark}

\subsection{Decompositions in weighted Sobolev spaces}

The next result generalizes the items $iii)$, $iv)$ of Proposition \ref{AppendixHelmoltzclassique}.
\begin{proposition}\label{AppendixHelmholtzWeighted}
~\\[4pt]
$\bullet$ Fix $\gamma\in[0;\gamma_D)$. Any $\u\in\mmV^0_{-\gamma}(\Om)$ decomposes as $\u=\nabla \varphi+\curl\boldsymbol{\psi}$  with $\varphi\in\mathring{\mV}^1_{-\gamma}(\Om)$ and $\boldsymbol{\psi}\in\mX_T(1)$ satisfying $\curl\boldsymbol{\psi}\in\mmV^0_{-\gamma}(\Om)$.\\[4pt]
$\bullet$ Fix $\gamma\in[0;\gamma_N)$. Any $\u\in\mmV^0_{-\gamma}(\Om)$ satisfying $\u\cdot \nu=0$ on $\partial\Om$ decomposes as $\u=\nabla \varphi+\curl\boldsymbol{\psi}$  with $\varphi\in\mathcal{V}^1_{-\gamma}(\Om)$ and $\boldsymbol{\psi}\in\mX_N(1)$ satisfying $\curl\boldsymbol{\psi}\in\mmV^0_{-\gamma}(\Om)$.
\end{proposition}
 \begin{proof} Fix $\gamma\in[0;\gamma_D)$ and consider some $\u\in\mmV^0_{-\gamma}(\Om)$. Since $\mmV^0_{-\gamma}(\Om)\subset\Lspace^2(\Om)$,  the item $iii)$ of Proposition \ref{AppendixHelmoltzclassique} guarantees that we have $\u=\nabla \varphi+\curl \boldsymbol{\psi}$ with $\varphi\in\mH^1_0(\Om)$ and $ \boldsymbol{\psi}\in\mX_T(1)$.
By observing that $\Delta \varphi=\div\,\u\in(\mathring{\mV}^1_\gamma(\Om))^\ast $ and by using Proposition \ref{LaplacianPoids}, we conclude that $\varphi$ belongs to $\mathring{\mV}^1_{-\gamma}(\Om)$. This  ends the proof of the first item. The second one can be established similarly. 
\end{proof}

Now we generalize the results of items $i)$, $ii)$ of Proposition \ref{AppendixHelmoltzclassique}. For $\gamma\in\R$, introduce the spaces
\[ 
\begin{array}{ll}
\mZ^\gamma_N:=   \{\u\in\Lspace^2(\Om)~|~ \curl\u\in\mmV^0_\gamma(\Om),~ \div\,\u=0\mbox{ in }\Om, ~\u\times\nu= 0\mbox{ on } \partial\Om\backslash\{O\}\} \\[3pt]
\mZ^\gamma_T:=   \{\u\in\Lspace^2(\Om)~|~ \curl\u\in \mmV^0_\gamma(\Om),~ \div\,\u=0\mbox{ in }\Om, ~\u\cdot\nu= 0\mbox{ on } \partial\Om\backslash\{O\}\}.
\end{array}
\]
We endow $\mZ^\gamma_N$ and $ \mZ^\gamma_T$ with the norm $\|\cdot\|_{\mZ^\gamma}:=(\|\cdot\|^2_\Om+\|\curl \cdot\|^2_{\mmV^0_\gamma(\Om)})^{1/2}$. Set 
\begin{equation}\label{DefGammatilde}
\Tilde{\gamma}_D:=\min(\gamma_D,1),\qquad \qquad \Tilde{\gamma}_N:=\min(\gamma_N,1).
\end{equation}

\begin{proposition}\label{AppendixWeightedPotenials}  
Let $\Tilde{\gamma}_D$, $\Tilde{\gamma}_N$ be as in (\ref{DefGammatilde}).\\[4pt]
$\bullet$ Fix $\gamma\in[0;\Tilde{\gamma}_D)$. If $\u\in\mmV^0_\gamma(\Om)$  satisfies $\div\,\u=0$ in $\Om\backslash\{O\}$,  then there exists a unique $\boldsymbol{\psi}\in\mZ^\gamma_T$ such that
									$\u=\curl \boldsymbol{\psi}$.\\[4pt]
$\bullet$ Fix $\gamma\in[0;\Tilde{\gamma}_N)$. If $\u\in\mmV^0_\gamma(\Om)$   satisfies $\div\,\u=0$ in $\Om\backslash\{O\}$ and $\u\cdot\nu=0$ on $\partial\Om\backslash\{O\}$, then there exists a unique $\boldsymbol{\psi}\in\mZ^\gamma_N$ such that
									$ \u=\curl \boldsymbol{\psi}$.
\end{proposition}
\begin{remark}
This result also extends the one obtained in \cite[Appendix  B]{dhia2022maxwell} where only the case $O\in\Om$ has been considered.  Furthermore, here we give a simpler proof.
\end{remark}
\begin{proof}
Let $\u\in\mmV^0_\gamma(\Om)$ with $\gamma\in[0;\tilde{\gamma}_D)$ such that $\div\,\u=0$ in $\Om\backslash\{O\}$. Proposition \ref{AppendixWeightedRegularityX(1)} ensures that $\mX_N(1)\subset\mmV^0_{-\gamma}(\Om)$. This shows that we have $\mmV^0_\gamma(\Om)\subset(\mX_N(1))^\ast$ so that we can consider the problem
$$
\begin{array}{|l}
\mbox{ Find }\boldsymbol{\zeta}\in\mX_N(1) \mbox{ such that }\\[4pt]
\int_\Om\curl \boldsymbol{\zeta}\cdot \curl \boldsymbol{\zeta}'\,dx=\int_\Om \u\cdot  \boldsymbol{\zeta}'\,dx,\qquad\forall \boldsymbol{\zeta}'\in \mX_N(1).
\end{array}
$$
Using Proposition \ref{PropoEmbeddingCla}, one proves classically that it has a unique solution $\boldsymbol{\zeta}\in\mX_N(1)$. Additionally, by exploiting that $\div\,\u=0$ in $\Om\backslash\{O\}$, one obtains
\[
\curl(\curl\boldsymbol{\zeta})=\u\quad\mbox{ in }\Om\backslash\{O\}.
\]
Set $\boldsymbol{\psi}:=\curl\boldsymbol{\zeta}$. Then $\boldsymbol{\psi}$ is an element of $\mZ^\gamma_T$ such that $\u=\curl \boldsymbol{\psi}$. This ends the proof of the first item. The demonstration of the second one follows the same steps.
\end{proof}

\subsection{Compact embedding results}

In this paragraph, we prove results of compact embedding of Maxwell's spaces into weighted Sobolev spaces. Let us mention that this has some connections with the work \cite{buffa2003anisotropic}. In our approach, we will study the regularity of fields in a neighbourhood of $O$. To proceed, we set $\mathscr{O}:=\Om \cap B(O,\rho)$, where $\rho$ is defined in (\ref{defRho}), and introduce the spaces
$$ 
\begin{array}{rcl}
\mY_N(\mathscr{O})&:=&\{\u\in\Lspace^2(\mathscr{O})\,|\,\curl\u\in\Lspace^2(\mathscr{O}),\,\div\,\u\in\mL^2(\mathscr{O}),\,\u\times\nu=0\mbox{ on } \partial\mathscr{O}\}\\[2pt]
\mY_T(\mathscr{O})&:=&\{\u\in\Lspace^2(\mathscr{O})\,|\,\curl\u\in\Lspace^2(\mathscr{O}),\,\div\,\u\in\mL^2(\mathscr{O}),\,\u\cdot\nu=0\mbox{ on } \partial\mathscr{O}\}\\[2pt]
\Hspace_N(\mathscr{O})&:=&\mY_N(\mathscr{O})\cap\Hspace^1(\mathscr{O})\\[2pt]
\Hspace_T(\mathscr{O})&:=&\mY_T(\mathscr{O})\cap\Hspace^1(\mathscr{O}).
\end{array}
$$
When $\mathscr{O}$ is convex, in particular when $O\subset\Om$, we have that $\mY_N(\mathscr{O})=\Hspace_N(\mathscr{O})$, $\mY_T(\mathscr{O})=\Hspace_T(\mathscr{O})$ (see \cite{Cost91}). However this is not true for all conical tips $\mathscr{O}$. Generally speaking however, it is known that the quotient spaces $\mY_N(\mathscr{O})/\Hspace_N(\mathscr{O})$, $\mY_T(\mathscr{O})/\Hspace_T(\mathscr{O})$ are always of finite dimension \cite{CoDa00} and that the latter depend on the features of the conical tip, more precisely on $\varpi$. In case where this dimension is positive, the Birman-Solomyak decomposition (see \cite[Theorem 3.1]{BiSo87b}, \cite{BiSo94}, \cite[Theorem 1.1]{CoDa00}) ensures that there exist some finite dimensional spaces $S_D\subset\mH^1_0(\mathscr{O})\backslash\mH^2(\mathscr{O})$, $S_N\subset\mH^1_\#(\mathscr{O})\backslash\mH^2(\mathscr{O})$ such that
\begin{equation}\label{BirmanSolDecompo}
\begin{array}{rcl}
\mY_N(\mathscr{O}) =  \Hspace_N(\mathscr{O})\oplus\nabla S_D, \qquad\quad\mY_T(\mathscr{O}) =  \Hspace_T(\mathscr{O})\oplus\nabla S_N.
\end{array}
\end{equation}
Additionally the elements of the bases of $S_D$, $S_N$ can be chosen such that their laplacian belong to $\mathscr{C}^{\infty}(\overline{\mathscr{O}}\backslash\{O\})$.\\
\newline
This said, we begin with a result concerning the classical spaces $\mX_N(1)$, $\mX_T(1)$ which generalizes  \cite[Proposition A.2]{dhia2022maxwell} where only the case $O\in\Om$ has been considered.
\begin{proposition}\label{AppendixWeightedRegularityX(1)}
Let $\Tilde{\gamma}_D$, $\Tilde{\gamma}_N$ be as in (\ref{DefGammatilde}).\\[4pt]
$\bullet$ Fix $\gamma\in[0;\Tilde{\gamma}_D)$. The space $\mX_N(1)$ is compactly embedded  in $\mmV_{-\gamma}^0(\Om)$. Moreover, there is $C>0$ independent of $\u$ such that
\[
\|\u\|_{\mmV^0_{-\gamma}(\Om)}\leq C\,\|\curl \u\|_{\Om},\qquad \forall \u\in\mX_N(1).
\]
$\bullet$ Fix $\gamma\in[0;\Tilde{\gamma}_N)$. The space $\mX_T(1)$ is compactly embedded  in $\mmV_{-\gamma}^0(\Om)$. Moreover, there is $C>0$ independent of $\u$ such that
\[
\|\u\|_{\mmV^0_{-\gamma}(\Om)}\leq C\,\|\curl \u\|_{\Om},\qquad \forall \u\in\mX_T(1).
\]
\end{proposition}
\begin{proof}
For $\gamma=0$, this is nothing but the classical result recalled in  Proposition \ref{PropoEmbeddingCla}. Now we treat the case $\gamma\ne0$. Let $(\u_n)_{n}$ be a bounded sequence of elements of $\mX_N(1)$. For $n\in\N$, write
\begin{equation}\label{DecompoIni}
\u_n=\v_n+\w_n\qquad\mbox{ with }\quad\v_n:=\chi\u_n,\quad\w_n:=(1-\chi)\u_n,
\end{equation}
where $\chi\in\mathscr{C}^{\infty}(\R)$ is the cut-off function introduced before (\ref{DefSpace}) such that $\chi(r)=1$ for $r\le\rho/2$ and $\chi(r)=0$ for $r\ge \rho$. Let us study the behaviours of $(\v_n)$, $(\w_n)$ separately.\\[3pt]
$\star$ We start by $(\w_n)$. We have $\curl\w_n=-\nabla\chi\times\u_n+(1-\chi)\,\curl\u_n$ and $\div\,\w_n=-\nabla\chi\cdot\u_n$ in $\Om$. Therefore the sequences $(\w_n)$, $(\curl\w_n)$ and $(\div\,\w_n)$ are bounded in $\Lspace^2(\Om)$. Since additionally there holds $\w_n\times\nu=0$ on $\partial\Om$, we infer from \cite{Webe80} that we can extract a subsequence from $(\u_n)$ such that $(\w_n)$ converges in $\Lspace^2(\Om)$. Observe that since $\w_n$ vanishes in $\Om\cap B(O,\rho/2)$, this ensures that $(\w_n)$ converges in $\mmV_{\gamma}^0(\Om)$ for all $\gamma\in\R$.\\[3pt]
$\star$ Now we work with $(\v_n)$. Note that $\v_n$ is an element of $\mY_N(\mathscr{O})$. According to the Birman-Solomyak decomposition  (\ref{BirmanSolDecompo}), for all $n\in\N$, we have
\begin{equation}\label{decomLocale}
\v_n=\v^{\mrm{reg}}_n+\nabla\varphi_n
\end{equation}
with $\v^{\mrm{reg}}_n\in\Hspace_N(\mathscr{O})$ and $\varphi_n\in S_D$ such that $\nabla\varphi_n\in\mY_N(\mathscr{O})$ (simply take $\varphi_n\equiv0$ when $\mathscr{O}$ is convex). Moreover this comes with the estimate
\[
\|\v^{\mrm{reg}}_n\|_{\Hspace^1(\mathscr{O})}+\|\nabla\varphi_n\|_{\mathscr{O}}+\|\Delta\varphi_n\|_{\mathscr{O}} \le C\,\|\curl\v_n\|_{\mathscr{O}}
\]
where $C$ is independent of $n$. Using that $\Delta:\mV^2_{\eta}(\mathscr{O})\cap\mH^1_0(\mathscr{O})\to\mV^0_{\eta}(\mathscr{O})$ is an isomorphism for all $\eta\in(1-\gamma_D;1)$ (see \cite{KoMR97,KoMR01,MaNP00}) and that $\Lspace^2(\mathscr{O})\subset\mmV_{\eta}^0(\mathscr{O})$ for all $\eta\ge0$, we deduce that $(\varphi_n)$ is bounded in $\mV^1_{\gamma}(\mathscr{O})$ for all $\gamma\in(1-\tilde{\gamma}_D;1)$. From Lemma \ref{CompacitePoids}, this implies that we can extract from $(\varphi_n)$ a subsequence such that $(\nabla\varphi_n)$ converges in $\mV^0_{\gamma}(\mathscr{O})$ for all $\gamma\in(-\tilde{\gamma}_D;0)$. On the other hand, since $\mH^1(\mathscr{O})=\mV^1_{0}(\mathscr{O})$, Lemma \ref{CompacitePoids} also guarantees that we can extract from $(\v^{\mrm{reg}}_n)$ a subsequence such that $(\v^{\mrm{reg}}_n)$ converges in 
$\mmV_{\gamma-1}^0(\mathscr{O})$ for all $\gamma>0$. Combining these two results with the decomposition (\ref{decomLocale}), we deduce that we can extract from $(\v_n)$ a subsequence which converges in $\mmV_{-\gamma}^0(\mathscr{O})$ for all $\gamma\in(0;\tilde{\gamma}_D)$.\\[3pt]
Using the results for $(\w_n)$, $(\v_n)$ in (\ref{DecompoIni}) gives the proof of the first item. The demonstration of the second one follows the same steps.
\end{proof}

We continue by studying the spaces $\mX_N(\eps)$, $\mX_T(\mu)$. 
\begin{proposition}\label{AppendixX(eps)}	Suppose that Assumptions \ref{AssumptionCritique}-\ref{AssumptionNoTrappedMode} hold.\\[3pt]
$\bullet$ Fix $\gamma\in(0;\tilde{\beta}_D)$ where $\tilde{\beta}_D=\min(\beta_D,\gamma_D,1)$. The space  $\mX_N(\eps)$  is compactly embedded in $\mmV^0_{-\gamma}(\Om)$. Moreover, there is a constant $C>0$ independent of $\u$ such that
\[
\|\u\|_{\mmV^0_{-\gamma}(\Om)}\leq C\,\|\curl \u\|_{\Om},\qquad \forall \u\in\mX_N(\eps).
\]
$\bullet$ Fix $\gamma\in(0;\tilde{\beta}_N)$ where $\tilde{\beta}_N=\min(\beta_N,\gamma_N,1)$. The space  $\mX_T(\mu)$  is compactly embedded in $\mmV^0_{-\gamma}(\Om)$.\\[3pt] 
Moreover, there is a constant $C>0$ independent of $\u$ such that
\[
\|\u\|_{\mmV^0_{-\gamma}(\Om)}\leq C\,\|\curl \u\|_{\Om},\qquad \forall \u\in\mX_T(\mu).
\]
\end{proposition}
\begin{proof}
Let $\u\in\mX_N(\eps)$. By means of item $iii)$ of  Proposition \ref{AppendixHelmoltzclassique}, introduce $\varphi\in\mrm{H}^1_0(\Om)$ and $\boldsymbol{\psi}\in\mX_T(1)$ such that $ \u=\nabla\varphi+\curl\boldsymbol{\psi}$. By observing that  $\curl\boldsymbol{\psi}$ belongs to $\mX_N(1)$ and by applying    Proposition \ref{AppendixWeightedRegularityX(1)}, we deduce that   $\curl\boldsymbol{\psi}\in\mmV^0_{-\gamma}(\Om)$ for $\gamma\in(0;\tilde{\beta}_D)$. Furthermore, we have 
								\begin{equation}
									\|\curl\boldsymbol{\psi}\|_{\mmV^0_{-\gamma}(\Om)}\leq C\,\|\curl(\curl\boldsymbol{\psi})\|_{\Om}=C\,\|\curl\u\|_{\Om}.
									\label{Appendixeq1}
								\end{equation}
								By observing that $\div(\eps\nabla \varphi)=-\div(\eps\,\curl\boldsymbol{\psi})\in(\mathring{\mV}^1_\gamma(\Om))^\ast$ for all $\gamma\in(0;\tilde{\beta}_D)$, one deduces from Proposition \ref{RegulariteScalaire} and (\ref{Estimmonomorphism}) that $\varphi\in\mathring{\mV}^1_{-\gamma}(\Om)$ with the estimate
\begin{equation}\label{Appendixeq2}
\|\nabla\varphi\|_{\mmV^0_{-\gamma}(\Om)}\leq C\,\|\curl\boldsymbol{\psi}\|_{\mmV^0_{-\gamma}(\Om)}\leq C\,\|\curl\u\|_{\Om}.									
\end{equation}
This shows that $\u\in\mmV^0_{-\gamma}(\Om)$.
								By combining \eqref{Appendixeq1} and \eqref{Appendixeq2}, we obtain the desired estimate. Now, let us prove the compactness result.  Consider $(\u_n)_{n}$ a bounded sequence of elements of $\mX_N(\eps)$. For $n\in\N$, we have the decomposition $\u_n=\nabla \varphi_n+\curl \boldsymbol{\psi}_n$ with $\varphi_n\in\mH^1_0(\Om)$ and $\boldsymbol{\psi}_n\in\mX_T(1)$. Thanks to Proposition \ref{AppendixWeightedRegularityX(1)}, we infer that up to a sub-sequence, still indexed by $n$, $(\curl \boldsymbol{\psi}_n)_{n}$ converges in $\mmV^0_{-\gamma}(\Om)$ for all $\gamma\in(0;\tilde{\beta}_D)$. The estimate \eqref{Appendixeq2} implies that $(\nabla \varphi_n)_{n}$ is a Cauchy sequence in $\mmV^0_{-\gamma}(\Om)$ and thus it   converges in $\mmV^0_{-\gamma}(\Om)$. This ends the proof of the first item. The second one can be shown similarly.
\end{proof}

\begin{proposition}\label{AppendixHelweightedRegularity} Fix $\gamma\in[0;1/2)$. The spaces $\mZ^\gamma_N$ and $\mZ^\gamma_T$ are compactly embedded in $\mmV^0_{-\eta}(\Om)$ for all $\eta<1/2$. Moreover, there is a constant $C>0$ independent of $\u$ such that
$$ \|\u\|_{\mmV^0_{-\eta}(\Om)}\leq C\,\|\curl \u\|_{\mmV^0_{\gamma}(\Om)},\qquad \forall\u\in\mZ^\gamma_N\cup\mZ^\gamma_T.
$$
\end{proposition}
\begin{proof}
Let $(\u_n)_{n}$ be a bounded sequence of elements of $\mZ^\gamma_N$. For $n\in\N$, write
\begin{equation}\label{DecompoIniBis}
\u_n=\v_n+\w_n\qquad\mbox{ with }\quad\v_n:=\zeta\u_n,\quad\w_n:=(1-\zeta)\u_n,
\end{equation}
where $\zeta\in\mathscr{C}^{\infty}(\R)$ is a cut-off function such that $\zeta(r)=1$ for $r\le\rho/4$ and $\zeta(r)=0$ for $r\ge \rho/2$. Let us study the behaviours of $(\v_n)$, $(\w_n)$ separately.\\[3pt]
$\star$ We start by $(\w_n)$. As in (\ref{DecompoIni}), using that the support of $\w_n$ does not meet the origin, we find that the sequences $(\w_n)$, $(\curl\w_n)$ and $(\div\,\w_n)$ are bounded in $\Lspace^2(\Om)$. Since additionally there holds $\w_n\times\nu=0$ on $\partial\Om$, we infer from \cite{Webe80} that we can extract a subsequence from $(\u_n)$ such that $(\w_n)$ converges in $\Lspace^2(\Om)$. Moreover, exploiting again that $\w_n$ vanishes in $\Om\cap B(O,\rho/4)$, this ensures that $(\w_n)$ converges in $\mmV_{\eta}^0(\Om)$ for all $\eta\in\R$.\\[3pt]
$\star$ Now we work with $(\v_n)$. The item $i)$ of Proposition \ref{AppendixHelmoltzclassique} ensures that for all $n\in\N$, there is a unique $\boldsymbol{\psi}_n\in\mX_T(1)$ such that $\u_n=\curl\boldsymbol{\psi}_n$. Then we obtain $\curl(\curl\boldsymbol{\psi}_n)=\curl\u_n$ in $\Om\backslash\{O\}$ and so 
\begin{equation}\label{PbConicalTip}
-\boldsymbol{\Delta} \boldsymbol{\psi}_n=\curl\u_n\in\mmV^0_{\gamma}(\Om)
\end{equation}
because we have $\div\,\boldsymbol{\psi}_n=0$ in $\Om\backslash\{O\}$. The desired result will be a consequence of a regularization property for Problem (\ref{PbConicalTip}) that we describe now.\\
\newline
Let $\chi\in\mathscr{C}^{\infty}(\R)$ be the cut-off function appearing in (\ref{DefSpace}) such that $\chi(r)=1$ for $r\le\rho/2$ and $\chi(r)=0$ for $r\ge \rho$. Note that we have $\chi\equiv1$ on the support of $\zeta$. The function $\boldsymbol{\Phi}_n:=\chi\psib_n$ satisfies
\begin{equation}\label{AnnexRegZbetaProof}
\begin{array}{|rcll}
-\boldsymbol{\Delta}\boldsymbol{\Phi}_n&=&\boldsymbol{f}_n:=\chi\curl\u_n-\psib_n\Delta\chi-2\boldsymbol{\nabla}\psib_n\nabla\chi  &\mbox{ in }\mathscr{O}\\[3pt]
\boldsymbol{\Phi}_n\cdot\nu&=&0&\mbox{ on } \partial \mathscr{O}\\[3pt]
\curl(\boldsymbol{\Phi}_n)\times\nu&=&(\nabla\chi\times\psib_n)\times\nu+\chi\curl\psib_n\times\nu =0&\mbox{ on } \partial \mathscr{O}.
\end{array}
\end{equation}
To obtain the third line above, we have used in particular that on $\partial\mathscr{O}$, both $\psib_n$ and $\nabla\chi$ are tangential. Working with localization functions from the equation (\ref{PbConicalTip}),  one can show that there holds $\psib_n\in\Hspace^1(\mathscr{O}\backslash\overline{B(O,\rho/2)})$ with the estimate
\[
\|\psib_n\|_{\Hspace^1(\mathscr{O}\backslash\overline{B(O,\rho/2)})} \le C\,(\|\curl\u_n\|_{\mmV^0_\gamma(\Om)}+\|\u_n\|_{\Om)}). 
\]
Here and below, $C>0$ is a constant which may change from one line to another but which remains independent of $n$. From this, we infer that the $\boldsymbol{f}_n$ in (\ref{AnnexRegZbetaProof}) is an element of $\mmV^0_{\eta}(\mathscr{O})$ for all $\eta\in(1/2;3/2)$ (because in this case $\mmV^0_{\gamma}(\mathscr{O})$ is continuously embedded in $\mmV^0_{\eta}(\mathscr{O})$) and we have
\[
\|\boldsymbol{f}_n\|_{\mmV^0_\eta(\mathscr{O})} \le C\,(\|\curl\u_n\|_{\mmV^0_\gamma(\Om)}+\|\u_n\|_{\Om)}).
\]
On the other hand, from the Birman-Solomyak decomposition  (\ref{BirmanSolDecompo}), for all $n\in\N$, we have
\[
\boldsymbol{\Phi}_n=\boldsymbol{\Phi}_n^{\mrm{reg}}+\nabla\varphi_n
\]
with $\boldsymbol{\Phi}_n^{\mrm{reg}}\in\Hspace_T(\mathscr{O})$ and $\varphi_n\in S_N$ such that $\nabla\varphi_n\in\mY_T(\mathscr{O})$ (simply take $\varphi_n\equiv0$ when $\mathscr{O}$ is convex). This comes with the estimate
\[
\|\nabla\varphi_n\|_{\mathscr{O}}+\|\Delta\varphi_n\|_{\mathscr{O}} \le C\,(\|\div\,\boldsymbol{\Phi}_n\|_{\mathscr{O}}+\|\curl\boldsymbol{\Phi}_n\|_{\mathscr{O}}) \le C\,\|\curl\psib_n\|_{\Om}=C\,\|\u_n\|_{\Om}.
\]
Using that $S_N$ is of finite dimension and spanned by some functions whose laplacian belongs to $\mathscr{C}^{\infty}(\mathscr{O}\backslash\{O\})$, we deduce that $\boldsymbol{\Phi}_n^{\mrm{reg}}$ solves the problem
\begin{equation}\label{AnnexRegZbetaProofBis}
\begin{array}{|rcll}
-\boldsymbol{\Delta}\boldsymbol{\Phi}_n^{\mrm{reg}}&=&\tilde{\boldsymbol{f}}_n:=\boldsymbol{f}_n+\nabla(\Delta\varphi_n) &\mbox{ in }\mathscr{O}\\[3pt]
\boldsymbol{\Phi}_n^{\mrm{reg}}\cdot\nu&=&0&\mbox{ on } \partial \mathscr{O}\\[3pt]
\curl(\boldsymbol{\Phi}_n^{\mrm{reg}})\times\nu&=& 0&\mbox{ on } \partial \mathscr{O}
\end{array}
\end{equation}
where $\tilde{\boldsymbol{f}}_n$ is such that for all $\eta\in(1/2;3/2)$, 
\[
\|\tilde{\boldsymbol{f}}_n\|_{\mmV^0_\eta(\mathscr{O})} \le C\,(\|\curl\u_n\|_{\mmV^0_\gamma(\Om)}+\|\u_n\|_{\Om)}).
\]
As a consequence of Theorem \ref{RegularityMaxwell}, we infer that $\boldsymbol{\Phi}_n^{\mrm{reg}}$ belongs to $\mmV^2_\eta(\mathscr{O})$ for all $\eta\in(1/2;3/2)$ with the estimate 
\[
\|\boldsymbol{\Phi}_n^{\mrm{reg}}\|_{\mmV^2_\eta(\mathscr{O})} \le C\,(\|\curl\u_n\|_{\mmV^0_\gamma(\Om)}+\|\u_n\|_{\Om)}).
\]
Therefore $(\curl\boldsymbol{\Phi}_n^{\mrm{reg}})$ is bounded in $\mmV^1_\eta(\mathscr{O})$. Since there holds 
\[
\v_n=\zeta\u_n=\zeta\curl\psib_n=\zeta\curl\boldsymbol{\Phi}_n=\zeta\curl\boldsymbol{\Phi}_n^{\mrm{reg}}
\]
because $\chi=1$ on the support of $\zeta$, we infer that $(\v_n)$ is bounded in $\mmV^1_\eta(\Om)$. From Lemma \ref{CompacitePoids}, this implies that we can extract a subsequence such that $(\v_n)$ converges in $\mmV^0_{\eta-1+\delta}(\Om)$ for all $\delta>0$.\\
\newline
Finally, gathering the results for $(\v_n)$ and $(\w_n)$ in the decomposition (\ref{DecompoIniBis}) yields the statement for the space $\mZ^\gamma_N$. The proof for $\mZ^\gamma_T$ can be obtained by obvious modifications of the above lines.
\end{proof}

\subsection{Density result} 
For $\gamma\in\R$, set
\begin{equation}\label{defHcurlDblePoids}
\Hspace^{-\gamma,\gamma}_N(\curl):=\{\u\in\mmV^0_{-\gamma}(\Om)\,|\,\curl\u\in\mmV^0_\gamma(\Om),\,\u\times\nu=0\mbox{ on } \partial\Om\backslash\{O\}\}
\end{equation}
\begin{proposition}\label{AppednixResDensity}
The space $\Cgras^{\infty}_0(\Om\backslash\{O\})$ is dense in  $\Hspace^{-\gamma,\gamma}_N(\curl)$ for $\gamma\in[0;1/2)$.
\end{proposition}
\begin{proof}
Let $\zeta\in\mathscr{C}^{\infty}(\R)$ be the cut-off function appearing in the proof of Proposition \ref{AppendixHelweightedRegularity} such that $\zeta(r)=1$ for $r\le\rho/4$ and $\zeta(r)=0$ for $r\ge \rho/2$. Consider some $\u$ in $\Hspace^{-\gamma,\gamma}_N(\curl)$. Observe that $(1-\zeta)\u$ belongs to $\Hspace_N(\curl)$ and vanishes in a neighbourhood of $O$. Therefore, classically, it can be approximated by a sequence of elements of $\Cgras^{\infty}_0(\Om\backslash\{O\})$. Now let us focus our attention on the approximation of $\zeta\u$. Since $\zeta\u\in\mmV^0_{-\gamma}(\Om)$, according to Proposition \ref{AppendixHelweightedRegularity}, we have the decomposition
\[
\zeta\u=\nabla  \varphi+\curl \psib
\]
with $\varphi\in\mathring{\mV}^1_{-\gamma}(\Om)$ and $\boldsymbol{\psi}\in\mX_T(1)$ such that $\curl\boldsymbol{\psi}\in\mmV^0_{-\gamma}(\Om)$. Using that $\zeta\chi=\zeta$, we get
\begin{equation}\label{DecompoLocale}
\zeta\u=\chi\nabla  \varphi+\chi\curl \psib.
\end{equation}
The proof of Proposition \ref{AppendixHelweightedRegularity} ensures that $\chi\curl \psib$ belongs to 
\begin{equation}\label{defHV}
\Hspace^\gamma_N(\mathscr{O}):=\{\u\in\mmV^1_{\gamma}(\mathscr{O})\,|\,\u\times\nu=0\mbox{ on } \partial\mathscr{O}\backslash\{O\}\}.
\end{equation}
Since $\mathscr{C}^{\infty}_0(\Om\backslash\{O\})$ is dense in $\mathring{\mV}^1_{-\gamma}(\Om)$, we can find a sequence $(\varphi_n)$ of elements of $\mathscr{C}^{\infty}_0(\Om\backslash\{O\})$ which converges to $\varphi$ in $\mathring{\mV}^1_{-\gamma}(\Om)$. Then we get
\[
\begin{array}{ll}
&\|\chi\nabla\varphi-\chi\nabla\varphi_n\|_{\mmV^0_{-\gamma}(\mathscr{O})}+\|\curl(\chi\nabla\varphi)-\curl(\chi\nabla\varphi_n)\|_{\mmV^0_{\gamma}(\mathscr{O})}\\[3pt]
=&\|\chi\nabla(\varphi-\varphi_n)\|_{\mmV^0_{-\gamma}(\mathscr{O})}+\|\nabla\chi\times\nabla(\varphi-\varphi_n)\|_{\mmV^0_{\gamma}(\mathscr{O})}\underset{n\to+\infty}{\longrightarrow}0.
\end{array}
\]
On the other hand, using that $\Cgras^{\infty}_0(\mathscr{O}\backslash\{O\})$ is dense in $\Hspace^\gamma_N(\mathscr{O})$, we can find a sequence $(\boldsymbol{\phi}_n)$ of elements of $\Cgras^{\infty}_0(\mathscr{O}\backslash\{O\})$ which converges to $\chi\curl \psib$ in $\Hspace^\gamma_N(\mathscr{O})$. This allows us to write
\[
\|\chi\curl \psib-\boldsymbol{\phi}_n\|_{\mmV^0_{-\gamma}(\mathscr{O})}+\|\curl(\chi\curl \psib)-\curl\boldsymbol{\phi}_n\|_{\mmV^0_{\gamma}(\mathscr{O})} \le C\,\|\chi\curl \psib-\boldsymbol{\phi}_n\|_{\mmV^1_{\gamma}(\mathscr{O})}\underset{n\to+\infty}{\longrightarrow}0.
\]
Above we used that for $\gamma\in[0;1/2)$, there holds $\gamma-1<-\gamma$ which ensures that we have the continuous embeddings $\mmV^1_{\gamma}(\mathscr{O})\subset\mmV^0_{\gamma-1}(\mathscr{O})\subset\mmV^0_{-\gamma}(\mathscr{O})$. From (\ref{DecompoLocale}), this shows that $(\chi\nabla\varphi_n+\boldsymbol{\phi}_n)$ is a sequence of elements of $\Cgras^{\infty}_0(\mathscr{O}\backslash\{O\})$ which converges to $\zeta\u$ in $\Hspace^{-\gamma,\gamma}_N(\curl)$. This ends the proof.
\end{proof}

\subsection{Homogeneous Maxwell's operators in domains with conical tips}

In this paragraph, we present regularity results concerning the Maxwell's operators in domains with conical tips. We work in $\mathscr{O}=\Om \cap B(O,\rho)$ and consider the problems
\begin{equation}
\label{MaxwellellipticE}
\begin{array}{|rcll}
\curl(\curl \u)-\nabla \div\,\u&=& \boldsymbol{f}&\mbox{ in } \mathscr{O}\\
\u\times\nu&=&0&\mbox{ on }\partial \mathscr{O}\\
\div\,\u&=&0&\mbox{ on }\partial \mathscr{O}
\end{array}
\end{equation}
\begin{equation}
\label{MaxwellellipticH}
\begin{array}{|rcll}
\curl(\curl \u)-\nabla \div\,\u&=& \boldsymbol{f}&\mbox{ in } \mathscr{O}\\
\u\cdot \nu&=&0&\mbox{ on }\partial \mathscr{O}\\
\curl\u\times \nu&=&0&\mbox{ on }\partial \mathscr{O}.
\end{array}
\end{equation}
where $\boldsymbol{f}\in\mmV^0_{\gamma}(\mathscr{O})$ for all $\gamma>1/2$. Define the bilinear form $a(\cdot,\cdot)$ such that
$$ 
a(\u,\v)=\int_\mathscr{O}\curl\u\cdot\curl\v+\div(\u)\div(\v)\,dx
$$
and introduce the problems
\begin{equation}\label{StudyReguMaxwell} 
\begin{array}{|l}
\mbox{Find } \u\in\Hspace_N(\mathscr{O}) \mbox{ such that }\\[2pt]
 a(\u,\v)=\int_{\Om}\boldsymbol{f}\cdot\v\,dx,\quad\forall\v\in\Hspace_N(\mathscr{O}),
\end{array}\qquad\quad\begin{array}{|l}
\mbox{Find } \u\in\Hspace_T(\mathscr{O}) \mbox{ such that }\\[2pt]
 a(\u,\v)=\int_{\Om}\boldsymbol{f}\cdot\v\,dx,\quad\forall\v\in\Hspace_T(\mathscr{O}).
\end{array}
\end{equation}
Observe that if $\v\in\Hspace^1(\mathscr{O})=\mmV^1_{0}(\mathscr{O})$, then $\v\in\mmV^0_{-1}(\mathscr{O})$ which guarantees that the right hand sides in the above problems define linear forms respectively on $\Hspace_N(\mathscr{O})$, $\Hspace_T(\mathscr{O})$. On the other hand, it is known (see \cite{Cost91}) that $a(\cdot,\cdot)$ is coercive both in $\Hspace_N(\mathscr{O})$ and in $\Hspace_T(\mathscr{O})$. Therefore the problems (\ref{StudyReguMaxwell}) admit unique solutions which solve respectively (\ref{MaxwellellipticE}), (\ref{MaxwellellipticH}). The next theorem presents a weighted regularity result for these solutions.

\begin{theorem}\label{RegularityMaxwell} ~\\[2pt]
$\bullet$ The unique solution of (\ref{MaxwellellipticE}) in $\Hspace_N(\mathscr{O})$ belongs to $\mmV^2_{\gamma}(\mathscr{O})$ for all $\gamma>1/2$.\\[2pt]
$\bullet$ The unique solution of (\ref{MaxwellellipticH}) in $\Hspace_T(\mathscr{O})$ belongs to $\mmV^2_{\gamma}(\mathscr{O})$ for all $\gamma>1/2$.\\[2pt]
Moreover, for each solution, we have the estimate, for all $\gamma\in(1/2;3/2)$,
\[
\|\u\|_{\mmV^2_{\gamma}(\mathscr{O})} \le C\,\|\boldsymbol{f}\|_{\mmV^0_{\gamma}(\mathscr{O})}
\]
where $C>0$ is independent of $\boldsymbol{f}$.
\end{theorem}
\begin{proof}
We only give the strategy to prove the first item, the second one can be obtained similarly.  For $\gamma\in\R$, define the continuous operator $\mathcal{A}_N^\gamma:\Hspace^\gamma_N(\mathscr{O})\to (\Hspace^{-\gamma}_N(\mathscr{O}))^\ast$ such that 
$$\langle \mathcal{A}_N^\gamma \u,\v\rangle= a(\u,\v),\qquad \forall\u\in \Hspace^{\gamma}_N(\mathscr{O}),\ \v\in \Hspace^{-\gamma}_N(\mathscr{O}).
$$
Here $\Hspace^{\pm\gamma}_N(\mathscr{O})$ stand for the spaces appearing in (\ref{defHV}) and such that
$$
\Hspace^{\pm\gamma}_N(\mathscr{O})=\{\u\in\mmV^1_{\pm\gamma}(\mathscr{O})\,|\,\u\times\nu=0\mbox{ on } \partial\mathscr{O}\backslash\{O\}\}
$$
Observe that for $\gamma\le0$, we have $\Hspace^\gamma_N(\mathscr{O})\subset \Hspace^0_N(\mathscr{O})=\Hspace_N(\mathscr{O})$. As mentioned in \cite{CoDa00}, the system is elliptic. As a consequence, we can apply the Kondratiev approach \cite{kond67} to study it. Denote by $\mathscr{L}_E(\cdot)$  the symbol obtained when applying the Mellin transform to \eqref{MaxwellellipticE}. The ellipticity of the problem allows one to show that the spectrum of $\mathscr{L}_E(\cdot)$ is  discrete. Additionally, using that $\curl \curl \u-\nabla \div\,\u=-\boldsymbol{\Delta} \u$ and adapting \cite{KoMR98}, one can show that $\mathscr{L}_E(\cdot)$ has no eigenvalue in the energy strip (as called in \cite{KoMR01}) $\{\lambda\in\Cplx\,|\,\Re e\,\lambda\in(-1;0)\}$. Note that this result is also mentioned at the end of \cite[\S4.f]{CoDa00}. 
As a consequence, the general theory presented in \cite{KoMR97} ensures that for all $\gamma\in(-1/2;1/2)$, $\mathcal{A}_N^\gamma$ is Fredholm of index zero. Moreover we know that $\ker\,\mathcal{A}_N^\gamma$ is independent of $\gamma\in(-1/2;1/2)$. Since $\mathcal{A}_N^0$ is injective (because we have $\Hspace^0_N(\mathscr{O})=\Hspace_N(\mathscr{O})$), we infer that  $\mathcal{A}_N^\gamma$ is an isomorphism for all $\gamma\in(-1/2;1/2)$. This ensures that the solution in $\Hspace_N(\mathscr{O})$ belongs to $\mmV^1_{\gamma}(\mathscr{O})$ for all $\gamma\in(-1/2;1/2)$. Additionally, since $\partial\varpi$ is smooth, the Kondratiev theory \cite{kond67} guarantees that $\boldsymbol{\Delta}$ is an isomorphism from 
\[
\{\u\in\mmV^2_{\gamma}(\mathscr{O})\,|\,\u\times\nu=0\mbox{ on }\partial\mathscr{O}\backslash\{O\},\, \div\,\u=0\mbox{ on }\partial\mathscr{O}\backslash\{O\}\}
\]
to $\mmV^0_{\gamma}(\mathscr{O})$ for all $\gamma\in(1/2;3/2)$. As a consequence, the solution in $\Hspace_N(\mathscr{O})$ also belongs to $\mmV^2_{\gamma}(\mathscr{O})$ for all $\gamma\in(1/2;3/2)$. This ends the proof.
\end{proof}

\bibliography{bib.bib}
\bibliographystyle{plain}
\end{document}